\documentclass[12pt,twoside,reqno,centertags,draft,tikz,border=3mm]{amsart}

\usepackage{amsmath}
\usepackage{amsthm}
\usepackage{amsfonts}
\usepackage{amssymb}
\usepackage{amsrefs}
\usepackage{ascmac} 
\usepackage{mathrsfs} 
\usepackage{enumerate}
\usepackage{color}
\usepackage{empheq}
\usepackage{diagbox}

\usepackage{mathtools} 
\mathtoolsset{showonlyrefs=true} 

\usepackage{etoolbox} 

\makeatletter

\@addtoreset{equation}{section}
\makeatother

\theoremstyle{plain} 
\newtheorem{thm}{Theorem}[section]
\newtheorem{lem}[thm]{Lemma}
\newtheorem{cor}[thm]{Corollary}
\newtheorem{prop}[thm]{Proposition}

\newtheorem{corr}{Corollary\hspace{-1.5mm}}

\theoremstyle{definition} 
\newtheorem{defi}[thm]{Definition}
\newtheorem{rem}[thm]{Remark}

\newcommand{\R}{\mathbb{R}}
\newcommand{\C}{\mathbb{C}}
\newcommand{\N}{\mathbb N}

\newcommand{\ep}{\varepsilon}

\newcommand{\Lim}{\displaystyle \lim}

\let\tilde=\widetilde

\newcommand{\eqsp}[1]{{\begin{equation}\begin{aligned}#1\end{aligned}\end{
equation}}}

\title[Large time behavior for H\'enon Parabolic Equations]{Asymptotically self-similar global solutions for Hardy-H\'enon parabolic equations}

\date\today

\author[N. Chikami, M. Ikeda, K. Taniguchi and S. Tayachi]{Noboru Chikami, Masahiro Ikeda, Koichi Taniguchi and Slim Tayachi}

\address[N. Chikami]
{
Graduate School of Engineering, 
Nagoya Institute of Technology, 
Gokiso-cho, Showa-ku, Nagoya 
466-8555, Japan.}
\email{chikami.noboru@nitech.ac.jp}

\address[M. Ikeda]
{Graduate School of Information Science and Technology,
       Osaka University, 1-5, Yamadaoka, Suita-shi, Osaka 565-0871, Japan\ / Center for Advanced Intelligence Project
RIKEN, 1-4-1, Nihonbashi, Chuo-ku, Tokyo 103-0027, Japan.}
\email{ikeda@ist.osaka-u.ac.jp/masahiro.ikeda@a.riken.jp}

\address[K. Taniguchi]
{Faculty of Engineering,
Shizuoka University ,
3-5-1 Johoku, Chuo-ku, Hamamatsu, 432-8561, Japan.}
\email{taniguchi.koichi@shizuoka.ac.jp}

\address[S. Tayachi]
{Universit\'e de Tunis El Manar, Facult\'e des Sciences de Tunis, D\'epartement de Math\'ematiques, Laboratoire \'Equations aux D\'eriv\'ees Partielles LR03ES04, 2092 Tunis, Tunisia /  
NYUAD Research Institute, New York University Abu Dhabi, Saadiyat Island, Abu Dhabi, P.O. Box 129188, United Arab Emirates.
}
\email{slim.tayachi@fst.rnu.tn}


\begin{document}


\begin{abstract}
We construct asymptotically self-similar global solutions 
to the Hardy-H\'enon parabolic equation $\partial_t u - \Delta u = \pm |x|^{\gamma} |u|^{\alpha-1} u,\, \alpha>1, \gamma \in \R$, $x\in \mathbb R^d$ or $x\in \R^d\setminus\{0\}$,
for a large class of initial data belonging to weighted Lorentz spaces. 
The solution may be asymptotic to a self-similar solution of the linear heat equation 
or to a self-similar solution to the Hardy-H\'enon parabolic equation depending on the speed of decay of the initial data at infinity. 
The asymptotic results are new for the H\'enon case $\gamma>0$. We also prove the stability of the asymptotic profiles. 
Our approach applies for $\gamma> -\min(2,d)$ and unifies the cases 
$\gamma>0$, $\gamma=0$ and $-\min(2,d)<\gamma<0.$ 
For complex-valued initial data, a more intricate asymptotic behaviors can be shown; 
if either one of the real part or the imaginary part of the initial data has a faster spatial decay, 
then the solution exhibits a combined Nonlinear-``Modified Linear" asymptotic behavior, which is completely new even for the Fujita case 
$\gamma=0$. In Appendix, we show the non-existence of local positive solutions for supercritical initial data. 
\end{abstract} 

\maketitle


\section{Introduction}\label{sec:1}

\subsection{Background and setting of the problem}
We consider the Cauchy problem of the Hardy-H\'enon parabolic equation
\begin{equation}	\label{HHNLH}
	\begin{cases}
		\partial_t u - \Delta u = a |x|^{\gamma} |u|^{\alpha-1} u,
			&(t,x)\in (0,T)\times D, \\
		u(0) = \varphi \in L^{q,r}_{l}(\R^d),
	\end{cases}
\end{equation}
where $T>0,$ $d\in \mathbb{N}$, $\gamma, a \in \R,$ $\alpha>1$
$D:=\R^d$ if $\gamma\ge0$ and $D:=\R^d\setminus\{0\}$ if $\gamma<0.$ 
Here, $\partial_t:=\frac{\partial}{\partial t}$ is the time derivative, 
$\Delta:=\sum_{j=1}^d \frac{\partial^2}{\partial x_j^2}$ is the Laplace operator on $\R^d$, 
$u=u(t,x)$ is the unknown real- or complex-valued function on 
$(0,T)\times \mathbb R^d$, and $\varphi=\varphi(x)$ is a prescribed real- or complex-valued function on $\mathbb R^d$. 
Hereafter, we denote the time-space-dependent function $u$ as $u(t)$ or $u(t,x)$ depending on circumstances. 
The initial data $\varphi$ belongs to the weighted Lorentz space $L^{q,r}_l(\R^d)$. 
For the precise definition, see Definition \ref{def:WLS}. 
From hereon, we assume $a=-1$ or $a=1$ without losing generality. 

The equation \eqref{HHNLH} with $\gamma<0$ is known as the 
{Hardy parabolic equation} while \eqref{HHNLH} with $\gamma>0$ 
is known as the {H\'enon parabolic equation}. The elliptic part of \eqref{HHNLH} 
$-\Delta \phi=|x|^{\gamma}|\phi|^{\alpha-1}\phi$ 
is proposed by H\'enon as a model to study the rotating stellar systems (see \cite{H-1973}), 
and has been extensively studied in the mathematical context, 
especially in the field of nonlinear analysis and variational methods (see \cite{GhoMor2013} for example). 
The case $\gamma=0$ corresponds to a heat equation with a standard power-type nonlinearity, 
often called the {Fujita equation}, which has been extensively studied in various directions. 
Regarding the well-posedness of the Fujita equation ($\gamma=0$) in Lebesgue spaces, 
we refer to \cites{Wei1979, Wei1980}, among many, and 
\cites{QuiSou2007} for a comprehensive survey on the equation. 

We introduce the exponent $\alpha_F(\gamma)$ defined by
\begin{equation}	\label{d:Fujita.exp}
	\alpha_F(\gamma):=1+\frac{2+\gamma}{d},
\end{equation}
which is often referred as the {Fujita exponent} when $\gamma=0$. For $\gamma=0$, Fujita proves in \cite{Fuj1966} 
that \eqref{HHNLH} admits a global solution for sufficiently small initial data if $\alpha>\alpha_F(0),$ 
while it has no global-in-time solutions for any nonnegative nontrivial initial data if $\alpha<\alpha_F(0)$. 
The borderline case $\alpha = \alpha_F(0)$ is settled by \cite{Hay1973} and \cite{Sug1975}, and it falls under the nonexistence case. 
See also \cites{Wei1980, Wei1981}. 
For $\gamma\neq0$, $\alpha(\gamma)$ is also known to divide the existence and nonexistence of 
positive global solutions (See \cite{Qi1998}*{Theorem 1.6}). 
More precisely, when $\alpha>\alpha_F(\gamma),$ a small initial data generates a global solution. 
On the other hand, when $\alpha\le \alpha_F(\gamma),$ any positive solution to \eqref{HHNLH} 
blows up in finite time regardless of the size of initial data.

Concerning the local well-posedness of the problem \eqref{HHNLH}, 
Wang \cite{Wan1993} has considered the issue for non-negative initial data in the space of continuous bounded functions, 
when $d\ge3$ and $-2<\gamma<0$. 
In \cite{BenTayWei2017}, Ben Slimene, Weissler and the the fourth author prove the local well-posedness in 
$C_0(\R^d)$, the space of continuous functions vanishing at infinity, as well as in $L^q(\R^d)$ for some $q$ and $-\min\{2,d\}<\gamma<0$. 
The existence result in \cite{BenTayWei2017} is extended to the H\'enon case $\gamma>0$ by \cite{CIT2022}, 
where the weighted Lebesgue spaces play a crucial role in handling the potential $|x|^\gamma$. 
In \cites{Tay2020, CITT2022}, unconditional uniqueness is discussed in (weighted) Lorentz spaces, 
which extend classical results such as \cites{BreCaz1996, Ter2002} for $\gamma=0$. 
For earlier results concerning conditional uniqueness when $\gamma<0$, see \cites{BenTayWei2017, Ben2019, CIT2022}. 
Articles \cites{HisIsh2018, HisSie2024, HisTak2021} give optimal results regarding  
the local singularity of initial data to obtain the local existence for $\gamma\le 0.$ 

In \cite{CIT2022}, local well-posedness for \eqref{HHNLH} is proved for data 
in weighted Lebesgue spaces $L_{\tilde s}^{q}(\R^d)$ with $\tilde s \le s_c(q) = d(\frac{1}{q_c(\gamma)}-\frac1{q}).$ 
However, the auxiliary spaces used in \cite{CIT2022} do not take into account the smoothing effect of the Lebesgue exponent in the heat kernel estimates. 
Although working solely with the weight exponent $\displaystyle \sup_{0\le t\le T} t^{\frac{\tilde s -s}{2}} \|u(t)\|_{L^q_s}$ 
is sufficient for the purpose of treating the H\'enon case $\gamma>0$, 
this restricts the admissible range of the Lebesgue exponent $q$. 
Thus, the result in \cite{CIT2022} is incomplete in the sense that it does not perfectly 
reproduce various earlier results on the Fujita, Hardy and H\'enon cases. 
In this article, our first objective is to integrate previous results 
\cites{Wei1979, Wei1980, Tay2020, TayWei2022, BenTayWei2017, CIT2022} in a unified theorem, using the more general auxiliary norm 
$\displaystyle
	\sup_{0\le t\le T}  t^{\frac{d}{2}(\frac1{q} - \frac1{p}) + \frac{\tilde s -s}{2}} \|u(t)\|_{L^p_s}.$
Even though this is a marginal improvement, it is essential for the unified local well-posedness theory of \eqref{HHNLH}. 

Our second aim in this paper is to investigate asymptotic behaviors of small global solutions of \eqref{HHNLH} for all cases of $\gamma$. 
There are choices of topology to consider when studying the asymptotic behaviors of solutions to nonlinear PDEs. 
Concerning the global dynamics of large solutions for \eqref{HHNLH} with the Sobolev-critical exponent for $\gamma<0$, we refer to \cite{CIT2021}. 
In the context of the current paper, we only consider the small neighborhood of the trivial solution, which is generally known to be stable. 
This is indeed the case as shown by Theorem \ref{t:globalsol} below, and all of our global solutions dissipate in the critical topology 
$\|\cdot \|_{L^{q,r}_{l}}$ as $t$ tends to infinity, except for the self-similar solutions constructed in Theorem \ref{t:HH.self.sim}. 
On the other hand, the solutions exhibit linear decay in the subcritical topology $\|\cdot \|_{L^{p,1}_{k}}$ for $(k,p)$ satisfying 
$\frac{k}{d}+\frac{1}{p}<\frac1{q_c(\gamma)}$ (See Definition \ref{d:aux.sp} and Theorem \ref{t:globalsol} below), that is, 
all of our solutions including the self-similar solutions asymptotically tend to the trivial solution in the subcritical topology. 
Our main interest is to find the intermediate behavior of the global solutions in the subcritical topology as $t\to\infty$, that is, 
we ought to find an asymptotic function $v=v(t,x)$ such that $u-v$ decays faster than $u$ alone as $t\to\infty$. 
The natural candidate for such an asymptotic function is the self-similar solutions, and when the function $v$ above can be taken 
as either the linear or nonlinear self-similar solution, we call the solution $u$ asymptotically self-similar. 
For the Fujita and Hardy cases, \cites{BenTayWei2017, Ben2019, CazWei1998, SnoTayWei2001} have shown that suitable initial data 
generate global solutions that are asymptotically self-similar. 
More precisely, for a certain class of initial data, the asymptotic behaviors of the global solutions constructed in Theorem \ref{t:globalsol} 
are determined by the spatial decay of the data at $|x|=\infty$. 
In Theorem \ref{t:asym.behv}, we show the similar asymptotic behaviors hold for all $\gamma\in \R$ and for larger class of initial data and topology. 
If the initial data $\varphi$ decays faster than $|x|^{-\frac{2+\gamma}{\alpha-1}}$ then the solution is asymptotic to the linear self-similar solution, 
whereas if $\varphi(x)\sim |x|^{-\frac{2+\gamma}{\alpha-1}}$, then the solution is asymptotic to the nonlinear self-similar solution 
constructed in Theorem \ref{t:HH.self.sim}. The behavior is even more complex if the solution is complex-valued, as is shown in 
Theorem \ref{t:asym.behv-complex} later.  

Before we present our results, we prepare some terminologies. 
Recall that the equation \eqref{HHNLH} is invariant under the scale transformation 
\begin{equation}	\label{scale}
	u_{\lambda}(t,x) 
	:= \lambda^{\frac{2+\gamma}{\alpha-1}} u(\lambda^2 t, \lambda x), 
	\quad \lambda>0.
\end{equation}
More precisely,  if $u$ is the solution to \eqref{HHNLH}, 
then $u_{\lambda}$ defined as above also solves the equation 
with the rescaled initial data $ \lambda^{\frac{2+\gamma}{\alpha-1}} \varphi(\lambda x).$ 
Under the scaling transformation \eqref{scale}, the weighted Lorentz norm behaves as follows:
\begin{equation}\nonumber
	\|u_{\lambda}(0,\cdot)\|_{L^{q,r}_{l}} 
	= \lambda^{\frac{2+\gamma}{\alpha-1}-l-\frac{d}{q}} 
		\|u(0,\cdot)\|_{L^{q,r}_{l}}, \quad \lambda>0. 
\end{equation}
Let $q_c(\gamma)$ be defined by 
\begin{equation}	\nonumber
	\frac1{q_c(\gamma)} :=\frac{2+\gamma}{d(\alpha-1)}. 
\end{equation}
We say that the problem \eqref{HHNLH} is \emph{critical} if $q$ and $l$ satisfy
\begin{equation}	\label{d:critical}
	\frac{l}{d}+\frac{1}{q} = \frac1{q_c(\gamma)};
\end{equation}
\emph{subcritical} if $q$ and $l$ satisfy
\begin{equation}	\label{d:subcritical}
	\frac{l}{d}+\frac{1}{q} < \frac1{q_c(\gamma)};
\end{equation}
\emph{supercritical} if $q$ and $l$ satisfy
\begin{equation}	\label{d:supercritical}
	\frac{l}{d}+\frac{1}{q} > \frac1{q_c(\gamma)}. 
\end{equation}
In particular, the exponent $q_c(\gamma)$ is the scale-critical exponent 
of \eqref{HHNLH} for initial data in the Lorentz space without weights $L^{q,r}(\R^d)=L^{q,r}_0(\R^d)$, 
and it is also independent of the index $r$. 

As it is a standard practice, we study the problem \eqref{HHNLH} via the integral formulation: 
\begin{equation}	\label{intHHNLH}
	u(t,x) = e^{t\Delta} \varphi(x) + a \int_0^t e^{(t-\tau)\Delta} 
	\left\{ |\cdot|^{\gamma} |u(\tau,\cdot)|^{\alpha-1}u(\tau, \cdot)\right\}(x) \, d\tau, 
\end{equation}
where $\{e^{t\Delta}\}_{t>0}$ denotes the heat semi-group:
\begin{equation}	\nonumber
	e^{t\Delta}f=G_t\ast f
\end{equation}
with
\begin{equation}	\nonumber
	G_t(x)=(4\pi t)^{-\frac{d}2} e^{ - \frac{|x|^2}{4t} },\quad t>0,\quad x\in \R^d
\end{equation}
and $\ast$ is the convolution operator. 
The convolution is defined on $\mathcal{S}'(\R^d),$ 
the space of tempered distributions, by duality: 
$G_t\ast f=\langle f,G_t(\cdot-y)\rangle,$ 
and with the usual definition on $L^{1}(\R^d)+L^{\infty}(\R^d)$: 
\[
	\left(e^{t\Delta}f\right)(x)=\int_{\R^d}G_t(x-y)f(y)dy.
\]

Now, we are in a position to define our notion of solutions. 
\begin{defi}[Mild solution]
	\label{d:mildsol}
Let $T \in (0,\infty],$ $q, r\in [1,\infty]$ and $l\in\R$. For $\varphi \in L^{q,r}_{l} (\R^d)$, 
a function $u : [0,T] \times \R^d \to \mathbb C\ \text{or}\ \R$ is called an 
$L^{q,r}_{l} (\R^d)$-mild solution to \eqref{HHNLH} with initial data $u(0)=\varphi$ if it satisfies $u\in C([0,T]; L^{q,r}_{l} (\R^d))$ 
and the integral equation \eqref{intHHNLH} for any $t \in [0,T]$ and almost everywhere $x \in \R^d$.
The time $T$ is said to be the maximal existence time, which is denoted by $T_m$, 
if the solution cannot be extended beyond $[0,T).$ More precisely,  $T_m$ is defined as follows: 
\begin{equation}	\label{d:Tm}
	T_m = T_m (\varphi) := \sup \left\{T>0 \,;\, 
		\left.\begin{aligned}&\text{There exists a unique solution $u$ of \eqref{HHNLH}} \\
			&\text{in } C([0,T]; L^{q,r}_{l}(\R^d)) 
			\text{ with initial data $\varphi$}
		\end{aligned}\right. \right\}.
\end{equation}
We say that $u$ is global in time if $T_m = + \infty$. 
\end{defi}

In order to state our results, we introduce the following auxiliary function spaces. 
We denote by $\mathscr{D}'((0,T)\times\R^d)$ the space of distributions on $(0,T)\times\R^d$. 

\begin{defi}[Auxiliary spaces]
	\label{d:aux.sp}
Let $T \in (0,\infty],$ $k,l\in\R$ and $p, q\in [1,\infty]$. For a pair $(l,q)$ such that  
\begin{equation}	\nonumber
	0\le \frac{l}{d} + \frac{1}{q}, 
\end{equation}
let $(k,p)$ be such that 
\begin{equation}	\label{l:lin.est.pk}
	k \le l \quad\text{and}\quad 
	0 \le \frac{k}{d} + \frac{1}{p} \le \frac{l}{d} + \frac{1}{q}. 
\end{equation}
The auxiliary space $\mathcal{K}^{p}_{k}(T)$ is defined by 
	\begin{equation}	\nonumber
	\mathcal{K}^{p}_{k}(T)
		:=\left\{u\in \mathscr{D}'((0,T)\times\R^d) \,;\, 
		\|u\|_{\mathcal{K}^{p}_{k}(T')} <\infty \ \text{for all } T' \in (0,T)\right\}	
	\end{equation}
endowed with a norm
	\begin{equation}	\nonumber
		\|u\|_{\mathcal K^{p}_{k}(T)} 
		:=
		\sup_{0< t< T}
			t^{\frac{d}{2} (\frac{l}{d}+\frac1{q}-\frac{k}{d}-\frac1{p}) } 
			\|u(t)\|_{L^{p,\infty}_{k}},
	\end{equation}
where 
\[
	\|f\|_{L^{p,\infty}_k} := \| |\cdot |^k f\|_{L^{p,\infty}}.
\]
See Definition \ref{def:WLS} for the precise definition of the norm. 
We simply write $\mathcal{K}^{p}_{k}=\mathcal{K}^{p}_{k}(\infty)$ when $T=\infty.$ 
Note that when $(l,q)$ satisfies \eqref{d:critical}, we may write the auxiliary norm as 
\begin{equation}	\label{d:aux.norm}
	\|u\|_{\mathcal K^{p}_{k}(T)} =
	\sup_{0< t < T} 
		t^{\frac{d}{2} (\frac1{q_c(\gamma)}-\frac{k}{d}-\frac1{p}) } 
		\|u(t)\|_{L^{p,\infty}_{k}}. 
\end{equation}
\end{defi}

\subsection{Local well-posedness}

Our first result is dedicated to the local well-posedness of \eqref{HHNLH} in the critical weighted Lorentz space 
$L^{q,r}_{l}(\R^d)$, where $(l,q)$ satisfies \eqref{d:critical}. 
For the definition of $\mathcal L^{q,r}_s (\mathbb R^d)$ in the following, see Definition \ref{def:WLS}. 

\begin{thm}[Local well-posedness in the critical space]
	\label{t:HH.LWP}
Let $d\in\mathbb{N},$ $\gamma\in\R$ and $\alpha\in\R$ satisfy 
\begin{equation}	\label{t:HH.LWP.c.paramet}
	\gamma> -\min(2,d)
		\quad\text{and}\quad
	\alpha> \alpha_F(\gamma),
\end{equation}
where $\alpha_F(\gamma)$ is defined by \eqref{d:Fujita.exp}. Let $l\in\R$, $q\in [1,\infty]$ and $r \in [1,\infty]$ satisfy 
\eqref{d:critical},  
\begin{equation}	\label{t:HH.LWP.c.ql}
	\frac{\gamma}{\alpha-1}\le l, \quad 
	\quad r=1 \text{ if } q=1 
		\quad\text{and}\quad
	r=\infty \text{ if } q=\infty. 
\end{equation}
Then the Cauchy problem \eqref{HHNLH} is locally well-posed in $L^{q,r}_{l}(\R^d)$ for arbitrary initial data. 
More precisely, the following assertions hold. 

\begin{itemize}
\item[$(\rm{i})$] {\rm (}Existence{\rm )} 
	For any $\varphi \in L^{q,r}_{l}(\R^d)$ with $r <\infty$ 
	(Replace $L^{q,\infty}_{l}(\R^d)$ with 
	$\mathcal{L}^{q,\infty}_{l}(\R^d)$ when $r = \infty$), 
	there exist a positive number $T$ depending on $\varphi$
	and an $L^{q,r}_{l}(\R^d)$-mild solution $u$
	to \eqref{HHNLH} satisfying 
	\begin{equation}	\label{t:HH.LWP.est}
		\|u\|_{\mathcal{K}^{p}_{k}(T)} 
			\le 2 \|e^{t\Delta} \varphi \|_{\mathcal{K}^{p}_{k}(T)}, 
	\end{equation}
	where $k\in\R$ and $p\in[1,\infty]$ satisfy 
	\begin{equation}	\label{t:HH.LWP.c.pk}
		\begin{aligned}
		& \alpha < p \le \infty, \quad 
			\frac{l+\gamma}{\alpha} \le k \le l, \\
		&0\le \frac{k}{d} + \frac{1}{p}
			\quad\text{and}\quad
		\frac{2+\gamma\alpha}{d\alpha(\alpha-1)} 
			< \frac{k}{d} + \frac1{p}
			<\min\left\{ \frac{1}{q_c(\gamma)}, \, \frac{d+\gamma}{d\alpha}\right\}.
		\end{aligned}
	\end{equation}

\item[$(\rm{ii})$] {\rm (}Uniqueness{\rm )} 
	Let $T>0.$ If $u, v \in \mathcal{K}^{p}_{k}(T)$ satisfy 
	\eqref{intHHNLH} with $u(0) = v(0)=\varphi \in L^{q,r}_{l}(\R^d)$ 
	$($Replace $L^{q,\infty}_{l}(\R^d)$ with $\mathcal{L}^{q,\infty}_{l}(\R^d)$ when $r=\infty$$)$, then $u=v$ on $[0,T].$
	Moreover, the solution can be extended to the maximal interval $[0,T_m),$ 
	where $T_m$ is defined by \eqref{d:Tm}. 	

\item[$(\rm{iii})$] {\rm (}Continuous dependence on initial data{\rm )}  
	Let $u$ and $v$ be the $L^{q,r}_{l}(\R^d)$-mild solutions constructed 
	in $(\rm{i})$ with given initial data $\varphi$ and $\psi$, respectively. 
	Let $T(\varphi)$ and $T(\psi)$ be the corresponding existence times. 
	Then there exists a constant $C$ depending on $\varphi$ and $\psi$ such that 
	the solutions $u$ and $v$ satisfy 
	\begin{equation}	\nonumber
		\|u-v\|_{L^\infty(0,T;L^{q,r}_{l}) \cap \mathcal{K}^{p}_{k}(T)} 
		\le C \|\varphi-\psi\|_{L^{q,r}_{l}}
	\end{equation}
	for $T< \min\{T(\varphi), T(\psi)\}.$ 

\item[$(\rm{iv})$] {\rm (}Blow-up criterion{\rm )} 
	If $u$ is an $L^{q,r}_{l}(\R^d)$-mild solution constructed in the assertion $(\rm{i})$ and 
	$T_m<\infty,$ then $\|u\|_{\mathcal{K}^{p}_{k}(T_m)}=\infty.$

\item[$(\rm{v})$] {\rm (}Additional regularity{\rm )} 
	We have $\displaystyle \|u\|_{\mathcal{K}^{\tilde p}_{\tilde k}(T)} < \infty$ 
	for $(\tilde k, \tilde p) \in \R \times [1,\infty]$ satisfying
	\begin{equation}	\label{t:HH.LWP.reg.c1}
		\tilde k\le l, \quad
		q \le \tilde p
			\quad\text{and}\quad
		0\le \frac{\tilde k}{d} + \frac1{\tilde p}. 
	\end{equation}

\end{itemize}
\end{thm}

For subcritical initial data satisfying \eqref{d:subcritical}, we also have a local theory, 
which we present in Appendix. 

\begin{rem}
A few remarks are in order concerning Theorem \ref{t:HH.LWP}. 
\begin{itemize}

\item[(a)] 
	Thanks to Proposition \ref{p:embd} below, we have the embedding 
	\begin{equation}	\label{embd:id}
		L^{\infty}_{\frac{2+\gamma}{\alpha-1}}(\R^d)
		\hookrightarrow L^{q,r}_{l}(\R^d) 
		\hookrightarrow L^{q_c(0),\infty}_{\frac{\gamma}{\alpha-1}}(\R^d)
	\end{equation}
	for $(l,q)\in \R\times [1,\infty]$ satisfying \eqref{d:critical} and \eqref{t:HH.LWP.c.ql}. 
	Thus, the largest spaces of initial data in our results are 
	$L^{q_c(0),r}_{\frac{\gamma}{\alpha-1}}(\R^d)$ ($r<\infty$) or 
	$\mathcal{L}^{q_c(0),\infty}_{\frac{\gamma}{\alpha-1}}(\R^d)$. 
	In particular, since $L^{q_c(\gamma)}(\R^d) \hookrightarrow L^{q_c(0)}_{\frac{\gamma}{\alpha-1}}(\R^d)$, 
	Theorem \ref{t:HH.LWP} improves the initial data in \cite[Theorem 1.1]{BenTayWei2017} for the critical case 
	when $\gamma<0$. 
	Theorem \ref{t:HH.LWP} also improves \cite[Theorem 1.4]{CIT2022} for $\gamma>0$ 
	by employing the Lorentz norm. 

\item[(b)] It is known that the condition $\alpha>\alpha_F(\gamma)$ in \eqref{t:HH.LWP.c.paramet} is 
	optimal when dealing with global solutions. 
	When $\alpha=\alpha_F(\gamma)$, we have $q_c(\gamma)=1$ and the critical spaces correspond to
	$L^{q,r}_l(\R^d)$ with $\frac{l}{d} + \frac1{q} = 1$. 
	In this case, we remark that local existence can be shown for a carefully chosen initial data \cites{HisSie2024, HisTak2021}. 
	Similarly to the Fujita case in the Lebesgue space $L^1(\R^d)$, we do not know the local solvability for generic data in this so-called double-critical case. 
	
\item[(c)] The condition $\frac{\gamma}{\alpha-1}\le l$ in \eqref{t:HH.LWP.c.ql} is 
	the optimal polynomial growth rate of initial data in order to assure the local existence of classical solutions. 
	For $d\ge 3$ and $\gamma\ge 0$, it is proved in \cite{Wan1993} that for $\varphi\in C(\R^d)$ satisfying 
	$0\le \varphi(x)\le c |x|^{-\frac{\gamma}{\alpha-1}}$ on $\R^d$, 
	\eqref{HHNLH} admits a local classical solution with the same spatial decay rate. 
	Meanwhile, it is also proved in \cite{Wan1993} that this spatial decay rate is optimal in the following sense: 
	If $\varphi$ has a spatial decay slower than the above rate, namely, if there is some $\ep>0$ such that 
	$c |x|^{-\frac{\gamma}{\alpha-1}+ \ep} \le \varphi(x)$ as $|x|\to\infty$, 
	then the problem \eqref{HHNLH} does not admit a local solution with polynomial spatial growth of any order. 
	We note that the proof of \cite[Theorem 2.5]{Wan1993} also works for $\gamma<0$ and for $q<\infty$. In the case $q<\infty$, 
	we eventually have $L^\infty(\R^d)$-regularity which allows us use the same argument as \cite[Theorem 2.5]{Wan1993}.

\item[(d)] We may define Besov-type spaces based on our weighted Lorentz norm, in which context 
	the norm $\|e^{t\Delta} \varphi \|_{\mathcal{K}^{p}_{k}(T)}$ is equivalent to the Besov-type norm of negative derivative, 
	that has the same scaling as $L^{q,r}_{l}(\R^d)$ under \eqref{d:critical}, or $L^{\infty}_{\frac{2+\gamma}{\alpha-1}}(\R^d)$. 
	This has been observed by many authors; see \cite{SnoTayWei2001} for instance. 

\item[(e)] We remark that if $\gamma\ge0$, then \eqref{t:HH.LWP.reg.c1} assures that 
	$u\in L^\infty((0,T) \times \R^d)$ regardless of the choice of the auxiliary space, while 
	we may not prove $u\in L^\infty((0,T) \times \R^d)$ for $\gamma<0$ depending on the choice of 
	$l$ in \eqref{t:HH.LWP.c.pk}. More precisely, if $\gamma<0$ and $l<0$, 
	we may not show $u(t) \in L^\infty(\R^d)$ due to the conditions in Lemma \ref{l:b.strap}. 
\end{itemize}
\end{rem}

\subsection{Global existence and forward self-similar solution}

As a consequence of the local theory and its proof, we may naturally deduce the following global existence result.

\begin{thm}[Global existence]
	\label{t:globalsol}
Let $d\in\N$, $\gamma\in\R$, $\alpha\in\R$ be such that \eqref{t:HH.LWP.c.paramet}. 
Let $l\in\R$ and $q\in [1,\infty]$ satisfy 
\eqref{d:critical} and \eqref{t:HH.LWP.c.ql} and let $r \in [1,\infty]$. 
Let $\varphi\in \mathcal{S}'(\R^d)$ be such that $\|e^{t\Delta}\varphi\|_{\mathcal K^{p}_{k}}$ is sufficiently small, 
where $(k,p)$ satisfies \eqref{t:HH.LWP.c.pk}. Then the following assertions hold. 
\begin{itemize}
	\item[(i)] There exists a global-in-time solution $u\in \mathcal K^{p}_{k}$ 
		satisfying the integral equation \eqref{intHHNLH} for any $t \in (0,\infty)$ 
		and almost everywhere $x \in \R^d$.
		Moreover $u(t)\to \varphi$ in $\mathcal{S}'(\R^d)$ as $t\to 0.$ 
	\item[(ii)] 
		If furthermore $\varphi \in L^{q,r}_{l}(\R^d)$, 
		then the solution is an $L^{q,r}_{l}(\R^d)$-mild solution as defined in Definition \ref{d:mildsol}, provided $r<\infty$. 
		The solution also dissipates in $L^{q,r}_{l}(\R^d)$, i.e., 
		$\|u(t)\|_{L^{q,r}_{l}}\to0$ as $t\to\infty$. 
\end{itemize}
\end{thm}


\begin{rem}
	\label{r:globalsol}
A few remarks are in order concerning Theorem \ref{t:globalsol}. 
\begin{itemize}

\item[(a)] 
	The examples of initial data satisfying the assumptions of Theorem \ref{t:globalsol} include the following:
	\begin{itemize}
	\item[(1)] $\varphi\in L^{q,r}_{l}(\R^d)$ with $\|\varphi\|_{L^{q,r}_{l}}$ sufficiently small;
	\item[(2)] $\varphi \in L^1_{loc}(\R^d)$ satisfying 
		$|\varphi(x)| \leq c\left(1+|x|^2\right)^{- \frac{\sigma}{2} }$ 
		for all $x\in \R^d$ with $c$ sufficiently small and
		\begin{equation}	\nonumber
			  \sigma>\displaystyle\frac{2+\gamma}{\alpha-1};
		\end{equation}
	\item[(3)] $\varphi\in L^1_{loc}(\R^d)$ satisfying $|\varphi(x)|\leq c |x|^{- \frac{2+\gamma} {\alpha-1} }$ 
		for all $x\in \R^d$ with $c$ sufficiently small. 
	\end{itemize}

\item[(b)] One of the results proved in \cite{Wan1993} is the following (See also \cites{GuiNiWan1992, GuiNiWan2001}):
	Let $d\ge 3$ and $\alpha>\frac{d+\gamma}{d-2}$.  Let $U(x)$ be the singular stationary solution of \eqref{HHNLH} with $a=1$, i.e., 
	\begin{equation}\nonumber
		U(x) = \left( \frac{(2+\gamma)(d-2)}{(\alpha-1)^2} \left( \alpha - \frac{d+\gamma}{d-2}\right) \right)^{\frac1{\alpha-1}} 
				|x|^{-\frac{2+\gamma}{\alpha-1}}. 
	\end{equation}
	Then if $\varphi\in C(\R^d)$, $\varphi\ge0$ and $\|\varphi\|_{C_U} := \|\varphi U^{-1}\|_{L^\infty} < \lambda$ for any $\lambda<1$, 
	there exists a global non-negative solution $u$ to \eqref{HHNLH} with $a=1$ 
	such that $\|u(t)\|_{C_U}<\lambda$ for all $t>0$ and $u(t) \to 0$ in $L^\infty(\R^d)$ as $t\to \infty$. 
	This result in particular indicates the existence of an invariant set, in which the trivial solution is the attractor. 
	
	To compare with our results, the norm $\|\varphi\|_{C_U}$ used in \cite{Wan1993} 
	corresponds to $\|\varphi\|_{L^\infty_{\frac{2+\gamma}{\alpha-1}}}$, which is one of the scaling-critical space. 
	In Theorems \ref{t:HH.LWP} and \ref{t:globalsol} the space of initial data is generalized to $L^{q,r}_{l}(\R^d)$ satisfying \eqref{d:critical}, 
	which contain discontinuous functions. 
	On the other hand, our results require sufficient smallness of the initial data, 
	while in \cite{Wan1993} the smallness of initial data is exactly	determined by 
	the constant appearing in the singular solution $U(x)$. 

\end{itemize}
\end{rem}

We now turn to the existence of forward self-similar solutions to \eqref{HHNLH}, 
which is the foundation of our asymptotic results. 
Due to the homogeneity properties of \eqref{HHNLH}, and hence \eqref{intHHNLH}, 
it is clear that the set of solutions of \eqref{intHHNLH} is invariant under the transformation \eqref{scale}. 
Recall that a self-similar solution is a solution such that $u_\lambda(t,x)=u(t,x)$ for all $\lambda>0$. 
We have the following result. 

\begin{thm}[Existence of forward self-similar solutions]
	\label{t:HH.self.sim}
Let $d\in\mathbb{N},$ $\gamma\in\R$ and $\alpha\in\R$ satisfy \eqref{t:HH.LWP.c.paramet}. 
Let $l\in\R$ and $q\in [1,\infty]$ satisfy \eqref{d:critical} and \eqref{t:HH.LWP.c.ql}. 
Let $\Phi(x) := \omega(x) |x|^{-\frac{2+\gamma}{\alpha-1}},$ 
where $\omega\in L^\infty(\R^d)$ is homogeneous of degree $0$ and $\|\omega\|_{L^\infty}$ is sufficiently small. 
Then there exists a self-similar solution 
$u_\mathcal{S} \in \mathcal K^{p}_{k}$ 
of \eqref{HHNLH} with the initial data $\Phi$, where $(k,p)$ satisfies \eqref{t:HH.LWP.c.pk}, 
such that $u_\mathcal{S}(t) \to \Phi$ in $\mathcal{S}'(\R^d)$ as $t\to0.$ 
\end{thm}

The existence of positive self-similar solution in the case $\gamma=0$ is studied by many authors. 
We refer to classical results such as \cites{HarWei1982,Wan1993}. 
For \eqref{HHNLH} with $\gamma\neq0$, the existence of radially symmetric self-similar solutions 
in the space $L^\infty_{\frac\gamma{\alpha-1}}(\R^d)$ is established in \cites{Hir2008,Wan1993} 
for the case $d\geq 3,$ $\gamma>-2$ and $\alpha\geq \frac{d+2+2\gamma}{d-2}$, or for the case 
$d=3,\; 0<\gamma<\sqrt{3}-1,$ and $\alpha_F(\gamma)<\alpha<5+2\gamma$, 
{\color{black} or for the case $-2<\gamma<0$, $\alpha_F(\gamma)<\alpha<\frac{d+2+2\gamma}{d-2}$ and $d\geq 3$ is established.}
Here, $L^\infty_{\frac\gamma{\alpha-1}}(\R^d)$ is the space of functions $f$ 
such that $|x|^{\frac\gamma{\alpha-1}} f(x)$ is bounded on $\R^d$
(For the precise definition, see Definition \ref{def:WLS}). 
For sign-changing initial data, not necessarily radially symmetric self-similar solutions 
for the case $\gamma<0$ have been studied by \cites{Hir2008,Wan1993,CIT2022,BenTayWei2017,Chi2019}. 
The case $\gamma>0$ is studied by \cite{CIT2022}. 

\subsection{Asymptotic behaviors: real-valued case}

If $\varphi\in L^{q,r}_{l}(\R^d)$ and $r<\infty$, all solutions are dissipative 
in the neighborhood of the trivial solution by Theorem \ref{t:globalsol} (ii). 
This is not the case for the solutions with the critical spatial decay $|x|^{-\frac{2+\gamma}{\alpha-1}}$, 
or the forward self-similar solutions constructed in Theorem \ref{t:HH.self.sim}, 
since we may only expect $u\in L^{\infty}((0,\infty); L^{q,\infty}_{l}(\R^d))$. 
In the next few theorems, we investigate the intermediate behavior of the dissipative solutions 
as well as solutions with critical spatial decay. 

Firstly, we prove the asymptotic stability of global solutions, which are 
generally known to hold for small solutions. See \cites{FerVil2006, MiaYua2007}. 
We may observe that for small global solutions, 
the asymptotic stability of linear solutions is equivalent to the stability of nonlinear solutions. 

\begin{thm}[Asymptotic stability]
	\label{t:stab.asym.behv}
Let $d\in\mathbb{N},$ $\gamma\in\R$ and $\alpha\in\R$ satisfy 
\eqref{t:HH.LWP.c.paramet}. 
Let $r \in [1,\infty]$ and let $(l,q)\in\R\times[1,\infty]$ and satisfy \eqref{t:HH.LWP.c.ql}. 
Assume that $u$ and $v \in C([0,\infty); L^{q,r}_{l} (\R^d)) \cap \mathcal{K}^{p}_{k}$ 
are the global $L^{q,r}_{l}(\R^d)$-mild solutions to \eqref{intHHNLH} corresponding to $\varphi$ and $\psi$, 
respectively, such that $\max\{\|u\|_{\mathcal{K}^{p}_{k}}, \, \|v\|_{\mathcal{K}^{p}_{k}} \}\le M$
for $(k,p)\in\R\times[1,\infty]$ satisfying \eqref{t:HH.LWP.c.pk}. Suppose that 
\begin{equation}	\label{t:stab.asym.behv:eq0}
	\lim_{t\to\infty} \|e^{t\Delta} (\varphi -\psi)\|_{L^{q,r}_l} = 0
\end{equation}
holds. Then 
\begin{equation}	\label{t:stab.asym.behv:eq1}
	\lim_{t\to\infty} \|u(t) -v(t)\|_{L^{q,r}_l} = 0, 
\end{equation}
provided that $M$ is sufficiently small. Moreover, 
\begin{equation}	\label{t:stab.asym.behv:eq2}
	\lim_{t\to\infty} t^{\frac{d}{2} (\frac1{q_c(\gamma)}-\frac{k}{d}-\frac1{p}) } 
	\|u(t) -v(t)\|_{L^{p,\infty}_{k}} = 0. 
\end{equation}
Conversely, \eqref{t:stab.asym.behv:eq1} for any $M>0$ 
implies \eqref{t:stab.asym.behv:eq0}, provided that $q>1$. 

In particular, under the assumptions of Theorem \ref{t:globalsol}, \eqref{t:stab.asym.behv:eq0} and 
\eqref{t:stab.asym.behv:eq1} are equivalent for $L^{q,r}_{l}(\R^d)$-mild solutions, provided that $M$ is sufficiently small and $q>1$. 
\end{thm}

In the following result we establish that a certain class of well-behaved initial data give rise to asymptotically self-similar global solutions. 
In particular, we see that the spatial decay at $|x|=\infty$ determines 
whether the solution is asymptotic to nonlinear self-similar solutions or linear self-similar solutions. 
The study in this direction as far as we know has been initiated by \cites{CazWei1998, SnoTayWei2001} 
for the case $\gamma=0$ and extended in \cite{BenTayWei2017} to the case $\gamma<0$. 
Here, we propose a unified approach for all $\gamma$. 
Moreover, we generalize the condition for the initial data compared to the previous works. 

\begin{thm}[Asymptotically self-similar global solutions]
	\label{t:asym.behv}
Let $d\in\N$, $\alpha\in\R$ and $\gamma\in\R$ satisfy \eqref{t:HH.LWP.c.paramet}. 
Let $l\in\R$, $q\in [1,\infty]$  satisfy \eqref{d:critical} and \eqref{t:HH.LWP.c.ql}. 
Let $\sigma\in\R$ be such that 
\begin{equation}	\nonumber
	{2+\gamma\over \alpha-1} \leq \sigma<d, 
\end{equation}
and let 
\begin{equation}	\label{t:asym.behv:c2.phi}
	\Phi_{\sigma}(x)=\omega(x)|x|^{-\sigma}, 
\end{equation}
where $\omega\in L^\infty(\R^d)$ is homogeneous of degree $0$, with $\|\omega\|_{L^{\infty}}$ sufficiently small.
Let $\varphi\in \mathcal{S}'(\R^d)$ be nontrivial and satisfy the assumptions of Theorem \ref{t:globalsol}. 
In addition, assume that there exist some positive constants $C$ and $\delta$ such that 
\begin{equation}	\label{t:asym.behv:c1.phi}
	\|e^{t\Delta} \varphi\|_{L^{p,\infty}_{k}}
	\le C t^{-\frac{d}2(\frac{\sigma}{d} - \frac{k}{d} - \frac1{p})} 
		\quad\text{and}\quad
	\|e^{t\Delta} (\varphi - \Phi_{\sigma} )\|_{L^{p,\infty}_{k}}
	\le C t^{-\frac{d}2(\frac{\sigma}{d} - \frac{k}{d} - \frac1{p})-\delta}
\end{equation}
for $t>0$, and let $u\in \mathcal{K}^{p}_{k}$ be the global solution of \eqref{HHNLH} 
with initial data $\varphi$ constructed in Theorem \ref{t:globalsol}, where 
$k\in\R$ and $p\in[1,\infty]$ satisfy \eqref{t:HH.LWP.c.pk}. 
Then we have the following assertions.
\begin{itemize}
	\item[(i)] $($Nonlinear behavior$)$ 
	If $\sigma={2+\gamma\over \alpha-1},$ then there exists a positive constant $C>0$ such that 
	\begin{equation}	\label{t:asym.behv.NL:est1}
		\|u(t)-u_{\mathcal S}(t)\|_{L^{\tilde p,\infty}_{\tilde k}}
		\leq C t^{-\frac{d}{2} 
				({1 \over q_c(\gamma)}-{\tilde k\over d} - {1\over \tilde p} )-\delta}
	\end{equation}
	and 
	\begin{equation}	\label{t:asym.behv.NL:est2}
		\|{t}^\frac{2+\gamma}{2(\alpha-1)} u(t, \sqrt t \cdot) 
	  	  	-u_{\mathcal S}(1, \cdot )\|_{L^{\tilde p,\infty}_{\tilde k}} 
		\leq C t^{ -\delta}
	\end{equation}
	for all $t>0$ and $(\tilde k, \tilde p) \in\R\times[1,\infty]$ satisfying 
	\begin{equation}	\label{t:asym.behv:c.kp}
		\tilde k\le k, \quad
		p \le \tilde p
			\quad\text{and}\quad
		0\le \frac{\tilde k}{d} + \frac1{\tilde p}. 
	\end{equation}
	where $u_{\mathcal S}$ is the global self-similar solution of
	\eqref{HHNLH} with initial data $\omega(x) |x|^{-\frac{2+\gamma}{{\alpha-1}}}$. 
	In particular, there exists a constant $\tilde C>0$ such that for $t$ large 
	\begin{equation}	\label{t:asym.behv.L:est1}
		{\tilde C}^{-1} t^{-\frac{d}{2} ({1 \over q_c(\gamma)}-{\tilde k\over d}-{1\over \tilde p})}
		\leq \|u(t)\|_{L^{\tilde p,\infty}_{\tilde k}}
		\leq \tilde C t^{-\frac{d}{2} ({1 \over q_c(\gamma)}-{\tilde k\over d}-{1\over \tilde p})}. 
	\end{equation}
	\item[(ii)] $($Linear behavior$)$ 
	If $\sigma>{2+\gamma\over \alpha-1}$ and $\frac{\sigma+\gamma}{d\alpha} < \frac{k}{d} + \frac1{p}$, 
	then there exists a positive constant $C>0$ such that 
	\begin{equation}	\label{t:asym.behv.lin:c1}
		\|u(t)-e^{t\Delta} \Phi_\sigma \|_{L^{\tilde p, \infty}_{\tilde k}}
		\leq C \, t^{-\frac{d}2(\frac{\sigma}{d} - \frac{ \tilde k}{d}-\frac{1}{\tilde p}) -\delta}
	\end{equation}
	and 
	\begin{equation}	\label{t:asym.behv.lin:c2}
		\|t^{\frac{\sigma}{2}} u(t, \sqrt t \cdot)-e^{\Delta} \Phi_\sigma\|_{L^{\tilde p, \infty}_{\tilde k}}
		\leq C t^{-\delta}
	\end{equation}
	for all $t>0$ and all $(\tilde k, \tilde p)\in\R\times[1,\infty]$ satisfying \eqref{t:asym.behv:c.kp}, 
	where $C$ is a positive constant. 
	In particular, there exist a constant $\tilde C>0$ such that for $t$ large
	\[
		{\tilde C}^{-1} t^{-\frac{d}{2} (\frac{\sigma}{d}-\frac{\tilde k}{d}-\frac1{\tilde p}) } 
		\leq \|u(t)\|_{L^{\tilde p,\infty}_{\tilde k}}
		\leq \tilde C t^{-\frac{d}{2} (\frac{\sigma}{d}-\frac{\tilde k}{d}-\frac1{\tilde p}) }. 
	\]
\end{itemize}
\end{thm}

\begin{rem}
A few remarks are in order concerning Theorem \ref{t:asym.behv}. 
\begin{itemize}

\item[(a)] Compared to the previous results \cites{BenTayWei2017, CazWei1998, SnoTayWei2001}, 
	the topology for which the asymptotic behaviors hold is generalized even for the case $\gamma=0$. 
	The self-similar solutions are asymptotically stable with different weighted norms. 

\item[(b)] Let $\varphi$ be a non-trivial function such that
	\begin{equation}	\nonumber
		|\varphi(x)|\leq {c\over (1+|x|^2)^{\frac{\sigma}{2}}} \quad (x\in \R^d)
	\end{equation}
	for $c > 0$ sufficiently small, and
	\begin{equation}	\nonumber
		\varphi(x)=\omega(x)|x|^{-\sigma} \quad (|x|\geq A)
	\end{equation}
	for some constant $A>0$ and some $\omega\in L^\infty(\R^d)$, 
	homogeneous of degree $0$, with $\|\omega\|_{L^{\infty}}$ sufficiently small.
	Then the above $\varphi$ satisfies \eqref{t:asym.behv:c1.phi}. 
	This initial data has been treated in past studies such as \cites{BenTayWei2017, SnoTayWei2001}. 

\item[(c)] The condition \eqref{t:asym.behv:c1.phi} is also more general than 
	the above initial data $\varphi$, which is used in \cites{BenTayWei2017, CazWei1998, SnoTayWei2001}. 
	In fact, \eqref{t:asym.behv:c1.phi} permits any supercritical perturbation of the initial data, 
	as we no longer specify the explicit spatial decay of the initial data.  
	A similar idea can be found in \cite[Theorem 3]{SnoTayWei2001}.
	More specifically, any initial data of the form 
	$\varphi = \omega |\cdot|^{-\sigma} + \tilde\varphi$ with $\tilde\varphi \in L^{\tilde q, \tilde r}_{\tilde l}(\R^d)$, 
	where $\frac{d}{\tilde l} + \frac1{\tilde q} > \frac1{q_c(\gamma)}$, satisfies \eqref{t:asym.behv:c1.phi}. 
	The concrete examples of such $\tilde\varphi$ include:
	\begin{itemize}
		\item $\tilde \varphi (x) = c (1+|x|)^{-\tilde \sigma}$ ($\sigma < \tilde \sigma$);
		\item $\tilde \varphi (x) = e^{-|x|^s}$ ($s>0$). 
	\end{itemize}

\item[(d)] The asymptotic behaviors \eqref{t:asym.behv.NL:est1} and 
	\eqref{t:asym.behv.L:est1} hold in $L^{p,1}_{k}(\R^d)$, but for simplicity, we do not prove this fact in this paper. 
	In particular, they hold in the Lebesgue space $L^p(\R^d)$ as well. 
	It is due to the fact that the solution may be upgraded to any Lorentz spaces. 
	See Lemma \ref{l:b.strap} below. 

\end{itemize}
\end{rem}


The following result is a particular case of the previous theorem.

\begin{corr}
Let $d\in\N$, $\alpha\in\R$ and $\gamma\in\R$ satisfy \eqref{t:HH.LWP.c.paramet}. 
Let $\sigma\in\R$ be such that 
\begin{equation}	\nonumber
	{2+\gamma\over \alpha-1} \leq \sigma<d, 
\end{equation}
and let 
\begin{equation}	\nonumber
	\Phi_{\sigma}(x)=\omega(x)|x|^{-\sigma}, 
\end{equation}
where $\omega\in L^\infty(\R^d)$ is homogeneous of degree $0$, with $\|\omega\|_{L^{\infty}}$ sufficiently small, $\omega\not\equiv 0$.
Let 
$\varphi \in C(\R^d)$ satisfying 
		$|\varphi(x)| \leq c\left(1+|x|^2\right)^{- \frac{\sigma}{2} }$ 
		for all $x\in \R^d$ with $c$ sufficiently small and $$\varphi(x)=\omega(x)|x|^{-\sigma},\; |x|\geq A,$$
for some constant  $A>0.$ Then we have the following assertions.
\begin{itemize}
	\item[(i)] $($Nonlinear behavior$)$ 
	If $\sigma={2+\gamma\over \alpha-1},$ then there exists a positive constant $C>0$ such that 
	\begin{equation}	\nonumber
		\|u(t)-u_{\mathcal S}(t)\|_{L^{\infty}_{\tilde k}}
		\leq C t^{-\frac{d}{2} 
				({1 \over q_c(\gamma)}-{\tilde k\over d}  )-\delta}
	\end{equation}
	for all $t>0$ and $0\leq \tilde k< \min({2+\gamma\over \alpha-1},{d+\gamma\over \alpha}),$
	where $u_{\mathcal S}$ is the global self-similar solution of
	\eqref{HHNLH} with initial data $\omega(x) |x|^{-\frac{2+\gamma}{{\alpha-1}}}$. 
	In particular, there exists a constant $\tilde C>0$ such that for $t$ large 
	\begin{equation}	\nonumber
		{\tilde C}^{-1} t^{-\frac{d}{2} ({1 \over q_c(\gamma)}-{\tilde k\over d})}
		\leq \|u(t)\|_{L^{\infty}_{\tilde k}}
		\leq \tilde C t^{-\frac{d}{2} ({1 \over q_c(\gamma)}-{\tilde k\over d})}. 
	\end{equation}
	\item[(ii)] $($Linear behavior$)$ 
	If $\sigma>{2+\gamma\over \alpha-1}$, 
	then there exists a positive constant $C>0$ such that 
	\begin{equation}	\nonumber
		\|u(t)-e^{t\Delta} \Phi_\sigma \|_{L^{\infty}_{\tilde k}}
		\leq C \, t^{-\frac{d}2(\frac{\sigma}{d} - \frac{ \tilde k}{d}) -\delta}
	\end{equation}
	for all $t>0$ and all $0\leq \tilde k< \min(\sigma,{d+\gamma\over \alpha}),$ 
	where $C$ is a positive constant. 
	In particular, there exist a constant $\tilde C>0$ such that for $t$ large
	\[
		{\tilde C}^{-1} t^{-\frac{d}{2} (\frac{\sigma}{d}-\frac{\tilde k}{d}) } 
		\leq \|u(t)\|_{L^{\infty}_{\tilde k}}
		\leq \tilde C t^{-\frac{d}{2} (\frac{\sigma}{d}-\frac{\tilde k}{d}) }. 
	\]
\end{itemize}
\end{corr}

In fact, the above corollary follows from Theorem \ref{t:asym.behv} by taking $\tilde p=p=\infty,$ 
$0\leq \tilde k\leq k=l< \min({2+\gamma\over \alpha-1},{d+\gamma\over \alpha})$ 
with $l$ sufficiently close to $\min({2+\gamma\over \alpha-1},{d+\gamma\over \alpha})$ for (i), 
and $\tilde p=p=\infty,$ $0\leq \tilde k\leq k=l<\min(\sigma,{d+\gamma\over \alpha})$ with $l$ sufficiently close to $\min(\sigma,{d+\gamma\over \alpha})$ for (ii).

\subsection{Asymptotic behaviors: complex-valued case}

All the results so far hold for real or complex-valued initial data. 
However, more intricate asymptotic behaviors emerge for complex-valued initial data. 
We consider \eqref{HHNLH}, with $a=1$ or $a=-1$, 
and $u(0) = \varphi\in L^{q,r}_{l}(\R^d)$ a complex-valued function. We write 
\[
	\varphi=\varphi_1+i\varphi_2,
\]
where $i^2=-1,$ $\varphi_1$ is the real part of $\varphi$ and $\varphi_2$ is the imaginary part of $\varphi.$ 
Let $u$ be the solution of \eqref{HHNLH}. Writing 
\[
	u=u_1+i u_2,
\]
we deduce that the pair $(u_1,u_2)$ is a solution of the real-valued parabolic system
\begin{equation}	\label{HHNLH-sys}
	\begin{cases}
		\partial_t u_1 - \Delta u_1 = a |x|^{\gamma} (\sqrt{u_1^2+u_2^2})^{\alpha-1} u_1, \\
		\partial_t u_2 - \Delta u_2 = a |x|^{\gamma} (\sqrt{u_1^2+u_2^2})^{\alpha-1} u_2, \\
		\quad(t,x)\in (0,T)\times D,\\	
		u_1(0) = \varphi_1\in L^{q,r}_{l}(\R^d),\; u_2(0) = \varphi_2 \in L^{q,r}_{l}(\R^d).
	\end{cases}
\end{equation}

For $j=1, 2$, let $\sigma_j$ be such that 
\begin{equation}	\label{d:sigmaj}
	{2+\gamma\over \alpha-1} \leq \sigma_j<d 
\end{equation}
and let 
\begin{equation}	\label{t:asym.behv-sys:c.Phi.sig}
	\Phi_{\sigma_j}(x)={\omega_j}(x)|x|^{-\sigma_j}, 
\end{equation}
where ${\omega_j}\in L^\infty(\R^d)$ is homogeneous of degree $0$, with $\|{\omega_j}\|_{L^{\infty}}$ sufficiently small, {$j=1,\; 2$}. 
Let $\varphi = \varphi_1 + i \varphi_2 \in \mathcal{S}'(\R^d)$ be nontrivial and satisfy the assumptions of Theorem \ref{t:globalsol} 
(Thus, $\varphi$ generates a global solution). 
In addition, assume that there exist some positive constants $C_j$ and $\delta_{0j}$ such that 
\begin{equation}	\label{t:asym.behv-sys:c1.phi}
	\|e^{t\Delta} \varphi_j\|_{L^{p,\infty}_{k}}
	\le C_j t^{-\frac{d}2(\frac{\sigma}{d} - \frac{k}{d} - \frac1{p})}
		\quad\text{and}\quad 
	\|e^{t\Delta} (\Phi_{\sigma_j} - \varphi_j)\|_{L^{p,\infty}_{k}}
	\le C_j t^{-\frac{d}2(\frac{\sigma}{d} - \frac{k}{d} - \frac1{p})-\delta_{0j}}. 
\end{equation}
We are interested in the large time behavior of the global solution. We distinguish four cases: 
\begin{align*}
	&\mbox{Case 1:}\quad {2+\gamma\over \alpha-1}=\sigma_1=\sigma_2. 
	&&\mbox{Case 2:}\quad {2+\gamma\over \alpha-1}<\sigma_1,\; \sigma_2. \\	
	&\mbox{Case 3:}\quad {2+\gamma\over \alpha-1}=\sigma_1<\sigma_2.  
	&&\mbox{Case 4:} \quad {2+\gamma\over \alpha-1}=\sigma_2<\sigma_1.
\end{align*}

Case 1 is considered in Theorem \ref{t:asym.behv} (i) for $u_1$ and $u_2,$ 
hence $u$ have nonlinear behavior. More precisely, 
there exists a constant $C>0$ such that
	\[
		\|u_j(t)-u_{j\mathcal S}(t)\|_{L^{\tilde p,\infty}_{\tilde k}}
		\leq C t^{-\frac{d}{2} 
				({1 \over q_c(\gamma)}-{\tilde k\over d} - {1\over \tilde p} )-\delta_{0j}}
	\]
for all $t>0$. Here $u_{\mathcal S}=(u_{1\mathcal S},u_{2\mathcal S})$ is the self-similar solution of 
\eqref{HHNLH-sys} with initial data 
$(\omega_1(x)|x|^{-\sigma_1}, \omega_2(x)|x|^{-\sigma_2} )=(\omega_1(x), \omega_2(x))|x|^{-{{2+\gamma\over \alpha-1}}}.$ 
Case 2 is considered in Theorem \ref{t:asym.behv} (ii) for $u_1$ and $u_2,$ 
hence $u$ have linear behavior. More precisely, there exists a constant $C>0$ such that
	\[
		\left\| 
			u_j(t)-e^{t\Delta}\left(\omega_j(\cdot)|\cdot|^{-\sigma_j})\right)
		\right\|_{L^{\tilde p,\infty}_{\tilde k}}
		\leq C t^{-\frac{d}{2} (\frac{\sigma_j}{d}-\frac{\tilde k}{d}-\frac1{\tilde p}) -\delta_{0j}}
	\]
for all $t>0.$

For Cases 3 and 4, we have new asymptotic behaviors for $u=(u_1,u_2)$. 
Consider the real-valued parabolic systems:
\begin{itemize}
\item For Case 3.
\begin{equation}	\label{HHNLH-syscase3}
	\begin{cases}
		\partial_t z_1 - \Delta z_1 = a |x|^{\gamma} |z_1|^{\alpha-1} z_1, \\
		\partial_t z_2 - \Delta z_2 = a |x|^{\gamma} |z_1|^{\alpha-1} z_2, \\
		z_1(0) =\omega_1(x)|x|^{-{2+\gamma\over \alpha-1}},\; z_2(0) =\omega_2(x)|x|^{-\sigma_2}.
	\end{cases}
\end{equation}

\item For Case 4. 
\begin{equation}	\label{HHNLH-syscase4}
	\begin{cases}
		\partial_t w_1 - \Delta w_1 = a |x|^{\gamma} |w_2|^{\alpha-1} w_1, \\
		\partial_t w_2 - \Delta w_2 = a |x|^{\gamma} |w_2|^{\alpha-1} w_2, \\
		w_1(0) =\omega_1(x)|x|^{-\sigma_1},\; w_2(0) =\omega_2(x)|x|^{-{2+\gamma\over \alpha-1}}.
	\end{cases}
\end{equation}

\end{itemize}

The existence of solutions to Systems \ref{HHNLH-syscase3} and \ref{HHNLH-syscase4} 
are ensured by Proposition \ref{p:HH.sys.self.sim} below. 
Note that we may write $z_1=u_{1\mathcal{S}}$, since it coincides with the 
global self-similar solution constructed in Theorem 
We say that $z_2$ and $w_1$ exhibit a ``modified linear" behavior since 
the equation satisfied by $z_2$ or $w_1$ contains a nonlinear term in the second member 
which is given by an integral as the Duhamel formulation of $z_1$ or $w_2.$ We have obtained the following.

\begin{thm}[Combined Nonlinear-``Modified Linear" Asymptotic Behavior]
	\label{t:asym.behv-complex0}
Let $d\in\N$, $\alpha$ and $\gamma$ satisfy \eqref{t:HH.LWP.c.paramet}. 
Let $l\in\R$, $q\in [1,\infty]$ and $r \in [1,\infty]$ satisfy \eqref{d:critical} and \eqref{t:HH.LWP.c.ql}. 
Let $\sigma_j\in\R$, $j=1,2,$ satisfy \eqref{d:sigmaj}. 
Let $\varphi = \varphi_1 + i \varphi_2 \in \mathcal{S}'(\R^d)$ be nontrivial and satisfy the assumptions of Theorem \ref{t:globalsol} 
and \eqref{t:asym.behv-sys:c1.phi}. 
Let $u\in \mathcal{K}^{p}_{k}$ be the global solution of \eqref{HHNLH} with initial data $\varphi$ 
constructed in Theorem \ref{t:globalsol}, where $k\in\R$ and $p\in[1,\infty]$ satisfy \eqref{t:HH.LWP.c.pk}. 
Let $u_{j\mathcal S}$ be the global self-similar solution of \eqref{HHNLH} with initial data $\omega_j(x) |x|^{-\frac{2+\gamma}{{\alpha-1}}}$. 
Let $(u_{1\mathcal S},z_2)$ be the self-similar solution of \eqref{HHNLH-syscase3} and 
$(w_1,u_{2\mathcal S})$ be the self-similar solution of \eqref{HHNLH-syscase4}. Then we have the following.

\begin{itemize}

	\item[(i)] $($Nonlinear-``Modified linear" behavior$)$
	If ${2+\gamma\over \alpha-1}=\sigma_1<\sigma_2,$ then there exist $\delta_j>0$ (possibly smaller than $\delta_{0j}$), $j=1, 2$, such that
	\begin{equation}\nonumber
		\|u_1(t)-u_{1\mathcal S}(t)\|_{L^{\tilde p,\infty}_{\tilde k}}
		\leq C t^{-\frac{d}{2} 
				({1 \over q_c(\gamma)}-{\tilde k\over d} - {1\over \tilde p} )-\delta_1},
	\end{equation}				
	and
	\begin{equation}\nonumber
		\left\| 
			u_2(t)-z_2(t)		\right\|_{L^{\tilde p,\infty}_{\tilde k}}
		\leq C t^{-\frac{d}{2} (\frac{\sigma_2}{d}-\frac{\tilde k}{d}-\frac1{\tilde p}) -\delta_2},
	\end{equation}	
	for sufficiently large $t>0$, where $C$ is a positive constant. 
	In particular, there exists a constant $\tilde C>0$ such that for $t$ large 
	\[
		{\tilde C}^{-1} t^{-\frac{d}{2} ({1 \over q_c(\gamma)}-{\tilde k\over d}-{1\over \tilde p})}
		\leq \|u_1(t)\|_{L^{\tilde p,\infty}_{\tilde k}}
		\leq \tilde C t^{-\frac{d}{2} ({1 \over q_c(\gamma)}-{\tilde k\over d}-{1\over \tilde p})},
	\]
	and
	\[
		{\tilde C}^{-1} t^{-\frac{d}{2} (\frac{\sigma_2}{d}-\frac{\tilde k}{d}-\frac1{\tilde p}) } 
		\leq \|u_2(t)\|_{L^{\tilde p,\infty}_{\tilde k}}
		\leq \tilde C t^{-\frac{d}{2} (\frac{\sigma_2}{d}-\frac{\tilde k}{d}-\frac1{\tilde p}) },
	\]

	\item[(ii)] $($``Modified linear"-Nonlinear behavior$)$
	If ${2+\gamma\over \alpha-1}=\sigma_2<\sigma_1$ then there exist $\delta_j>0$  (possibly smaller than $\delta_{0j}$), $j=1, 2$, such that 
	\[
		\left\| 
			u_1(t)-w_1(t)
		\right\|_{L^{\tilde p,\infty}_{\tilde k}}
		\leq C t^{-\frac{d}{2} (\frac{\sigma_1}{d}-\frac{\tilde k}{d}-\frac1{\tilde p}) -\delta_1}
	\]
	and
	\[
		\|u_2(t)-u_{2\mathcal S}(t)\|_{L^{\tilde p,\infty}_{\tilde k}}
		\leq C t^{-\frac{d}{2} 
				({1 \over q_c(\gamma)}-{\tilde k\over d} - {1\over \tilde p} )-\delta_2}
	\]
	for sufficiently large $t>0$, where $C$ is a positive constant. 
	In particular, there exists a constant $\tilde C>0$ such that for $t$ large 
	\[
		{\tilde C}^{-1} t^{-\frac{d}{2} (\frac{\sigma_1}{d}-\frac{\tilde k}{d}-\frac1{\tilde p}) } 
		\leq \|u_1(t)\|_{L^{\tilde p,\infty}_{\tilde k}}
		\leq \tilde C t^{-\frac{d}{2} (\frac{\sigma_1}{d}-\frac{\tilde k}{d}-\frac1{\tilde p}) }. 
	\]
	and
	\[
		{\tilde C}^{-1} t^{-\frac{d}{2} ({1 \over q_c(\gamma)}-{\tilde k\over d}-{1\over \tilde p})}
		\leq \|u_2(t)\|_{L^{\tilde p,\infty}_{\tilde k}}
		\leq \tilde C t^{-\frac{d}{2} ({1 \over q_c(\gamma)}-{\tilde k\over d}-{1\over \tilde p})}. 
	\]

\end{itemize}
\end{thm}

Let us now give an heuristic idea about our results. 
It is known that the dilation structure of solutions for the nonlinear heat equation and 
in particular self-similar solutions have a key role in the description of the asymptotic behavior of solutions.  
Note that our proofs do not use the scaling argument techniques.

Let $(u_1,u_2)$ be  a solution of the problem \eqref{HHNLH-sys} with initial data $(u_1(0),u_2(0))$, 
and consider the scaling for  $\lambda>0$ defined by 
\begin{equation*}
	u_{1,\lambda}(t,x)=\lambda^{\sigma_{1}}u_1(\lambda^{2}t, \lambda x),\quad
	u_{2,\lambda}(t,x)=\lambda^{\sigma_{2}}u_2(\lambda^{2}t, \lambda x),
\end{equation*}
where $ t\geq 0, x\in\R^d.$ Then $(u_{1,\lambda},u_{2,\lambda})$  satisfies the following system 
\begin{equation*}
	\left\{
	\begin{array}{ll}
	\partial_t u_{1,\lambda}
		=\Delta u_{1,\lambda}+ a |x|^{\gamma} [\lambda^{-2(\sigma_1-{2+\gamma\over \alpha-1})} u_{1,\lambda}^2
				+\lambda^{-2(\sigma_2-{2+\gamma\over \alpha-1})}u_{2,\lambda}^2]^{{\alpha-1\over 2}} u_{1,\lambda},&\\
	\partial_t u_{2,\lambda}
		=\Delta u_{2,\lambda}+a |x|^{\gamma} [\lambda^{-2(\sigma_1-{2+\gamma\over \alpha-1})}u_{1,\lambda}^2
				+\lambda^{-2(\sigma_2-{2+\gamma\over \alpha-1})}u_{2,\lambda}^2]^{{\alpha-1\over 2}} u_{2,\lambda},&
	\end{array}
	\right.
\end{equation*}
with initial data
 \begin{equation}	 \label{rescaledinitialdata}
	 u_{1,\lambda}(0,x)=\lambda^{\sigma_1} u_{1}(0,\lambda x), \quad 
	 u_{2,\lambda}(0,x)=\lambda^{\sigma_2}u_{2}(0,\lambda x).
 \end{equation}

We now observe that studying the asymptotic behavior of $\big(u_1(t,x),u_2(t,x)\big)$ as $t\to \infty$ is the
same as studying the asymptotic behavior of $(u_{1,\lambda}(t,x),u_{2,\lambda}(t,x))$ as $\lambda\to \infty$ 
with a fixed time $t>0$. Clearly, by \eqref{rescaledinitialdata}, the behavior of $( u_{1}(0), u_{2}(0))$ in space as 
$|x|\rightarrow\infty$ (or of $u_{j,\lambda}(0,x), \; j=1,\; 2,$ as $\lambda\rightarrow\infty,$ 
with fixed $x$) will also have an effect on the behavior of the solutions. 
In fact if $u_{1}(0,x)=\omega_1(x)|x|^{-\sigma_1}\mathbf{1}_{\{|x|>A_1\}},$ 
$u_{2}(0,x)= \omega_2(x)|x|^{-\sigma_2}\mathbf{1}_{\{|x|>A_2\}},$ 
$0<\sigma_j<d;\; A_j>0,\; j=1,\; 2;$ 
$\omega_j,\; j=1,\; 2,$ are homogeneous of degree $0,$ 
and $\mathbf{1}_{\{|x|>A\}}$ is the indicator function of the subset $\{|x|>A\}\subset \R^d$, then    
$$
	u_{1,\lambda}(0,x)=\omega_1(x)|x|^{-\sigma_1}\mathbf{1}_{\{|x|>{A_1\over \lambda}\}}
	\to \omega_1(x)|x|^{-\sigma_1},\; \mbox{as}\;  \lambda\rightarrow\infty,$$ 
and
$$
	u_{2,\lambda}(0,x)=\omega_2(x)|x|^{-\sigma_2}\mathbf{1}_{\{|x|>{A_2\over \lambda}\}}
	\to \omega_2(x)|x|^{-\sigma_2},\; \mbox{as}\;  \lambda\rightarrow\infty,
$$
pointwise for  $x\not= 0,$ and  in the sense of distributions.

Since we hope to have finite limits as $\lambda\rightarrow\infty$ for the coefficients in the system satisfied 
by $(u_{1,\lambda},u_{2,\lambda})$, we restrict ourselves to the case 
$\sigma_1\geq (2+\gamma)/(\alpha-1)$ and $\sigma_2\geq (2+\gamma)/(\alpha-1).$
We then obtain the four cases by letting $\lambda\to \infty.$ 

\begin{rem}
Further remarks are in order concerning Theorem \ref{t:asym.behv}. 
\begin{itemize}

\item[(a)] The estimates in Theorem \ref{t:asym.behv-complex} hold for sufficiently large $t$, 
	in contrast to the real case (Theorem \ref{t:asym.behv}). 
	We believe that the corresponding homogeneous estimates (that is, estimates for all $t>0$) 
	are possible for all range of $\alpha$, but the proofs are more delicate and complex. 

\item[(b)] The study of the asymptotic behavior of the complex-valued parabolic equation has been carried out by some, but not many. 
	To our knowledge, all are with the nonlinearity $u^{\alpha}$ and $\gamma=0$.
	The asymptotic behaviors for a complex-valued initial data with the nonlinearity $u^{\alpha}$ have been shown by \cite{ChoOtsTay2015} for the 
	case $\alpha=2$ and $\gamma=0$. The target equations similar to \eqref{HHNLH-syscase3} and 
	\eqref{HHNLH-syscase4} have also been identified by \cite{ChoOtsTay2015}. 
	In \cite{ChoMajTay2018} the asymptotic results for the nonlinearity $u^{\alpha}$ with $\gamma=0$ are extended to integers $\alpha>1$. 
	In \cite{Ots2021}, the extension to a generalized heat kernel is considered. 
	The full picture of the asymptotic behaviors for the nonlinearity $|\cdot|^{\gamma} |u|^{\alpha-1}u$ and for general $\alpha>1$ is new, 
	even for $\gamma=0$. 

\end{itemize}
\end{rem}

\bigbreak
The rest of the paper is organized as follows: 
In Section 2, we collect some preliminary results for the weighted Lorentz spaces. 
In Section 3, we prove nonlinear estimates in weighted Lorentz spaces that are required to prove 
Theorems \ref{t:HH.LWP}, \ref{t:globalsol}, \ref{t:HH.self.sim} and \ref{t:asym.behv}. 
In Section 4, we prove the local well-posedness theory, that is, Theorem \ref{t:HH.LWP}. 
Section 5 is dedicated to the proof of Theorems \ref{t:globalsol}, \ref{t:HH.self.sim} and Theorem \ref{t:stab.asym.behv}. 
In Section 7, we prove the asymptotic result for the real-valued initial data (Theorem \ref{t:asym.behv}). 
Finally, Section 8 is dedicated to the asymptotic result for the complex-valued initial data (Theorem \ref{t:asym.behv-complex}). 
In Appendix, we prove the non-existence result for supercritical initial data, which extends the result in \cite[Theorem 1.16]{CITT2022}.

\section{Preliminaries}
	\label{sec:2}

In this paper we use the symbols $a \lesssim b$ for $a, b \ge0$, 
which mean that there exists a harmless constant $C>0$ such that $a \le C b$. 
The symbol $a \sim b$ means that $a \lesssim b$ and $b \lesssim a$ happen simultaneously. 
Throughout the rest of the paper, we denote by $C$ 
a harmless constant that may change from line to line. 
We denote by $C_0^\infty(\mathbb R^d)$ the set of all $C^\infty$-functions 
having compact support in $\mathbb R^d$, and by 
$L^0(\mathbb R^d)$ the set of all Lebesgue measurable functions on $\mathbb R^d$.
We define the distribution function $d_f$ of a measurable function $f$ by
\[
	d_f (\lambda) := \left| \left\{ x \in \mathbb R^d \,;\, |f(x)| > \lambda \right\} \right|,
\]
where $|A|$ denotes the Lebesgue measure of a measurable set $A$.


\subsection{Weighted Lorentz space}

\begin{defi}
For $0< q,r \le \infty$, the Lorentz space $L^{q,r}(\mathbb R^d)$ is defined by 
\[
L^{q,r}(\R^d):=\left\{ f \in L^0(\mathbb R^d) \,;\, 
	\|f\|_{L^{q,r}}
	< \infty \right\}
\]
endowed with a quasi-norm 
\[
	\|f\|_{L^{q,r}}
	:=
	\left\{\begin{aligned}
	&\left( 
		\int_0^\infty ( t^{\frac{1}{q}} f^*(t) )^r\frac{dt}{t}
	\right)^{\frac{1}{r}}
	\quad & \text{if } r<\infty, \\
	&\sup_{t>0} t^\frac{1}{q} f^*(t) & \text{if }r=\infty,
	\end{aligned}\right.
\]
where $f^*$ is the radially symmetric decreasing rearrangement of $f$ given by 
\[
	f^*(t) := \inf \{ \lambda > 0 \,; \, d_f(\lambda) \le t \}. 
\]
Note that $L^{\infty,r}_s(\R^d) = \{0\}$ for any $r<\infty$. 
\end{defi}


\begin{defi}
	\label{def:WLS}
Let $s \in \mathbb R$ and $0< q,r \le \infty$. 

\begin{itemize} 

\item[(i)]
The weighted Lebesgue space $L^{q}_s(\mathbb R^d)$ is defined by
\[
L^{q}_s(\mathbb R^d)
:= 
\left\{
f \in L^0(\mathbb R^d) \, ;\, 
\|f\|_{L^{q}_s} <\infty
\right\}
\]
endowed with a quasi-norm
\[
	\|f\|_{L^{q}_s}:= 
	\left(
	\int_{\mathbb R^d} (|x|^s|f(x)|)^q\, dx
	\right)^\frac{1}{q}.
\]
The space $\mathcal L^{q}_s (\mathbb R^d)$ is defined as the completion of $C^\infty_0(\mathbb R^d)$ with respect to $\|\cdot\|_{L^{q}_s}$.

\item[(ii)]
The weighted Lorentz space $L^{q,r}_s(\mathbb R^d)$ is defined by
\[
L^{q,r}_s(\mathbb R^d)
:= 
\left\{
f \in L^0(\mathbb R^d) \, ;\, 
\|f\|_{L^{q,r}_s} <\infty
\right\}
\]
endowed with a quasi-norm
\[
	\|f\|_{L^{q,r}_s}:= \| |\cdot |^s f\|_{L^{q,r}}.
\]
The space $\mathcal L^{q,r}_s (\mathbb R^d)$ is defined as the completion of $C^\infty_0(\mathbb R^d)$ with respect to $\|\cdot\|_{L^{q,r}_s}$.
\end{itemize}
\end{defi}

\begin{rem}\label{rem:w-Lorentz}
Let us give several remarks on $L^{q,r}_s(\R^d)$. 
\begin{itemize}

\item[(a)] 
$L^{q,r}_s(\R^d)=L^{q,r}(\R^d)$ when $s=0$. 

\item[(b)]
$L^{q,q}_s(\mathbb R^d)=L^{q}_s(\mathbb R^d)$ 
and $\mathcal L^{q,q}_s(\mathbb R^d)=\mathcal L^{q}_s(\mathbb R^d)$ for 
any $s\in \mathbb R$ and $1\le q\le \infty$.

\item[(c)] 
$L^{\infty,r}_s(\R^d) = \{0\}$ for any $r<\infty$. 
Hence, in this paper, we always take $r=\infty$ 
when $q=\infty$ in $L^{q,r}_s(\R^d)$ even if it is not explicitly stated.

\item[(d)] $L^{q,r}_s(\R^d)$ is a quasi-Banach space 
for any $s \in \mathbb R$ and $0<q,r \le \infty$, and it is normable 
if $(q,r) \in (1,\infty) \times [1,\infty]$, $(q,r) = (1,1)$ or $(q,r) =(\infty,\infty)$. 
When $q=1$, $L^{1,1}_s(\R^d)=L^1_s(\R^d)$ is Banach while $L^{1,r}_s(\R^d)$ is not normable for any $r>1$. 
Hence, in this paper, we always take $r=1$ 
when $q=1$ in $L^{q,r}_s(\R^d)$ even if it is not explicitly stated.

\item[(e)] $\mathcal L^{q,r}_s (\mathbb R^d) = L^{q,r}_s (\mathbb R^d)$ if $q<\infty$ and $r<\infty$, and $\mathcal L^{q,r}_s (\mathbb R^d) \subsetneq L^{q,r}_s (\mathbb R^d)$ if $q=\infty$ or $r=\infty$.

\item[(f)] $L^{q,r}_s(\R^d) \subset L^1_{\mathrm{loc}}(\R^d)$ 
if and only if either of {\rm (e-1)}--{\rm (e-3)} holds:
	\begin{itemize}
		\item[(e-1)] $q>1$ and $\frac{s}{d} + \frac{1}{q} < 1$;
		\item[(e-2)] $q>1$, $\frac{s}{d} + \frac{1}{q} = 1$ and $r\le 1$;
		\item[(e-3)] $q=1$, $\frac{s}{d} + \frac{1}{q} \le 1$ and $r\le 1$.
	\end{itemize}

\end{itemize}
\end{rem}

The proofs of the above can be found in \cites{BenSha1988, CIT2022, CITT2022, Gr2008} and the references therein. 
The general density result for weighted Lorentz spaces is stated in \cite{CITT2022}. 

The following H\"older inequality in Lorentz spaces is well known. 

\begin{lem}[Generalized H\"older's inequality \cite{ONe1963}]
	\label{lem:Holder}
Let $1\le q, q_1,q_2 < \infty$ and $1\le r,r_1,r_2 \le \infty$.
Then the following assertions hold:
\begin{itemize}
\item[(i)] If 
\[
	\frac{1}{q} = \frac{1}{q_1} + \frac{1}{q_2}\quad \text{and}\quad \frac{1}{r} \le \frac{1}{r_1} + \frac{1}{r_2},
\]
then 
there exists a constant $C>0$ such that 
\[
\|f g\|_{L^{q,r}}
\le C \|f\|_{L^{q_1,r_1}}\|g\|_{L^{q_2,r_2}}
\]
for any $f \in L^{q_1,r_1}(\mathbb R^d)$ and $g \in L^{q_2,r_2}(\mathbb R^d)$.

\item[(ii)] 
There exists a constant $C>0$ such that 
\[
\|f g\|_{L^{q,r}}
\le C \|f\|_{L^{q,r}}\|g\|_{L^{\infty}}
\]
for any $f \in L^{q,r}(\mathbb R^d)$ and $g \in L^{\infty}(\mathbb R^d)$.
\end{itemize}
\end{lem}

The following proposition follows by Lemma \ref{lem:Holder}. 
This proposition gives the embedding \eqref{embd:id}. 

\begin{prop}
	\label{p:embd}
The embedding $L^{q_1,r_1}_{l_1} (\R^d) \hookrightarrow L^{q_2,r_2}_{l_2} (\R^d)$ holds, where
\begin{equation}\nonumber
	l_2\le l_1, \quad 
	1\le q_2 \le q_1 \le \infty, \quad 
	1\le r_1 \le r_2 \le \infty
		\quad\text{and}\quad
	\frac{l_1}{d} + \frac1{q_1} = \frac{l_2}{d} + \frac1{q_2}. 
\end{equation}
\end{prop}


%
%

\subsection{Linear estimate}

The following estimate for the heat semigroup $\{e^{t\Delta}\}_{t\ge0}$ 
in weighted Lorentz space is from \cite{CITT2022}. See also \cite{TayWei2022}. 

\begin{lem}
	\label{l:linear-main}
Let $d \in \mathbb N$, $l_1,l_2\in \mathbb R$, $1 \le q_1, q_2 \le \infty$ and $1\le r_1,r_2 \le \infty$. 
Then there exists a constant $C>0$ 
such that 
\begin{equation}	\nonumber
	\|e^{t\Delta}f\|_{L^{q_2,r_2}_{l_2}} 
	\le C t ^{-\frac{d}{2} (\frac{1}{q_1} - \frac{1}{q_2}) - \frac{l_1 - l_2}{2}} \|f\|_{L^{q_1,r_1}_{l_1}}
\end{equation}
for any $f\in L^{q_1,r_1}_{s_1}(\R^d)$ and $t>0$ if $l_1$, $l_2$, $q_1,$ $q_2,$ $r_1$ and $r_2$ satisfy 
\begin{equation}\nonumber
	l_2 \le l_1, \quad
	0 \le \frac{l_2}{d} + \frac{1}{q_2} \le \frac{l_1}{d} + \frac{1}{q_1} \le 1 
\end{equation}
and
\begin{empheq}[left={\empheqlbrace}]{alignat=2}
  & r_1 = 1
 	\quad \text{if } \frac{l_1}{d} + \frac{1}{q_1}=1 
 	\text{ or } q_1=1, 
		\nonumber\\
  & r_2 = \infty 
  	\quad \text{if } \frac{l_2}{d} + \frac{1}{q_2}=0, 
		\nonumber\\
  & r_1 \le r_2 
  	\quad \text{if }\frac{l_2}{d} + \frac{1}{q_2}=\frac{l_1}{d} + \frac{1}{q_1}, 
		\nonumber\\
  & r_i=\infty 
  	\quad \text{if } q_i=\infty \quad (i=1,2). 
		\nonumber
\end{empheq}
\end{lem}

\section{Nonlinear estimates}

\subsection{Nonlinear estimates for local well-posedness}

Here, we study the estimate of the difference of a nonlinear map associated to the integral equation \eqref{intHHNLH}. 
Let $N : u \mapsto N(u)$ be a map defined on $\mathcal{K}^{p}_{k}(T)$ by 
\begin{equation}	\label{mapN}
	N(u)(t) :=  \int_0^t e^{(t-\tau)\Delta} 
	\left\{ |\cdot|^{\gamma} F(u(\tau,\cdot)) \right\} d\tau
	\quad\text{and}\quad
	F(u) := |u|^{\alpha-1}u. 
\end{equation}
Hereafter, we make use of the beta function $B:(0,\infty)^2\rightarrow \R_{>0}$ defined by 
\begin{equation}	\nonumber
	B(s,t):=\int_0^1\tau^{s-1}(1-\tau)^{t-1}d\tau. 
\end{equation}

\subsubsection{Contraction estimate}
The following are the contraction estimates for both critical and subcritical regimes. 

\begin{lem}
	\label{l:Kato.est}
Let $T \in (0,\infty]$ and $d\in\mathbb{N}.$ Let $\gamma\in\R$ and $\alpha\in\R$ satisfy 
\begin{equation}\nonumber
	\alpha>\max\left\{1, 1+\frac{\gamma}{d} \right\}
		\quad\text{and}\quad
	-\min\{2,d\}<\gamma.
\end{equation}
Let $l\in\R$ and $q\in [1,\infty]$ satisfy 
\begin{equation}	\label{d:sub+critical}
	0\le \frac{l}{d} + \frac{1}{q} \le \min\left\{1, \frac1{q_c(\gamma)} \right\}. 
\end{equation}
Let $k\in\R$ and $p\in [1,\infty]$ satisfy
\begin{equation}	\label{l:Kato.est.c1}
	\begin{aligned}
	&\frac{\gamma}{\alpha-1} \le k, 
		\quad
	\alpha < p \le \infty, 
		\quad
	\max\left\{0, \,\frac{\gamma}{d(\alpha -1)} \right\} \le \frac{k}{d} + \frac{1}{p}\\
		\text{and}\quad
	&\frac{l}{d} + \frac1{q} -\frac{2}{d\alpha} 
		< \frac{k}{d} + \frac1{p}
		< \min\left\{\frac{1}{q_c(\gamma)}, \, \frac{d+\gamma}{d\alpha} \right\}. 
	\end{aligned}
\end{equation}
Then there exists a constant $C_0>0$ 
such that the map $N$ defined by \eqref{mapN} satisfies 
\begin{equation}	\nonumber
\begin{aligned}
	\|N(u) - N(v)\|_{\mathcal{K}^{p}_{k}(T)} 
	\le C_0 T^{\frac{d(\alpha-1)}{2}(\frac1{q_c(\gamma)} - \frac{l}{d}-\frac{1}{q})}
	\left(\|u\|_{\mathcal{K}^{p}_{k}(T)}^{\alpha-1} 
				+\|v\|_{\mathcal{K}^{p}_{k}(T)}^{\alpha-1}  \right)
		\|u-v\|_{\mathcal{K}^{p}_{k}(T)}
\end{aligned}
\end{equation}
for all $u,v \in \mathcal{K}^{p}_{k}(T)$. 
\end{lem}

\begin{proof}
For all $\alpha>1$, there exists a constant $C=C(\alpha)$ such that 
\begin{equation}	\label{diff.pt.est}
	|F(u)-F(v)| \le C (|u|^{\alpha-1} + |v|^{\alpha-1})|u-v| 
	\quad\text{for all} \quad u,v \in \mathbb{C}. 
\end{equation}
We have 
\begin{align*}
	&\|N(u)(t) - N(v)(t)\|_{L^{p,\infty}_{k}} \\
	&\lesssim \int_0^t (t-\tau)^{-\frac{d(\alpha-1)}{2p}
			- \frac12 \{(\alpha-1) k -\gamma\}} 
	\left\| |\cdot|^{\gamma} 
	(|u(\tau)|^{\alpha-1} + |v(\tau)|^{\alpha-1})|u(\tau)-v(\tau)| \right\|_{L^{\frac{p}{\alpha},\infty}_{\alpha k-\gamma}} 
	d\tau \\
	&\lesssim t^{\frac{d(\alpha-1)}{2}(\frac1{q_c(\gamma)} - \frac{l}{d}-\frac{1}{q})} 
	t^{-\frac{d}{2}(\frac{l}{d} + \frac1{q} - \frac{k}{d} - \frac1{p})} 
	 \left( \|u\|_{\mathcal{K}^{p}_{k}(T)}^{\alpha-1} + \|v\|_{\mathcal{K}^{p}_{k}(T)}^{\alpha-1} \right)
		\, \|u-v\|_{\mathcal{K}^{p}_{k}(T)} \\
	&\times
	B\left( \frac{d(\alpha-1)}2 \left( \frac{1}{q_c(\gamma)} - \frac1{p}-\frac{k}{d}\right), 
		\frac{d\alpha}2 \left( \frac{k}{d} + \frac1{p}  
			+ \frac{2}{d\alpha} - \frac{l}{d} - \frac1{q}\right)\right), 
\end{align*}
thanks to Lemma \ref{l:linear-main} 
and Lemma \ref{lem:Holder}. This concludes the lemma. 
\end{proof}

When $\frac{l}{d} + \frac{1}{q} = \frac1{q_c(\gamma)}$, the above reads as follows. 

\begin{cor}	\label{c:Kato.est}
Let $T \in (0,\infty]$ and $d\in\mathbb{N}.$ 
Let $\gamma\in\R$ and $\alpha\in\R$ satisfy \eqref{t:HH.LWP.c.paramet}. 
Let $l\in\R$ and $q\in [1,\infty]$ satisfy \eqref{d:critical}. 
Let $k\in\R$ and $p\in [1,\infty]$ satisfy
\begin{equation}	\label{c:Kato.est.c1}
	\begin{aligned}
	&\frac{\gamma}{\alpha-1} \le k, 
		\quad
	\alpha < p \le \infty, \\
	&0\le \frac{k}{d} + \frac{1}{p} 
		\quad\text{and}\quad
	\frac{2+\gamma\alpha}{d\alpha(\alpha-1)} 
		< \frac{k}{d} + \frac1{p}
		< \min\left\{\frac{1}{q_c(\gamma)}, \, \frac{d+\gamma}{d\alpha} \right\}. 
	\end{aligned}
\end{equation}
Then there exists a constant $C_0>0$ 
such that the map $N$ defined by \eqref{mapN} satisfies 
\begin{equation}	\nonumber
\begin{aligned}
	\|N(u) - N(v)\|_{\mathcal{K}^{p}_{k}(T)} 
	\le C_0 \left(\|u\|_{\mathcal{K}^{p}_{k}(T)}^{\alpha-1} 
				+\|v\|_{\mathcal{K}^{p}_{k}(T)}^{\alpha-1}  \right)
		\|u-v\|_{\mathcal{K}^{p}_{k}(T)}
\end{aligned}
\end{equation}
for all $u,v \in \mathcal{K}^{p}_{k}(T)$. 
\end{cor}

\subsubsection{Stability estimates}

The following are the stability estimates for both subcritical and critical norms. 

\begin{lem}	\label{l:stb.est}
Let $T \in (0,\infty]$ and $d\in\mathbb{N}.$ 
Let $\gamma\in\R$ and $\alpha\in\R$ satisfy $\alpha>1$ and $-\min\{2,d\}<\gamma$.
Let $l\in\R$ and $q\in [1,\infty]$ satisfy \eqref{d:sub+critical} and 
\begin{equation}	\label{l:stb.est.c0}
	0\le \frac{l}{d} + \frac1{q} <1. 
\end{equation}
Let $r\in[1,\infty]$, and let $r=\infty$ if $q=\infty$ or $0= \frac{l}{d} + \frac{1}{q}$. Let $k\in\R$ and $p \in [1,\infty]$ be such that 
\begin{equation}	\label{l:stb.est.c1}
\begin{aligned}
	&\frac{l+\gamma}{\alpha} \le k, 
		\quad
	\alpha < p \le \infty \\
		\quad\text{and}\quad
	&\frac1{\alpha}\left( \frac{l}{d} + \frac1{q} + \frac{\gamma}{d} \right)
		<  \frac{k}{d} +\frac1{p}
		< \min\left\{ \frac{d+\gamma}{d\alpha}, \, 
		\frac1{\alpha} \left(\frac{l}{d} + \frac1{q} + \frac{2+\gamma}{d} \right) \right\}. 
\end{aligned}
\end{equation}
Then there exists a constant $C_1>0$ 
such that the map $N$ defined by \eqref{mapN} satisfies 
\begin{equation}	\label{l:stb.est:a1}
	\|N(u)-N(v)\|_{L^\infty(0,T ; L^{q,r}_{l})} 
		\le C_1 T^{\frac{d(\alpha-1)}{2}(\frac1{q_c(\gamma)} - \frac{l}{d}-\frac{1}{q})}
		\left(\|u\|_{\mathcal{K}^{p}_{k}(T)}^{\alpha-1}
		+ \|v\|_{\mathcal{K}^{p}_{k}(T)}^{\alpha-1}\right)
		\|u-v\|_{\mathcal{K}^{p}_{k}(T)}
\end{equation}
for all $u,v \in \mathcal{K}^{p}_{k}(T)$. 
\end{lem}

\begin{proof}
Let $T>0$ and $u,v \in \mathcal{K}^{p}_{k}(T).$ From \eqref{diff.pt.est}, we have 
\begin{align*}
\|&N(u)(t)-N(v)(t)\|_{L^{q,r}_{l}} \\
	&\lesssim t^{\frac{d(\alpha-1)}{2}(\frac1{q_c(\gamma)} - \frac{l}{d}-\frac{1}{q})} 
		\left(\|u\|_{\mathcal{K}^{p}_{k}(T)}^{\alpha-1}
				+ \|u\|_{\mathcal{K}^{p}_{k}(T)}^{\alpha-1}\right)
			\|u-v\|_{\mathcal{K}^{p}_{k}(T)} \\
	&\qquad \times B\left(\frac{d\alpha}2 
		\left( \frac1{\alpha} \left(\frac{2+\gamma}{d} + \frac{l}{d} + \frac1{q} \right)
			 - \frac{k}{d} - \frac1{p}\right), \,
			\frac{d\alpha}2 \left( \frac{k}{d} + \frac1{p}  
			+ \frac{2}{d\alpha} - \frac{l}{d} - \frac1{q}\right)
		\right), 
\end{align*}
thanks to Lemma \ref{lem:Holder} and Lemma \ref{l:linear-main}. 
This completes the proof of the lemma. 
\end{proof}

In the critical case, the above lemma reads as follows. 

\begin{cor}	\label{c:crt.stb}
Let $T \in (0,\infty]$ and $d\in\mathbb{N}.$ 
Let $\gamma\in\R$ and $\alpha\in\R$ satisfy \eqref{t:HH.LWP.c.paramet} . 
Let $l\in\R$ and $q\in [1,\infty]$ satisfy \eqref{d:critical} and \eqref{l:stb.est.c0}, 
let $r\in[1,\infty]$ and $l=0$, and let $r=\infty$ if $q=\infty$. 
Let $k\in\R$ and $p \in [1,\infty]$ be such that 
\begin{equation}	\label{c:crt.stb.c1}
	\frac{l+\gamma}{\alpha} \le k 
		\quad
	\alpha < p \le \infty
		\quad\text{and}\quad
	\frac{2+\gamma\alpha}{d\alpha(\alpha-1)}
		<  \frac{k}{d} +\frac1{p}
		< \min\left\{ \frac{1}{q_c(\gamma)}, \, \frac{d+\gamma}{d\alpha}\right\}. 
\end{equation}
Then there exists a constant $C_1>0$ 
such that the map $N$ defined by \eqref{mapN} satisfies 
\begin{equation}\nonumber	
	\|N(u)-N(v)\|_{L^\infty(0,T ; L^{q,r}_{l})} 
		\le C_1 \left(\|u\|_{\mathcal{K}^{p}_{k}(T)}^{\alpha-1}
		+ \|v\|_{\mathcal{K}^{p}_{k}(T)}^{\alpha-1}\right)
		\|u-v\|_{\mathcal{K}^{p}_{k}(T)}
\end{equation}
for all $u,v \in \mathcal{K}^{p}_{k}(T)$. 
\end{cor}

\subsection{Nonlinear estimates for asymptotic behaviors}

The following lemma is used to prove the asymptotic stability result (Theorem \ref{t:stab.asym.behv}). 
Its proof is inspired by \cite{MiaYua2007}. 

\begin{lem}	\label{l:crt.est2}
Let $d\in\mathbb{N}.$ 
Let $\gamma\in\R$ and $\alpha\in\R$ satisfy \eqref{t:HH.LWP.c.paramet}. 
Let $l\in\R$ and $q\in (1,\infty]$ satisfy \eqref{d:critical}. 
Let $r\in[1,\infty]$ and let $r=\infty$ if $q=\infty$. 
Let $k\in\R$ and $p \in [1,\infty]$ be such that 
\begin{equation}	\label{l:crt.est2.c1}
\begin{aligned}
	&0\le \frac1{p} < \frac1{\alpha-1} \left(1-\frac1{q}\right), \quad
	\frac{\gamma}{\alpha-1} \le k \\
		\quad\text{and}\quad
	&\frac{\gamma}{d(\alpha-1)}
		< \frac{k}{d} + \frac1{p}
		< \min\left\{\frac{1}{q_c(\gamma)}, \, 
		\frac{1}{\alpha-1} \left(1+\frac{\gamma}{d}-\frac1{q_c(\gamma)}\right) \right\}. 
\end{aligned}
\end{equation}
Then there exists a constant $C_2>0$ 
such that the map $N$ defined by \eqref{mapN} satisfies 
\begin{equation}	\label{l:crt.est2:a1}
\begin{aligned}
	\limsup_{t\to\infty}&\|N(u)(t)-N(v)(t)\|_{L^{q,r}_{l}} \\
	&\le C_2 \left(\|u\|_{\mathcal{K}^{p}_{k}}^{\alpha-1}
		+ \|v\|_{\mathcal{K}^{p}_{k}}^{\alpha-1}\right)
		\limsup_{t\to\infty}\|u(t)-v(t)\|_{L^{q,r}_{l}}
\end{aligned}
\end{equation}
for all $u,v \in \mathcal{K}^{p}_{k}$. 
\end{lem}

\begin{proof}
Let $\tilde p$ be defined by $\frac1{\tilde p} := \frac{\alpha-1}{p} + \frac{1}{q}.$ From \eqref{diff.pt.est}, we have 
\begin{align*}
\|&N(u)(t)-N(v)(t)\|_{L^{q,r}_{l}} \\
	&\lesssim \int_0^t 
		(t-\tau)^{-\frac{d}{2} ((\alpha-1) (\frac{k}{d} + \frac{1}{p}) -\frac{\gamma}{d} ) } 
			\left( \| u(\tau) \|_{L^{p,\infty}_{k}}^{\alpha-1} 
				+\| v(\tau) \|_{L^{p,\infty}_{k}}^{\alpha-1} \right) 
			\| u(\tau) -v(\tau)\|_{L^{q,r}_{l}} \,d\tau \\
	&\lesssim \int_0^1
		(1-\tau)^{-\frac{d}{2} ((\alpha-1) (\frac{k}{d} + \frac{1}{p}) -\frac{\gamma}{d} ) } 
		\tau^{-\frac{d(\alpha-1)}{2} (\frac1{q_c(\gamma)} - \frac{k}{d} - \frac1{p} ) } 
		\|u(t \tau)-v(t \tau)\|_{L^{q,r}_{l}} d\tau \\
	&\qquad 
	\times \left(\|u\|_{\mathcal{K}^{p}_{k}}^{\alpha-1}
				+ \|v\|_{\mathcal{K}^{p}_{k}}^{\alpha-1}\right),
\end{align*}
thanks to Lemma \ref{l:linear-main}. 
Then we deduce 
\begin{align*}
 \sup_{t\ge t'}\|&N(u)(t)-N(v)(t)\|_{L^{q,r}_{l}} \\
	&\lesssim \int_0^1
		(1-\tau)^{-\frac{d}{2} ((\alpha-1) (\frac{k}{d} + \frac{1}{p}) -\frac{\gamma}{d} ) } 
		\tau^{-\frac{d(\alpha-1)}{2} (\frac1{q_c(\gamma)} - \frac{k}{d} - \frac1{p} ) } 
		 \sup_{t\tau \ge t'\tau} \|u(t \tau)-v(t \tau)\|_{L^{q,r}_{l}} d\tau \\
	&\qquad 
	\times \left(\|u\|_{\mathcal{K}^{p}_{k}}^{\alpha-1}
				+ \|v\|_{\mathcal{K}^{p}_{k}}^{\alpha-1}\right). 
\end{align*}
Now, by assumption we have 
\begin{equation}\nonumber
	 \sup_{t\tau \ge t'\tau} \|u(t \tau)-v(t \tau)\|_{L^{q,r}_{l}} \le  \|u\|_{L^{\infty}(0,T;L^{q,r}_{l})}+\|v\|_{L^{\infty}(0,T;L^{q,r}_{l})} < +\infty
	 \quad\text{for all } \tau \in [0,1]
\end{equation}
and 
\begin{equation}\nonumber
	\lim_{t'\to\infty} \sup_{t\tau \ge t'\tau} \|u(t \tau)-v(t \tau)\|_{L^{q,r}_{l}} 
	= \limsup_{t\to\infty} \|u(t)-v(t)\|_{L^{q,r}_{l}} 
	 \quad\text{for a.e. } \tau \in [0,1]. 
\end{equation}
Therefore, Lebesgue's dominated convergence theorem yields
\begin{align*}
\limsup_{t\to\infty} \|&N(u)(t)-N(v)(t)\|_{L^{q,r}_{l}} \\
	&\lesssim \int_0^1
		(1-\tau)^{-\frac{d}{2} ((\alpha-1) (\frac{k}{d} + \frac{1}{p}) -\frac{\gamma}{d} ) } 
		\tau^{-\frac{d(\alpha-1)}{2} (\frac1{q_c(\gamma)} - \frac{k}{d} - \frac1{p} ) } d\tau \\
	&\qquad 
	\times \left(\|u\|_{\mathcal{K}^{p}_{k}}^{\alpha-1}
				+ \|v\|_{\mathcal{K}^{p}_{k}}^{\alpha-1}\right) 
		\limsup_{t\to\infty} \|u(t \tau)-v(t \tau)\|_{L^{q,r}_{l}}. 
\end{align*}
This completes the proof of the lemma. 
\end{proof}

The following estimate is crucial to prove the asymptotic behavior for the linear case 
(Theorem \ref{t:asym.behv} (ii)). 

\begin{lem}	\label{l:Kato.est2}
Let $T \in (0,\infty]$ and $d\in\mathbb{N}.$ 
Let $\gamma\in\R$ and $\alpha\in\R$ satisfy \eqref{t:HH.LWP.c.paramet}. 
Let $l\in\R$ and $q\in [1,\infty]$ satisfy \eqref{d:critical}. 
Let $\sigma\in\R$ satisfies 
\begin{equation}	\label{l:Kato.est2.c.l0.q0}
	 \frac{d}{q_c(\gamma)} < \sigma <d. 
\end{equation}
Let $k, k_0\in\R$ and $p, p_0 \in [1,\infty]$ be such that 
\begin{equation}	\label{l:Kato.est2.c1}
\begin{aligned}
	&0\le \frac{\alpha-1}{p} + \frac1{p_0}< 1, 
		\quad
	\frac{\gamma}{\alpha-1} \le k, 
		\quad
	\frac{k}{d} + \frac1{p}< \frac{1}{q_c(\gamma)}\\
		\text{and}\quad
	&\frac{\sigma}{d}
	<(\alpha-1) \left(\frac{k}{d} + \frac1{p} \right) + \frac{k_0}{d} + \frac{1}{p_0} - \frac{\gamma}{d}
		 < 1. 
\end{aligned}
\end{equation}
Then there exists a constant $C_3>0$ 
such that the map $N$ defined by \eqref{mapN} satisfies 
\begin{equation}	\nonumber
\begin{aligned}
	\sup_{0<t< T} &t^{\frac{d}{2} (\frac{\sigma}{d}-\frac{k_0}{d}-\frac1{p_0}) } 
		\|N(u)(t)-N(v)(t)\|_{L^{p_0,\infty}_{k_0}} \\
	&\le C_3 \left(\|u\|_{\mathcal{K}^{p}_{k}(T)}^{\alpha-1} + \|v\|_{\mathcal{K}^{p}_{k}(T)}^{\alpha-1}\right)
		\sup_{0<t< T} t^{\frac{d}{2} (\frac{\sigma}{d}-\frac{k_0}{d}-\frac1{p_0}) } 
		\|u(t) -v(t)\|_{L^{p_0,\infty}_{k_0}}
\end{aligned}
\end{equation}
for all $u,v \in \mathcal{K}^{p}_{k}(T)$ such that 
\[
	\sup_{0<t< T} t^{\frac{d}{2} (\frac{\sigma}{d}-\frac{k_0}{d}-\frac1{p_0}) } 
		\|u(t)-v(t)\|_{L^{p_0,\infty}_{k_0}} < \infty. 
\] 
\end{lem}

\begin{proof}
Let $\tilde p$ be defined by $\frac1{\tilde p} := \frac{\alpha-1}{p} + \frac{1}{p_0}$. From \eqref{diff.pt.est}, we have 
\begin{align*}
	\|&N(u)(t)-N(v)(t)\|_{L^{p_0,\infty}_{k_0}} \\
	&\lesssim \int_0^t 
		(t-\tau)^{-\frac{d}{2} ((\alpha-1) (\frac{k}{d} + \frac{1}{p}) -\frac{\gamma}{d} ) } 
			\left( \| u(\tau) \|_{L^{p,\infty}_{k}}^{\alpha-1} 
				+\| v(\tau) \|_{L^{p,\infty}_{k}}^{\alpha-1} \right) 
			\| u(\tau) -v(\tau)\|_{L^{p_0,\infty}_{k_0}} \,d\tau,
\end{align*}
thanks to Lemma \ref{l:linear-main}. 
Thus, we have 
\begin{align*}
	&\sup_{0<t<T} t^{\frac{d}{2} (\frac{\sigma}{d}-\frac{k_0}{d}-\frac1{p_0}) } 
	\|N(u)(t)-N(v)(t)\|_{L^{p_0,\infty}_{k_0}} \\
	&\lesssim B\left( -\frac{d}{2} \left((\alpha-1) \left(\frac{k}{d} + \frac{1}{p} \right) -\frac{\gamma}{d} \right),\,   
		\frac{d}2 \left( (\alpha-1) \left(\frac{k}{d} + \frac1{p} \right) + \frac{k_0}{d} + \frac{1}{p_0} \right) - \frac{\sigma+\gamma}{2}\right) \\
	&\qquad 
	\times \left(\|u\|_{\mathcal{K}^{p}_{k}(T)}^{\alpha-1}
				+ \|u\|_{\mathcal{K}^{p}_{k}(T)}^{\alpha-1}\right)
		\sup_{0<t< T} t^{\frac{d}{2} (\frac{\sigma}{d}-\frac{k_0}{d}-\frac1{p_0}) } 
		\|u(t) -v(t)\|_{L^{p_0,\infty}_{k_0}}.
\end{align*}
This completes the proof of the lemma. 
\end{proof}

\subsection{Upgrade of regularity}

The following lemma is used to show the regularity of the $L^{q,r}_{l}(\R^d)$-mild solution. 

\begin{lem}	\label{l:b.strap}
Let $k_1,k_2 \in \R$, $p_1,p_2 \in [1,\infty]$ and $T\in (0,\infty]$. 
Under condition \eqref{t:HH.LWP.c.paramet}, 
let pairs $(k_1, p_1)$ and $(k_2, p_2)$ be such that 
\begin{equation}	\label{l:b.strap:c}
\begin{aligned}
	&k_2\le \min\{k_1,  \alpha k_1 - \gamma\}, 
	\quad \alpha \le p_1 \le \infty,  
	\quad 1\le p_2 \le \infty,  \\
	&0\le \frac{k_2}{d} + \frac1{p_2} 
	\le \min\left\{  \frac{k_1}{d} + \frac1{p_1},\, 
		\alpha \left( \frac{k_1}{d} + \frac1{p_1} \right) - \frac{\gamma}{d} \right\}, \\
	&\max\left\{  \frac{k_1}{d} + \frac1{p_1},\, 
		\alpha \left( \frac{k_1}{d} + \frac1{p_1} \right) - \frac{\gamma}{d} \right\} 
	\le 1
		\quad\text{and}\quad 
	\alpha\left( \frac{k_1}{d} + \frac1{p_1}\right) - \frac{2+\gamma}{d}
	<\frac{k_2}{d} + \frac1{p_2}. 
\end{aligned}
\end{equation}
Let $u$ be the solution to \eqref{intHHNLH}. Then the following assertions hold:
\begin{itemize}
\item[(i)] $($$r_1=1$$)$ \ If $p_1<\infty$ and 
	\begin{equation}\nonumber
		\sup_{0<t<T} t^{ \frac{d}2(\frac1{q_c(\gamma)} - \frac{k_1}{d} - \frac1{p_1}) } 
		\|u(t)\|_{L^{p_1,1}_{k_1}} < \infty
	\end{equation} 
	then it follows that 
	\begin{equation}\nonumber
		\sup_{0<t<T} 
		t^{ \frac{d}2(\frac1{q_c(\gamma)} - \frac{k_2}{d} - \frac1{p_2}) } 
			\|u(t)\|_{L^{p_2,\infty}_{k_2}} < \infty. 
	\end{equation}
	Moreover, 
	\begin{equation}\nonumber
		\sup_{0<t<T} 
		t^{ \frac{d}2(\frac1{q_c(\gamma)} - \frac{k_2}{d} - \frac1{p_2}) } 
			\|u(t)\|_{L^{p_2,1}_{k_2}} < \infty
	\end{equation}
	provided further that
	\begin{equation}	\label{l:b.strap:c1}
		p_2<\infty 
			\quad\text{and}\quad
		0< \frac{k_2}{d} + \frac1{p_2}. 
	\end{equation}
\item[(ii)] $($$r_1=\infty$$)$ If 
	\begin{equation}\nonumber
		\sup_{0<t<T} 
		t^{ \frac{d}2(\frac1{q_c(\gamma)} - \frac{k_1}{d} - \frac1{p_1} ) } 
		\|u(t)\|_{L^{p_1,\infty}_{k_1}} < \infty
	\end{equation}
	then it follows that 
	\begin{equation}\nonumber
		\sup_{0<t<T} 
		t^{ \frac{d}2(\frac1{q_c(\gamma)} - \frac{k_2}{d} - \frac1{p_2}) } 
			\|u(t)\|_{L^{p_2,\infty}_{k_2}} < \infty,
	\end{equation}
	provided further that 
	\begin{equation}	\label{l:b.strap:c2}
		\alpha<p_1, 
			\quad\text{and}\quad 
		\max\left\{  \frac{k_1}{d} + \frac1{p_1},\, 
			\alpha \left( \frac{k_1}{d} + \frac1{p_1} \right) - \frac{\gamma}{d} \right\} 
		< 1. 
	\end{equation}
	Moreover, 
	\begin{equation}\nonumber
		\sup_{0<t<T} 
		t^{ \frac{d}2(\frac1{q_c(\gamma)} - \frac{k_2}{d} - \frac1{p_2}) } 
			\|u(t)\|_{L^{p_2,1}_{k_2}} < \infty
	\end{equation}
	provided further that \eqref{l:b.strap:c1}, \eqref{l:b.strap:c2} and 
	\begin{equation}		\label{l:b.strap:c3}
		\frac{k_2}{d} + \frac1{p_2} 
		< \min\left\{  \frac{k_1}{d} + \frac1{p_1},\, 
			\alpha \left( \frac{k_1}{d} + \frac1{p_1} \right) - \frac{\gamma}{d} \right\}.
	\end{equation}
\end{itemize}
\end{lem}

\begin{proof}
We use a similar argument as in \cite{SnoTayWei2001} 
(See also \cite{BenTayWei2017}). We first prove assertion (i) for $r_1=r_2=1$. Let $p_1, p_2<\infty$ and 
\begin{equation}\nonumber
	A := \sup_{0<t<T} t^{ \frac{d}2(\frac1{q_c(\gamma)} -\frac{k_1}{d} -\frac1{p_1}) } 
		\|u(t)\|_{L^{p_1,1}_{k_1}} <\infty.
\end{equation}
Let $t\in (0,T)$. We use the integral representation 
\begin{equation}	\label{l:b.strap:int.rep}
	u(t) = e^{\frac{t}2 \Delta} u(t/2) 
	+ a \tilde N(u)(t),
\end{equation}
where 
\[
	 \tilde N(u)(t) 
	 :=\int_{\frac{t}{2}}^t e^{(t-\tau)\Delta} \left\{ |\cdot|^{\gamma} 
		F(u(\tau,\cdot)) \right\} d\tau. 
\]
It follows from Lemma \ref{l:linear-main} that
\begin{align*}
	\|e^{\frac{t}2\Delta} u(t/2)\|_{L^{p_2,1}_{k_2}} 
	\lesssim t^{-\frac{d}{2}(\frac1{p_1}-\frac1{p_2}) - \frac{k_1-k_2}2} \|u(t/2)\|_{L^{p_1,1}_{k_1}} 
	\lesssim t^{-\frac{d}{2}(\frac1{q_c(\gamma)}-\frac{k_2}{d}-\frac{1}{p_2})} A, 
\end{align*}
provided that 
\begin{equation}	\label{l:b.strap:pr0}
\begin{aligned}
	&1\le p_1, p_2 < \infty, \quad
	k_2 \le k_1
		\quad\text{and}\quad 
	0<\frac{k_2}{d} + \frac1{p_2} \le \frac{k_1}{d} + \frac1{p_1} \le 1. 
\end{aligned}
\end{equation}
On the other hand, 
\begin{equation}\nonumber
\begin{aligned}
\|&\tilde N(u)(t)\|_{L^{p_2,1}_{k_2}} 
\lesssim \int_{\frac{t}{2}}^t 
	(t-\tau)^{-\frac{d}2(\frac{\alpha}{p_1}-\frac1{p_2}) - \frac{\alpha k_1 -\gamma-k_2}2} 
		 \| u(\tau) \|_{L^{p_1,1}_{k_1}}^{\alpha} d\tau \\
&\lesssim A^{\alpha}  t^{-\frac{d}2(\frac1{q_c(\gamma)} - \frac{k_2}{d} - \frac1{p_2} )} 
	\int_{\frac{1}{2}}^{1} (1-\tau)^{-\frac{d}2(\frac{\alpha}{p_1}-\frac1{p_2}) 
 						-\frac{\alpha k_1 -k_2-\gamma}2} 
	\tau^{-\frac{d\alpha}2(\frac1{q_c(\gamma)} - \frac{k_1}{d} - \frac1{p_1}) } d\tau, 
\end{aligned}
\end{equation}
thanks to Lemma \ref{l:linear-main} 
where we require 
\begin{equation}	\label{l:b.strap:pr1}
\begin{aligned}
	&\alpha \le p_1 < \infty,  \quad 1\le p_2 <\infty, \quad
	k_2 \le \alpha k_1 - \gamma, \\
		\quad\text{and}\quad 
	&0<\frac{k_2}{d} + \frac1{p_2} 
		\le \frac{\alpha k_1 - \gamma}{d} + \frac{\alpha}{p_1} \le 1. 
\end{aligned}
\end{equation}
The final integral is convergent if 
\begin{equation}	\label{l:b.strap:pr2}
	\alpha\left( \frac{k_1}{d} + \frac1{p_1}\right) - \frac{2+\gamma}{d} < \frac{k_2}{d} + \frac1{p_2}. 
\end{equation}
Thus, we have 
\begin{equation}\nonumber
	\sup_{0<t<T} 
		t^{ \frac{d}2(\frac1{q_c(\gamma)} - \frac{k_2}{d} - \frac1{p_2}) } 
		\|u(t)\|_{L^{p_2,1}_{k_2}} 
	\le C (A + A^{\alpha})
\end{equation}
under \eqref{l:b.strap:pr0}, \eqref{l:b.strap:pr1} and \eqref{l:b.strap:pr2}. 
By a straightforward computation, 
we see that  \eqref{l:b.strap:pr0}, \eqref{l:b.strap:pr1} and \eqref{l:b.strap:pr2} are 
equivalent to \eqref{l:b.strap:c} with \eqref{l:b.strap:c1}. 
The case $r_1=1$, $r_2 = \infty$ is proved by a similar argument. 
We omit the proof of assertion (ii). 
\end{proof}

\section{Local well-posedness}

In this section, we prove Theorems \ref{t:HH.LWP} and \ref{t:globalsol}. 
Given an initial data $\varphi\in L^{q,r}_{l}(\R^d)$ with $(q,l)$ satisfying \eqref{d:sub+critical} and a time $T>0,$ 
we define a map $\mathscr{F} : u \mapsto \mathscr{F}(u)$ on $\mathcal{K}^{p}_{k}(T)$ for suitable $(k,p)$ by 
\begin{equation}	\label{map}
	\mathscr{F}_{\varphi} (u) (t) = \mathscr{F}(u) (t) := e^{t\Delta} \varphi + a N(u)(t)
\end{equation}
with $N(u)$ defined by \eqref{mapN}. 

In order to prove Theorems \ref{t:HH.LWP} and \ref{t:globalsol}, we prepare the following lemma, 
which is inspired by \cite{CazWei1998}*{Theorem 2.1} (See \cite{BenTayWei2017}*{Theorem 4.1} too). 

\begin{lem}[Abstract well-posedness result]
	\label{l:exist.crt}
Let $d\in\mathbb{N}.$ Let $\gamma\in\R$ and $\alpha\in\R$ satisfy \eqref{t:HH.LWP.c.paramet}. 
Let positive numbers $\rho>0$ and $M>0$ satisfy 
\begin{equation}	\label{l:exist.crt.c0}
	\rho + 2 C_0 M^{\alpha} \le M, 
\end{equation}
where $C_0$ is as in Corollary \ref{c:Kato.est}. 
Let $l\in\R$ and $q\in [1,\infty]$ satisfy \eqref{d:critical}, let $r\in[1,\infty]$, and let $r=\infty$ if $q=\infty$. 
Let $k\in\R$ and $p \in [1,\infty]$ be such that 
\begin{equation}	\label{l:exist.crt.c.pk}
	\begin{aligned}
	&\max\left\{\frac{\gamma}{\alpha -1}, \frac{l+\gamma}{\alpha} \right\}\le k, 
		\quad
	\alpha < p \le \infty, \\
	&0\le \frac{k}{d} + \frac{1}{p}
		\quad\text{and}\quad
	\frac{2+\gamma\alpha}{d\alpha(\alpha-1)} 
		< \frac{k}{d} + \frac1{p}
		<\min\left\{ \frac{1}{q_c(\gamma)}, \, \frac{d+\gamma}{d\alpha}\right\}.
	\end{aligned}
\end{equation}
If $\varphi \in \mathcal{S}'(\R^d)$ and $T\in (0,\infty]$ are such that 
\begin{equation}	\label{l:exist.crt:c.small}
	\|e^{t\Delta}\varphi\|_{\mathcal{K}^{p}_{k}(T)} \leq \rho, 
\end{equation}
then a solution $u \in \mathcal{K}^{p}_{k}(T)$ to \eqref{HHNLH} exists such that 
\[
	u -e^{t\Delta} \varphi 
	\in L^\infty(0,T ; L^{q,r}_{l}(\R^d)) \cap C((0,T) ; L^{q,r}_{l}(\R^d)), 
\]
\begin{equation}	\label{l:exist.crt:est}
	\|u\|_{\mathcal{K}^{p}_{k}(T)} \le M 
		\quad\text{and}\quad
	\|u -e^{t\Delta} \varphi\|_{L^\infty(0,T ; L^{q,r}_{l})} \le C_1 \|u\|_{\mathcal{K}^{p}_{k}(T)}^{\alpha}, 
\end{equation}
where $C_1$ is as in Corollary \ref{c:crt.stb}. Moreover, the solution satisfies the following properties: 
\begin{itemize}

\item[(i)] $u(t) -e^{t\Delta} \varphi \in L^{p_1,r_1}_{k_1}(\R^d)$ 
	for all $t>0$ and $(k_1,p_1,r_1) \in \R \times [1,\infty]\times[1,\infty]$ such that 
	\begin{equation}	\label{l:exist.crt.c.pk1}
	\left\{\begin{aligned}
	& k_1 \le \alpha k -\gamma, \quad 
		0\le  \frac{k_1}{d} + \frac1{p_1} 
		\le \alpha \left( \frac{k}{d} + \frac{1}{p} \right) - \frac{\gamma}{d}, \\
	&\alpha\left(\frac{k}{d}+\frac1{p} - \frac{2+\gamma}{d\alpha} \right) < \frac{k_1}{d} + \frac1{p_1}, \\
	& r_1 = \infty 
		\quad \text{if } \frac{k_1}{d} + \frac{1}{p_1}=0 
		\text{ or } \frac{k_1}{d} + \frac{1}{p_1}
			= \alpha\left(\frac{k}{d}+\frac{1}{p}\right)-\frac{\gamma}{d}
		\text{ or } p_1=\infty. 
	\end{aligned}\right.
	\end{equation}
	Moreover, there exists some positive constant $C$ such that 
	\[
		\|u(t) -e^{t\Delta} \varphi\|_{L^{p_1,r_1}_{k_1}} 
		\le Ct^{\frac{d}2(\frac{k_1}{d} + \frac1{p_1} -\frac1{q_c(\gamma)} )}
			 \|u \|_{\mathcal{K}^{p}_{k}(T)}^{\alpha}
	\]
	for any $t\in(0,T)$. 

\item[(ii)] For $(k_1,p_1) \in \R\times [1,\infty]$ satisfying 
	\begin{equation}	\label{l:exist.crt.c.pk1.2}
		\frac1{q_c(\gamma)} < \frac{k_1}{d} + \frac1{p_1} 
	\end{equation}
	in addition to \eqref{l:exist.crt.c.pk1}, we have 
	$u -e^{t\Delta} \varphi \in C([0,T) ; L^{p_1,r_1}_{k_1}(\R^d))$ and 
	\[
		\Lim_{t\to0} \|u(t) -e^{t\Delta} \varphi\|_{L^{p_1,r_1}_{k_1}} = 0.
	\] 

\item[(iii)]	$\displaystyle \lim_{t\to0} u(t) = \varphi$ in the sense of distributions. 

\item[(iv)] We have $\displaystyle \|u\|_{\mathcal{K}^{\tilde p}_{\tilde k}(T)} < \infty$ 
	for all $(\tilde k,\tilde p) \in \R \times [1,\infty]$ satisfying \eqref{t:asym.behv:c.kp}. 

\item[(v)] 
	Let $M$ be smaller than in \eqref{l:exist.crt.c0} if necessary. 
	Let $\sigma\in\R$ satisfy 
	\eqref{l:Kato.est2.c.l0.q0} and let 
	\begin{equation}	\label{l:exist.crt:c.sigma}
		\|e^{t\Delta} \varphi \|_{\tilde{\mathcal{K}}^{p_0}_{k_0}(T)} 
		=\sup_{t>0} t^{\frac{d}{2} (\frac{\sigma}{d}-\frac{k_0}{d}-\frac1{p_0} )} 
		\|e^{t \Delta} \varphi \|_{L^{p_0,\infty}_{k_0}}  <\infty, 
	\end{equation}
	$k_0 \in \R$ and $p_0 \in [1,\infty]$ satisfy \eqref{l:Kato.est2.c1}, that is, 
	\begin{equation}		\label{l:exist.crt.c.p0k0}
	\begin{aligned}
		0\le \frac{\alpha-1}{p} + \frac1{p_0}< 1 
			\quad\text{and}\quad
		\frac{\sigma}{d}
		<(\alpha-1) \left(\frac{k}{d} + \frac1{p} \right) + \frac{k_0}{d} + \frac{1}{p_0} - \frac{\gamma}{d}
		< 1. 
	\end{aligned}
	\end{equation}
	Then 
	\begin{equation}	\label{l:exist.crt':est2}
		\|u\|_{\tilde{\mathcal{K}}^{\tilde p}_{\tilde k}(T)} 
		< \infty
	\end{equation}
	for all $(\tilde{k},\tilde{p}) \in \R \times [1,\infty]$ such that 
	\begin{equation}	\label{l:exist.crt':c.tild.kp}
		\tilde k \le k_0, \quad
		p_0 \le \tilde p  
			\quad\text{and}\quad
		0\le \frac{\tilde k}{d} + \frac1{\tilde p} \,.
	\end{equation}

\end{itemize}
	\smallbreak
Moreover, let $\varphi$ and $\psi$ satisfy \eqref{l:exist.crt:c.small} 
and let $u$ and $v$ respectively be the solutions of \eqref{intHHNLH}
with initial values $\varphi$ and $\psi$. Then there exists a positive constant $C$ such that  
\begin{equation}	 \label{l:exist.crt:cont.depnd}
	\|u-v\|_{\mathcal{K}^{\tilde p}_{\tilde k}(T)}
	\leq C \|e^{t\Delta}(\varphi-\psi)\|_{\mathcal{K}^{p}_{k}(T)}
\end{equation}
for all $(\tilde k,\tilde p)\in \R \times [1,\infty]$ satisfying \eqref{t:asym.behv:c.kp}. 
	\smallbreak
If furthermore, given arbitrary $T>0$, $e^{t\Delta}(\varphi-\psi)$ satisfies the estimate 
\begin{equation}	\label{l:exist.crt:strng.dcy.c}
	\sup_{0<t<T} 
		t^{\frac{d}2(\frac1{q_c(\gamma)} - \frac{k}{d} - \frac1{p} )+\delta} 
		\|e^{t\Delta}(\varphi-\psi)\|_{L^{p,\infty}_{k}} 
		<\infty
\end{equation}
for some $\delta>0$ such that $\frac{d\alpha }2(\frac1{q_c(\gamma)} - \frac{k}{d}- \frac1{p} )+\delta<1$, then
\begin{equation}	\label{l:exist.crt:strng.dcy.est}
	\sup_{0<t<T}
		t^{\frac{d}2(\frac1{q_c(\gamma)} - \frac{\tilde k}{d} - \frac1{\tilde p} ) +\delta}
		\|u(t)-v(t)\|_{L^{\tilde p,\infty}_{\tilde k}}
	\leq C \sup_{0<t<T} 
		t^{\frac{d}2(\frac1{q_c(\gamma)} - \frac{k}{d} - \frac1{p} ) +\delta}
		\|e^{t\Delta}(\varphi-\psi)\|_{L^{p,\infty}_{k}}
\end{equation}
holds for $(\tilde k,\tilde p)\in \R \times [1,\infty]$ satisfying \eqref{t:asym.behv:c.kp}, 
where $C>0$ is a positive constant depending on $M$, with $M$ perhaps smaller. 
\end{lem}

\begin{rem}	\label{r:exist.crt}
Note that \eqref{l:exist.crt.c.pk} implies that $\gamma<0$ necessarily if $k<0$. 
In other words, the weight $k$ of the auxiliary norm can be negative only for the Hardy case. 
\end{rem}

\begin{rem}
Lemma \ref{l:exist.crt} (v) is used to prove the linear asymptotic behavior. 
The additional regularity is used to utilize the interpolation argument in the proof for Theorem 
\ref{t:asym.behv.lin}. 
\end{rem}


\begin{proof}
Let 
\[
	X_M := \{ u \in\mathcal{K}^{p}_{k}(T) 
			\,;\, \|u\|_{\mathcal{K}^{p}_{k}(T)} \le M \}
\]
be the closed ball in $\mathcal{K}^{p}_{k}(T)$ centered at the origin with radius $M$.
Setting the metric $d(u,v) := \|u-v\|_{\mathcal{K}^{p}_{k}(T)}$, 
we may show that $(X_M,d)$ is a nonempty complete metric space. 
We prove that the map $\mathscr{F}$ defined by \eqref{map} has a fixed point in $X_M$, provided that $M$ is sufficiently small. 
To this end, let $\varphi$ and $\psi$ satisfy \eqref{l:exist.crt:c.small}. 
By Corollary \ref{c:Kato.est} and \eqref{l:exist.crt.c0}, we have 
\begin{equation}	\label{t:HH.LWP.pr.Lip}
	d\left( \mathscr{F}_{\varphi} (u), \mathscr{F}_{\psi} (v) \right) 
	\le \|e^{t\Delta} (\varphi-\psi) \|_{\mathcal{K}^{p}_{k}(T)} 
		+ 2 C_0 M^{\alpha-1} \|u-v\|_{\mathcal{K}^{p}_{k}(T)}, 
\end{equation}
which establishes \eqref{l:exist.crt:cont.depnd} for $\tilde p = p$. 
Estimate \eqref{t:HH.LWP.pr.Lip} obviously implies
\begin{equation}	\nonumber
\begin{aligned}
	\|\mathscr{F}_{\varphi} (u)\|_{\mathcal{K}^{p}_{k}(T)} 
	\le \rho + 2 C_0 M^{\alpha} \le M
\end{aligned}\end{equation} 
and 
\begin{equation}	\nonumber
	d(\mathscr{F}_{\varphi} (u), \mathscr{F}_{\varphi} (v)) \le 2 C_0 M^{\alpha-1} d(u,v).
\end{equation}
for any $u, v\in X_M,$ where $2 C_0 M^{\alpha-1}<1.$ 
These prove that $\mathscr{F}(u) \in X_M$ and that $\mathscr{F}$ is a contraction mapping in $X_M.$ 
Thus, Banach's fixed point theorem ensures the existence of a unique fixed point $u$ for the map $\mathscr{F}_{\varphi}$ in $X_M,$ 
provided that $k$ and $p$ satisfy \eqref{c:Kato.est.c1}. 
The fixed point $u$ also satisfies, by construction, the estimate $\|u\|_{\mathcal{K}^{p}_{k}(T)} \le M.$ 

Having obtained a fixed point in $\mathcal{K}^{p}_{k}(T)$ for some $T,$ 
we may confirm that $u -e^{t\Delta} \varphi$ belongs to $L^\infty(0,T;L^{q,r}_{l}(\R^d))$ and 
\begin{equation}\nonumber	
	\|u -e^{t\Delta} \varphi\|_{L^\infty(0,T ; L^{q,r}_{l})} \le C_1 \|u\|_{\mathcal{K}^{p}_{k}(T)}^{\alpha}, 
\end{equation}
where $C_1$ is that in Corollary \ref{c:crt.stb}, provided further that $k$ and $p$ satisfy \eqref{c:crt.stb.c1}. 
Thus, we gather that the conditions for the existence 
are summarized by \eqref{l:exist.crt.c.pk}. 

	\smallbreak
We next prove assertions (i)--(iii) of Lemma \ref{l:exist.crt}. 
Fix a solution $u \in \mathcal{K}^{p}_{k}(T)$ 
with $k$ and $p$ satisfying \eqref{l:exist.crt.c.pk}. By the same calculation as that for Lemma \ref{l:stb.est}, we have 
\begin{equation}	\label{t:HH.LWP:pr2}
\begin{aligned}
\|&N(u)(t)\|_{L^{p_1, r_1}_{k_1}} 
	\le C  t^{\frac{d}2(\frac{k_1}{d} + \frac1{p_1} -\frac1{q_c(\gamma)} )} 
		\|u \|_{\mathcal{K}^{p}_{k}(T)}^{\alpha} \\
	&\quad \times B\left(\frac{d}{2}\left( \frac{k_1}{d}+\frac1{p_1} 
		- \alpha\left(\frac{k}{d}+\frac1{p} - \frac{2+\gamma}{d\alpha} \right) \right), \, 
	\frac{d\alpha}2 \left( \frac{k}{d} + \frac1{p} 
			- \frac{2+\gamma\alpha}{d\alpha(\alpha-1)} \right) \right) \\
\end{aligned}
\end{equation}
thanks to Lemma \ref{l:linear-main}, 
where we require 
\begin{equation}	\label{t:HH.LWP:pr3}
\left\{\begin{aligned}
& 1\le p_1\le \infty, \quad 
	k_1 \le \alpha k-\gamma, \quad
	0 \le \frac{k_1}{d} + \frac{1}{p_1} 
		\le \alpha\left(\frac{k}{d}+\frac{1}{p}\right)-\frac{\gamma}{d}, \\
& r_1 = \infty 
	\quad 
	\text{if } \frac{k_1}{d} + \frac{1}{p_1}=0 
	\text{ or } \frac{k_1}{d} + \frac{1}{p_1}
		= \alpha\left(\frac{k}{d}+\frac{1}{p}\right)-\frac{\gamma}{d}
	\text{ or } p_1=\infty. 
\end{aligned}\right.
\end{equation}
The Beta function in \eqref{t:HH.LWP:pr2} is convergent if 
\begin{equation}	\label{t:HH.LWP:pr4}
	\alpha\left(\frac{k}{d}+\frac1{p} - \frac{2+\gamma}{d\alpha} \right)
		<\frac{k_1}{d}+\frac1{p_1}
	\quad\text{and}\quad
	\frac{2+\gamma\alpha}{d\alpha(\alpha-1)} < \frac{k}{d} + \frac1{p}. 
\end{equation}
If furthermore \eqref{l:exist.crt.c.pk1.2} is satisfied, 
then the right-hand side of \eqref{t:HH.LWP:pr2} goes to zero as $t\to0.$ 
Hence, assertions (i)--(iii) hold. 

	\smallbreak
We next prove assertion (iv). 
Fix a solution $u \in \mathcal{K}^{p}_{k}(T)$ of \eqref{intHHNLH} constructed by the first part of Lemma \ref{l:exist.crt}. 
Then we have 
\begin{equation*}
	\sup_{0<t<T} 
	t^{\frac{d}2(\frac1{q_c(\gamma)} - \frac{k}{d} - \frac1{p} )} 
	\|u(t)\|_{L^{p,\infty}_{k}} < \infty 
\end{equation*}
for $(k,p)$ satisfying \eqref{l:exist.crt.c.pk}. 
We first assume $k\ge 0$ and use Lemma \ref{l:b.strap} with an iterative argument as in \cite{SnoTayWei2001}: 
With $k$ fixed, $(k, p_j)$ will play the role of $(k_1, p_1)$, and $(k, p_{j+1})$ will play the role of $(k_2, p_2)$, for each integer $j\ge0$. 
By this iterative procedure, we can reach $p_{J+1}=\infty$ for some finite $J$ 
while keeping $k_j=k_{j+1}=k$ fixed, provided that $k$ is non-negative. 
Once this is accomplished, then we may interpolate between $(k, p)$ and $(k,\infty)$ to obtain 
\begin{equation}	\nonumber
	\sup_{0<t<T} t^{\frac{d}2(\frac1{q_c(\gamma)} - \frac{k}{d} -\frac{1}{\tilde p})}
		\|u(t)\|_{L^{\tilde p, \infty}_{k}}<\infty
\end{equation}
for $p< \tilde p < \infty$. Next, we fix $\tilde p \in [p,\infty]$ arbitrarily and prove (iv) 
for the case $(\tilde k, \tilde p)=(-\frac{d}{\tilde p}, \tilde p)$. 
We may perform this by a similar iterative procedure as above using Lemma \ref{l:b.strap}. 
By interpolation, we have $\displaystyle \|u\|_{\mathcal{K}^{\tilde p}_{\tilde k}(T)} < \infty$ 
for all $(\tilde k,\tilde p) \in [-\frac{d}{p}, k] \times [p,\infty]$ 
satisfying \eqref{t:asym.behv:c.kp}, provided that $k\ge 0$. 
The proof for the other case $k<0$ is similar so we omit the detail.

	\smallbreak
We next prove assertion (v). 
Define a metric space $(\tilde X_M,d)$ by 
\[
	\tilde X_M := \{ u \in\mathcal{K}^{p}_{k}(T) \,;\, 
			\|u\|_{\mathcal{K}^{p}_{k}(T)} \le M
				\ \text{and} \ 
			\|u\|_{\tilde{\mathcal{K}}^{p_0}_{k_0}(T)} 
			\le 2 \|e^{t\Delta} \varphi \|_{\tilde{\mathcal{K}}^{p_0}_{k_0}(T)}   \}
\]
with the metric $d(u,v) := \|u-v\|_{\mathcal{K}^{p}_{k}(T)}$, 
so that $\tilde X_M$ is a ball contained in $X_M$. 
Consider the same map $\mathscr{F}$ defined by \eqref{map}. 
By Lemma \ref{l:Kato.est2}, there exists a constant $C_3>0$ such that 
\begin{equation}	\nonumber
\begin{aligned}
	\|u\|_{\tilde{\mathcal{K}}^{p_0}_{k_0}(T)}
	&\le \|e^{t\Delta} \varphi \|_{\tilde{\mathcal{K}}^{p_0}_{k_0}(T)} 
	+ \|N(u)\|_{\tilde{\mathcal{K}}^{p_0}_{k_0}(T)} 
	\le  \|e^{t\Delta} \varphi \|_{\tilde{\mathcal{K}}^{p_0}_{k_0}(T)}  
		+C_3 \|u\|_{\mathcal{K}^{p}_{k}(T)}^{\alpha-1}
		\|u\|_{\tilde{\mathcal{K}}^{p_0}_{k_0}(T)} \\
	&\le  \|e^{t\Delta} \varphi \|_{\tilde{\mathcal{K}}^{p_0}_{k_0}(T)}  
		+C_3 M^{\alpha-1}
		\|u\|_{\tilde{\mathcal{K}}^{p_0}_{k_0}(T)}. 
\end{aligned}
\end{equation}
Thus, if $M$ is smaller so that $2 C_3 M^{\alpha-1}<1$, then 
$\|u\|_{\tilde{\mathcal{K}}^{p_0}_{k_0}(T)}\le 2 \|e^{t\Delta} \varphi \|_{\tilde{\mathcal{K}}^{p_0}_{k_0}(T)}.$ 
Therefore, there exists a fixed point $\tilde u\in \tilde X_M$ for the map $\mathscr{F}$. 
Since $\tilde X_M$ is a subspace of $X_M$, $\tilde u$ coincides with $u$ obtained before 
thanks to the uniqueness of the fixed point. 
Lastly, \eqref{l:exist.crt':est2}, which is a new result on the upgraded regularity, 
follows by iterative argument as in the proof of Lemma \ref{l:exist.crt} (iii), so we omit the detail. 
We note that it is easy to check that an assertion similar to Lemma \ref{l:b.strap} holds 
with $\frac1{q_c(\gamma)}$ replaced by $\frac{\sigma}{d}$.

	\smallbreak
The continuous dependence \eqref{l:exist.crt:cont.depnd} with $(\tilde k,\tilde p)=(k,p)$ 
of the solution on the initial data is deduced by setting in \eqref{t:HH.LWP.pr.Lip} 
$\mathscr{F}_\varphi(u)=u$ and $\mathscr{F}_\varphi(v)=v$. 
Formula \eqref{l:exist.crt:cont.depnd} for all $(\tilde k,\tilde p)$ 
satisfying \eqref{t:asym.behv:c.kp} can be proved 
by using an iterative procedure similar to the proof of (iv). In particular, 
one can prove a version of Lemma \ref{l:b.strap} with $u$ replaced by $u-v$.

	\smallbreak
To prove the stronger decay estimate \eqref{l:exist.crt:strng.dcy.est}, we observe that the previous calculations yield 
\begin{align*}
	&\|u(t)-v(t)\|_{L^{p,\infty}_k} 
		\leq \|e^{t \Delta}(\varphi-\psi)\|_{L^{p,\infty}_k} \\
	& \quad + 2 C_0 M^{\alpha-1} 
		\int_0^t 
		(t-\tau)^{-\frac{d (\alpha-1)}{2 p}-\frac{k(\alpha-1)-\gamma}{2}} 
		\tau^{- \frac{d (\alpha-1)}2(\frac1{q_c(\gamma)} - \frac{k}{d} - \frac1{p} )} 
		\| u(\tau)- v(\tau)\|_{L^{p,\infty}_k} \, d\tau. 
\end{align*}
Let $\delta>0$ be such that 
$\frac{d}2(\frac1{q_c(\gamma)} - \frac{k}{d} - \frac1{p} ) \alpha+\delta<1.$ 
For arbitrary $T>0,$ we have 
\begin{align*}
	&\int_0^t 
		(t-\tau)^{-\frac{d (\alpha-1)}{2 p}-\frac{k(\alpha-1)-\gamma}{2}} 
		\tau^{- \frac{d (\alpha-1)}2(\frac1{q_c(\gamma)} - \frac{k}{d} - \frac1{p} )} 
		\| u(\tau)- v(\tau)\|_{L^{p,\infty}_k} \, d\tau \\
	&\leq t^{-\frac{d}2(\frac1{q_c(\gamma)} - \frac{k}{d} - \frac1{p} )-\delta} 
		\int_0^1 (1-\tau)^{-\frac{d(\alpha-1)}{2 p} -\frac{k(\alpha-1)-\gamma}{2}}
		\tau^{-\frac{d\alpha}2(\frac1{q_c(\gamma)} - \frac{k}{d} - \frac1{p} )-\delta} \,d\tau\\
	& \qquad \times \sup_{0<t\leq T} 
			t^{\frac{d}2(\frac1{q_c(\gamma)} - \frac{k}{d} - \frac1{p} )+\delta} 
			\| u(t)- v(t)\|_{L^{p,\infty}_k}. 
\end{align*}
Now the above estimate and \eqref{l:exist.crt:strng.dcy.c} yield 
\begin{align*}
	&t^{ \frac{d}2(\frac1{q_c(\gamma)} - \frac{k}{d} - \frac1{p} )+\delta} 
		\|u(t)-v(t)\|_{L^{p,\infty}_k} 
	\leq t^{\frac{d}2(\frac1{q_c(\gamma)} - \frac{k}{d} - \frac1{p} )+\delta} 
		\|e^{t \Delta}(\varphi-\psi)\|_{L^{p,\infty}_k} \\
	& \quad + 2 C_0 M^{\alpha-1} 
		\int_0^1 (1-\tau)^{-\frac{d(\alpha-1)}{2 p} -\frac{k(\alpha-1)-\gamma}{2}}
		\tau^{-\frac{d\alpha}2(\frac1{q_c(\gamma)} - \frac{k}{d} - \frac1{p} )-\delta} \,d\tau\\
	& \qquad \times \sup_{0<t\leq T} 
			t^{\frac{d}2(\frac1{q_c(\gamma)} - \frac{k}{d} - \frac1{p} )+\delta} 
			\| u(t)- v(t)\|_{L^{p,\infty}_k}
\end{align*}
for all $t \in (0,T]$. Note that the constant $C_0$ does not depend on $T$ and finite by the hypothesis on $\delta$, 
and that the above argument works for arbitrary finite time $T>0$. 
The estimate \eqref{l:exist.crt:strng.dcy.est} for $(\tilde k, \tilde p) = (k,p)$ thus follows by taking $M$ smaller if necessary. 

We now prove \eqref{l:exist.crt:strng.dcy.est} for all $(\tilde k, \tilde p)$ satisfying 
\eqref{t:asym.behv:c.kp}. Let $T>0$ be an arbitrary real number. We write
\begin{equation}	\label{l:exist.crt:strng.dcy:pr1}
\begin{aligned}
	u(t)-v(t) 
	&= e^{\frac{t}{2}\Delta}\left(u \left(\frac{t}{2}\right)-v\left(\frac{t}{2}\right) \right) \\
	&\quad 
		+ a \int_\frac{t}{2}^t e^{(t-\tau)\Delta}
			\big[|\cdot|^{\gamma} \left(|u(\tau)|^{\alpha-1} u(\tau) 
				- |v(\tau)|^{\alpha-1} v(\tau)\right)\big] \,d\tau.
\end{aligned}
\end{equation}
Lemma \ref{l:linear-main} yields
\begin{equation}\nonumber
	t^{-\frac{d}2({\tilde k\over d}+ {1\over \tilde p})}
	\left\|e^{\frac{t}{2}\Delta}
			\left(u\left(\frac{t}{2}\right)-v\left(\frac{t}{2}\right) \right)\right\|_{L^{\tilde p,\infty}_{\tilde k}} 
	\le C t^{-\frac{d}2({k\over d}+{1\over p})}
	\left\| u\left(\frac{t}{2}\right)-v\left(\frac{t}{2}\right)\right\|_{L^{p,\infty}_{k}} 
\end{equation}
and 
\begin{align*}
	&t^{\frac{d}2(\frac1{q_c(\gamma)}-{\tilde k\over d} - {1\over \tilde p})+\delta}
		\|u(t)-v(t)\|_{L^{\tilde p,\infty}_{\tilde k}} 
	\leq C t^{\frac{d}{2} (\frac{1}{q_c(\gamma)} - \frac{k}{d} - \frac{1}{p})+\delta} 
		\left\| u\left(\frac{t}{2}\right)-v\left(\frac{t}{2}\right) \right\|_{L^{p,\infty}_k}\\ 
	&\qquad + Ct^{\frac{d}2(\frac1{q_c(\gamma)}-{\tilde k\over d}-{1\over\tilde p})+\delta} \\
	&\qquad\times
		\int_\frac{t}{2}^t (t-\tau)^{-\frac{d}2\frac{\alpha-1}{\tilde p}
							-\frac{(\alpha-1) \tilde k-\gamma}{2}}
		\big( \|u(\tau)\|^{\alpha-1}_{L^{\tilde p,\infty}_{\tilde k}}
			+\|v(\tau)\|^{\alpha-1}_{L^{\tilde p,\infty}_{\tilde k}}\big)
		\|u(\tau)-v(\tau)\|_{L^{\tilde p,\infty}_{\tilde k}} \, d\tau.
\end{align*}
We have also 
\begin{equation}\nonumber
	\|u(t)\|_{L^{\tilde p,\infty}_{\tilde k}} 
	\leq C(M) t^{-\frac{d}{2} (\frac{1}{q_c(\gamma)} - \frac{\tilde k}{d} - \frac{1}{\tilde p})} 
			\quad\text{and}\quad 
	\|v(t)\|_{L^{\tilde p,\infty}_{\tilde k}}
	\leq C(M) t^{-\frac{d}{2} (\frac{1}{q_c(\gamma)} - \frac{\tilde k}{d} - \frac{1}{\tilde p})} 
\end{equation}
thanks to the additional regularity $(\mathrm{iv})$, where $C(M)$ is a positive constant 
satisfying $C(M)\to 0$ as $M\to0$. 
Using \eqref{l:exist.crt:strng.dcy.est} with $(\tilde k, \tilde p) = (k,p)$ 
and the above fact to estimate the first term and the last term, respectively, we deduce
\begin{align*}
	&t^{\frac{d}2(\frac1{q_c(\gamma)}-{\tilde k\over d} - {1\over \tilde p})+\delta}
		\|u(t)-v(t)\|_{L^{\tilde p,\infty}_{\tilde k}} 
	\leq C \sup_{t\in (0,T]} 
		t^{\frac{d}2(\frac1{q_c(\gamma)} - \frac{k}{d} - \frac1{p} )+\delta} 
		\|e^{t\Delta}(\varphi-\psi)\|_{L^{p,\infty}_{k}} \\
	&\quad
	+2C(M)^{\alpha-1} 
		\int_\frac{1}{2}^1(1-\tau)^{-\frac{d}2\frac{\alpha-1}{\tilde p}-\frac{(\alpha-1)\tilde k -\gamma}{2}} 
		\tau^{-\frac{d\alpha}2(\frac1{q_c(\gamma)}-{\tilde k\over d}-{1\over\tilde p})-\delta} \,d\tau\\
	&\hspace{+2cm} 
		\times \sup_{t\in (0,T]}
		t^{\frac{d}2(\frac1{q_c(\gamma)}-{\tilde k\over d}-{1\over \tilde p})+\delta}
		\|u(t)-v(t)\|_{L^{\tilde p,\infty}_{\tilde k}}.
\end{align*}
This completes the proof of Lemma \ref{l:exist.crt}. 
\end{proof}

Now, we may give the proof of Theorem \ref{t:HH.LWP}.


\begin{proof}[Proof of Theorem \ref{t:HH.LWP}]

{\it Proof of $(\rm{i})$.} 
Let $\varphi \in L^{q,r}_{l}(\R^d)$ ($\varphi \in \mathcal{L}^{q,\infty}_{l}(\R^d)$ if $r=\infty$). 
Then $\varphi$ satisfies the assumptions of Lemma \ref{l:exist.crt} with some $T,$ thanks to Lemma \ref{l:linear-main}. 
Combining \eqref{l:lin.est.pk} and \eqref{l:exist.crt.c.pk}, we obtain \eqref{t:HH.LWP.c.pk}. 
Thus, under \eqref{d:critical}, \eqref{t:HH.LWP.c.paramet}, \eqref{t:HH.LWP.c.ql} and \eqref{t:HH.LWP.c.pk}, 
we may construct an $L^{q,r}_{l}(\R^d)$-mild solution. 
Indeed, since $C_0^\infty(\R^d)$ is dense in $L^{q,r}_{l}(\R^d)$, there exists some positive number $T$ that is 
small enough so that $\|e^{t\Delta} \varphi\|_{\mathcal{K}^{p}_{k}(T)}\le \rho.$ 
Now for this $T$, Lemma \ref{l:exist.crt} asserts that 
\begin{equation}\nonumber
	\|u\|_{L^\infty(0,T ; L^{q,r}_{l})} 
	\le \|e^{t\Delta} \varphi\|_{L^\infty(0,T; L^{q,r}_{l})} 
		+ C_2 \|u \|_{\mathcal{K}^{p}_{k}(T)}^{\alpha}
	\le \|\varphi\|_{L^{q,r}_{l}} + C_2 M^{\alpha},
\end{equation}
where $M$ is as in \eqref{l:exist.crt:est}. 
The time-continuity at $t=0$ follows from a classical argument. 
Thus, $u$ is an $L^{q,r}_{l}(\R^d)$-mild solution to \eqref{HHNLH} on $[0,T]$ such that $\|u\|_{\mathcal{K}^{p}_{k}(T)}\le M.$ 
To deduce the estimate \eqref{t:HH.LWP.est}, it suffices to take 
$\rho=\|e^{t\Delta} \varphi\|_{\mathcal{K}^{p}_{k}(T)}$ and $M$ appropriately. 

\smallbreak
{\it Proof of $(\rm{ii})$.} 
For a fixed time $T>0$, the uniqueness in $\mathcal{K}^{p}_{k}(T)$ 
can be proved in the same method as in \cite{CIT2021} (See also \cite{Tay2020}). We omit the technical details here. 

Given $\varphi \in L^{q,r}_{l}(\R^d),$ let the maximal existence time $T_m = T_m (\varphi)$ be defined by \eqref{d:Tm}. 
By a standard argument, uniqueness ensures that the solution can be extended to the maximal interval $[0,T_m).$ 

\smallbreak
{\it Proof of $(\rm{iii})$.}
Given two initial data $\varphi, \psi \in L^{q,r}_{l}(\R^d),$ 
let $u$ and $v$ be two solutions associated with the initial data $\varphi$ and $\psi,$ 
respectively, constructed in (i) with the estimate 
$\|u\|_{\mathcal{K}^{p}_{k}(T)}\le 2\|e^{t\Delta} \varphi\|_{\mathcal{K}^{p}_{k}(T)}.$ 
We may show the Lipschitz continuity of the flow map 
by a similar calculation leading to \eqref{t:HH.LWP.pr.Lip}. 
We deduce that there exists some positive constant $C$ such that 
\begin{align*}
	\|u-v\|_{L^\infty(0,T; L^{q,r}_{l})\cap \mathcal{K}^{p}_{k}(T)} 
	\le \|\varphi-\psi\|_{L^{q,r}_{l}} 
		+ 2 C M^{\alpha-1} \|u-v\|_{\mathcal{K}^{p}_{k}(T)}, 
\end{align*}
where $M$ is the same as in the proof of Lemma \ref{l:exist.crt}. 
By taking $M$ smaller if necessary, we deduce the Lipschitz continuity on the short time-interval $[0,T].$ 
Extension to the whole interval is carried out by a classical argument. 

\smallbreak
{\it Proof of $(\rm{iv})$.}
Once the uniqueness in $\mathcal{K}^{p}_{k}(T)$ is established, we may 
prove the blow-up criterion by a contradiction argument as in \cite{CIT2021}. 
As the argument is similar, we omit the detail. 

{\it Proof of $(\rm{v})$.}
We prove the additional regularity. By Lemma \ref{l:exist.crt} (v), 
we have $\displaystyle \|u\|_{\mathcal{K}^{p_1}_{k_1}(T)} < \infty$ 
for all $(k_1, p_1) \in \R \times [1,\infty]$ satisfying \eqref{t:asym.behv:c.kp}: 
\begin{equation}	\nonumber
	k_1 \le k, \quad
	p \le p_1 
		\quad\text{and}\quad
	0\le \frac{k_1}{d} + \frac1{p_1}.
\end{equation}
Fix pairs $(k_1, p_1)$ and $(l,q)$ satisfying \eqref{t:asym.behv:c.kp} and \eqref{d:critical}, respectively. 
By interpolation, we have 
\begin{align*}
	\|u(t)\|_{L^{p_{2}, 1}_{k_{2}}}
	&\le C \|u(t)\|_{L^{p_1, 1}_{k_1} }^{\theta}
		\|u(t)\|_{L^{q, r}_{l} }^{1-\theta},
\end{align*}
where 
\begin{equation}\nonumber
	\theta\in (0,1), \quad 
	\theta k_1 + (1-\theta) l = k_2
		\quad\text{and}\quad
	\frac{\theta}{p_1} + \frac{1-\theta}{q} = \frac1{p_2}. 
\end{equation}
This implies
\begin{equation}\nonumber
	\sup_{0<t <T} 
	t^{\frac{d}{2} (\frac1{q_c(\gamma)} - \frac{k_{2}}{d} - \frac1{p_{2}} )} 
	\|u(t)\|_{L^{p_{2},1}_{k_{2}}} < \infty
\end{equation}
for all $k_2 \le l$ and $q\le p_2 \le\infty$ such that $0\le \frac{k_2}{d} + \frac1{p_2}.$ 
\end{proof}

\section{Global existence, self-similar solutions and asymptotic stability}

In this section we prove Theorems \ref{t:globalsol}, \ref{t:HH.self.sim} and \ref{t:stab.asym.behv}. 
We first prove the global existence of solutions.

\begin{proof}[Proof of Theorem \ref{t:globalsol}]

Assertion (i) of Theorem \ref{t:globalsol} follows from Lemma \ref{l:exist.crt} and the smallness of 
the quantity $\|e^{t\Delta}\varphi\|_{\mathcal{K}^p_k}$. 
The dissipation also follows from the standard argument (See \cite{CIT2022} for example). 
If $\varphi\in L^{q,r}_l(\R^d)$ and $r<\infty$, then the continuity at $t=0$ follows from the classical density argument. 
\end{proof}

\begin{proof}[Proof of Theorem \ref{t:HH.self.sim}]

Let $\Phi(x) = \omega(x)  |x|^{-\frac{2+\gamma}{\alpha-1}}$ 
be as in the assumption of Theorem \ref{t:HH.self.sim}. 
Then we note that $\Phi$ is homogeneous of degree $-\frac{2+\gamma}{\alpha-1}.$ 
We show that the global solution $u$ to \eqref{HHNLH} with the initial data $\Phi$, 
which is obtained by (v) in Theorem \ref{t:HH.LWP}, is also self-similar. 
To this end, for $\lambda>0$, let $\Phi_{\lambda}$ be defined by 
$\Phi_{\lambda} (x) := \lambda^{\frac{2+\gamma}{\alpha-1}} \Phi(\lambda x).$  
Since the identity $\|\Phi_{\lambda}\|_{\mathcal{K}^{p}_{k}} =\|\Phi \|_{\mathcal{K}^{p}_{k}}$ 
holds for all $\lambda>0,$ it follows that $\Phi_{\lambda}$ also 
satisfies the assumptions of (v) in Theorem \ref{t:HH.LWP}. 
As $u_\lambda$ given by \eqref{scale} is a solution of \eqref{HHNLH} with initial data 
$\Phi_{\lambda},$ and $\|u_{\lambda}\|_{\mathcal{K}^{p}_{k}} = \|u \|_{\mathcal{K}^{p}_{k}}$ 
for all $\lambda>0,$ we deduce that $u$ must be self-similar since 
$\Phi_{\lambda}=\Phi$. We denote the global self-similar solution $u$ 
by $u_{\mathcal{S}}$. The fact $u_{\mathcal{S}}(t)\rightarrow\Phi$ in 
$\mathcal{S}'(\R^d)$ as $t\rightarrow +0$ follows from (iii) in 
Lemma \ref{l:exist.crt}. This completes the proof of Theorem \ref{t:HH.self.sim}. 
\end{proof}

\begin{proof}[Proof of Theorem \ref{t:stab.asym.behv}]
Let $(l,q)$ and $(k,p)$ satisfy \eqref{t:HH.LWP.c.ql} and \eqref{t:HH.LWP.c.pk}, 
respectively. Suppose that $u$ and $v$ are two global $L^{q,r}_{l}(\R^d)$-mild solutions to \eqref{intHHNLH} 
belonging to $C([0,\infty); L^{q,r}_{l} (\R^d)) \cap \mathcal{K}^{p}_{k}$
with initial data $\varphi$ and $\psi\in L^{q,r}_{l}(\R^d)$, respectively, 
such that $\max\{\|u\|_{\mathcal{K}^{p}_{k}}, \|v\|_{\mathcal{K}^{p}_{k}} \} \le M$ for some $M>0$. 
Then \eqref{t:HH.LWP.pr.Lip} gives 
\begin{align*}
	\|u-v\|_{\mathcal{K}^{p}_{k}}
	&\le \|e^{t\Delta} (\varphi-\psi) \|_{\mathcal{K}^{p}_{k}} 
		+ 2 C_0 M^{\alpha-1} \|u-v\|_{\mathcal{K}^{p}_{k}} \\
	&\le C \|e^{\frac{t}2\Delta} (\varphi-\psi) \|_{L^{q,r}_{l}} 
		+ 2 C_0 M^{\alpha-1} \|u-v\|_{\mathcal{K}^{p}_{k}}. 
\end{align*}
The above implies that, by taking $M$ smaller if necessary, 
there exists some constant $\tilde C_0$ such that  
\begin{equation}	\label{t:stab.asym.behv:pr1}
	\|u-v\|_{\mathcal{K}^{p}_{k}}
	=\sup_{t>0} t^{\frac{d}2(\frac1{q_c(\gamma)} - \frac{k}2 -\frac{1}{p})}
		\|u(t)-v(t)\|_{L^{p, \infty}_{k}}
		\le \tilde C_0 \|e^{\frac{t}2\Delta} (\varphi -\psi)\|_{L^{q,r}_l}.
\end{equation}
In particular, we have 
\begin{equation}\nonumber
	 t^{\frac{d}2(\frac1{q_c(\gamma)} - \frac{k}2 -\frac{1}{p})}
		\|u(t)-v(t)\|_{L^{p, \infty}_{k}} \to 0 \quad\text{as}\quad t\to\infty
\end{equation}
if $\displaystyle \lim_{t\to\infty} \|e^{t\Delta} (\varphi -\psi)\|_{L^{q,r}_l} = 0$. 
Moreover \eqref{t:stab.asym.behv:pr1} along with \eqref{l:stb.est:a1} leads to 
\begin{align*}	\nonumber
	\|u(t)-v(t)\|_{L^{q,r}_{l}} 
		&\le \|e^{t\Delta} (\varphi -\psi)\|_{L^{q,r}_l} 
			+ C_1 \left(\|u\|_{\mathcal{K}^{p}_{k}}^{\alpha-1} + \|v\|_{\mathcal{K}^{p}_{k}}^{\alpha-1}\right)
				\|u-v\|_{\mathcal{K}^{p}_{k}} \\
		&\le \|e^{t\Delta} (\varphi -\psi)\|_{L^{q,r}_l} 
			+ 2 C_1M^{\alpha-1} \|u-v\|_{\mathcal{K}^{p}_{k}} \\
		&\le \|e^{t\Delta} (\varphi -\psi)\|_{L^{q,r}_l} 
		+ 2 \tilde C_0 C_1M^{\alpha-1} \|e^{\frac{t}2\Delta} (\varphi-\psi) \|_{L^{q,r}_{l}}. 
\end{align*}
Therefore 
$\displaystyle \lim_{t\to\infty} \|e^{t\Delta} (\varphi -\psi)\|_{L^{q,r}_l} = 0$ implies 
\begin{equation}\nonumber
	\lim_{t\to\infty} \|u(t) -v(t)\|_{L^{q,r}_l} = 0, 
\end{equation}
provided that $M$ is small enough. 

To prove the converse implication, we resort to Lemma \ref{l:crt.est2}.
Let $k_1\in\R$ and $p_1\in[1,\infty]$ satisfy 
\begin{equation}	\label{t:stab.asym.behv:c1}
	\frac1{p_1} < \frac1{\alpha-1} \left(1-\frac1{q}\right) 
		\quad\text{and}\quad
	\frac{k_1}{d} + \frac1{p_1}
		< \frac{1}{\alpha-1} \left(1+\frac{\gamma}{d}-\frac1{q_c(\gamma)}\right).
\end{equation}
Note that by Lemma \ref{l:exist.crt} (iv), 
the solution belongs to $\mathcal{K}^{p_1}_{k_1}$. Writing 
\begin{equation}\nonumber
	e^{t\Delta} (\varphi-\psi) = u(t) -v(t) - \left\{ N(u)(t) - N(v)(t) \right\}
\end{equation}
and then using \eqref{l:crt.est2:a1}, 
we obtain
\begin{equation}\nonumber
\begin{aligned}
	\limsup_{t\to\infty}&\|e^{t\Delta} (\varphi -\psi)\|_{L^{q,r}_l} 
	\le \limsup_{t\to\infty} \|u(t) - v(t)\|_{L^{q,r}_l} \\
	&+ C_2 \left(\|u\|_{\mathcal{K}^{p_1}_{k_1}}^{\alpha-1}
		+ \|v\|_{\mathcal{K}^{p_1}_{k_1}}^{\alpha-1}\right)
		\limsup_{t\to\infty}\|u(t)-v(t)\|_{L^{q,r}_{l}},
\end{aligned}
\end{equation}
thanks to \eqref{t:stab.asym.behv:c1}. Thus, the assumption 
$\displaystyle \lim_{t\to\infty}\|u(t)-v(t)\|_{L^{q,r}_{l}}=0$ gives the converse relation. 
Note that for the above implication, we do not need the smallness of $\|u\|_{\mathcal{K}^{p_1}_{k_1}}$ and 
$\|v\|_{\mathcal{K}^{p_1}_{k_1}}$. 
The proof of Theorem \ref{t:stab.asym.behv} is complete. 
\end{proof}

\section{Asymptotic behaviors: Real-valued case}
In this section we give the proof of Theorem \ref{t:asym.behv}. 
Assertion (i) is a direct consequence of Lemma \ref{l:exist.crt}. 
Indeed, from the latter part of Lemma \ref{l:exist.crt} and, in particular, formula \eqref{l:exist.crt:strng.dcy.est}, we have
\begin{align*}
	\sup_{t>0} t^{\frac{d}{2} (\frac{1}{q_c(\gamma)} 
				-\frac{\tilde k}{d} - \frac{1}{\tilde p})+\delta} 
		\|u(t)-u_{\mathcal S}(t)\|_{L^{\tilde p,\infty}_{\tilde k}}
	&\leq C \sup_{t>0}
		t^{\frac{d}{2} (\frac{1}{q_c(\gamma)} - \frac{k}{d} - \frac{1}{p})+\delta} 
		\|e^{t\Delta}(\Phi_\sigma-\varphi)\|_{L^{p,\infty}_k}
\end{align*}
with $\sigma = \frac{2+\gamma}{\alpha-1}$ for $(\tilde k,\tilde p)$ satisfying \eqref{t:asym.behv:c.kp} and for $\delta>0$ sufficiently small. 
This gives \eqref{t:asym.behv.NL:est1} directly. 
Estimate \eqref{t:asym.behv.NL:est2} follows from \eqref{t:asym.behv.NL:est1} by a simple dilation argument. 

In what follows, we focus on the linear case, that is, assertion (ii) of Theorem \ref{t:asym.behv}. 

\subsection{Proof of Theorem \ref{t:asym.behv} (ii). }

Here, we state a more general result than Theorem \ref{t:asym.behv} (ii). 
Substituting $k_0 = k$ and $p_0=p$ below, we obtain Theorem \ref{t:asym.behv} (ii). 

\begin{thm}[Linear behavior]
	\label{t:asym.behv.lin}
Let $d\in\N$, $\gamma\in\R$ and $\alpha\in\R$ satisfy \eqref{t:HH.LWP.c.paramet}. 
Let $\Phi_\sigma(x)=\omega(x)|x|^{-\sigma}$ and let 
\begin{equation}\nonumber
	e^{t\Delta} \Phi_\sigma
	= t^{-\frac{\sigma}{2}} (e^{\Delta} \Phi_\sigma) \left(\frac{x}{\sqrt{t}} \right) 
\end{equation}
be the self-similar solution of the linear heat equation with initial data $\Phi_\sigma$, 
where $\sigma$ satisfies 
\begin{equation}	\label{t:asym.behv.lin_c.sigma}
	\frac{d}{q_c(\gamma)} < \sigma < d. 
\end{equation}
Let $\varphi$ be nontrivial and satisfy the assumptions of Theorem \ref{t:globalsol} and 
\eqref{l:exist.crt:c.sigma}. 
Let $u \in \mathcal K^{p}_{k} \cap \tilde{\mathcal{K}}^{p_0}_{k_0}$ 
be the global solution of \eqref{intHHNLH} with initial data 
$\varphi$ constructed by Lemma \ref{l:exist.crt} with $T=\infty$, 
where $(k,p)$ satisfies \eqref{t:HH.LWP.c.pk}, and 
$(k_0, p_0)$ satisfies \eqref{l:exist.crt.c.p0k0}. 
Then there exists a constant $C>0$ such that
\begin{equation}	
	\|u(t)-e^{t\Delta} \Phi_\sigma \|_{L^{\tilde p, \infty}_{\tilde k}}
	\leq C \, t^{-\frac{d}2(\frac{\sigma}{d} - \frac{ \tilde k}{d}-\frac{1}{\tilde p}) -\delta}
\end{equation}
and 
\begin{equation}	
	\|t^{\frac{\sigma}{2}} u(t, \sqrt t \cdot)-e^{\Delta} \Phi_\sigma\|_{L^{\tilde p, \infty}_{\tilde k}}
	\leq C t^{-\delta}
\end{equation}
for all $t>0$ and $(\tilde k, \tilde p) \in \R \times [1,\infty]$ satisfying \eqref{l:exist.crt':c.tild.kp}:
	\begin{equation}	\nonumber
		\tilde k \le k_0, \quad
		p_0 \le \tilde p
			\quad\text{and}\quad
		0\le \frac{\tilde k}{d} + \frac1{\tilde p} \,.
	\end{equation}
In particular, there exists a constant $C>0$ such that
\begin{equation*}
	C^{-1} t^{-\frac{d}{2} (\frac{\sigma}{d}-\frac{\tilde k}{d}-\frac1{\tilde p}) } 
	\leq \|u(t)\|_{L^{\tilde p,\infty}_{\tilde k}} 
	\leq C t^{-\frac{d}{2} (\frac{\sigma}{d}-\frac{\tilde k}{d}-\frac1{\tilde p})} 
\end{equation*}
for large time and for all $(\tilde k, \tilde p) \in \R \times [1,\infty]$ satisfying \eqref{l:exist.crt':c.tild.kp}. 
\end{thm}

\begin{proof} 
Note that by assumption, $\varphi$ generates a global solution. 
We are now in a position to prove \eqref{t:asym.behv.lin:c1} and \eqref{t:asym.behv.lin:c2}. 
We first prove \eqref{t:asym.behv.lin:c1} for $(\tilde k, \tilde p) = (k_0, p_0)$. 

Starting from \eqref{intHHNLH}, we may write  
\begin{equation}	\label{t:asym.behv.lin:pr.integral.eq}
	u(t) - e^{t\Delta} \Phi_\sigma = e^{t\Delta} (\varphi-\Phi_\sigma) + a N(u)(t), 
\end{equation}
where $N(u)(t)$ is defined by \eqref{mapN}, so it suffices to estimate $N(u)$ in $L^{p_0,\infty}_{k_0}(\R^d)$. Let 
\begin{equation}\nonumber
	\frac{1}{p_\theta} = \frac{\theta}{p} + \frac{1-\theta}{p_0}
		\quad\text{and}\quad
	k_\theta = \theta k + (1-\theta) k_0,
		\quad\text{where}\quad 
	0 < \theta <1. 
\end{equation}
We choose $\theta$ appropriately later. By Lemma \ref{l:linear-main}, there exists a constant $C>0$ such that 
\begin{equation}	\label{t:asym.behv.lin:pr.key.est}
\begin{aligned}
	\|&N(u)(t) \|_{L^{p_0,\infty}_{k_0}} 
	\leq C \int_0^t 
			(t-\tau)^{ -\frac{d}{2} (\frac{\alpha}{p_\theta}- \frac{1}{p_0}) - \frac{\alpha k_\theta - \gamma - k_0} {2} } 
			\| u(\tau) \|_{L^{p_\theta, \infty}_{k_\theta}}^{\alpha} 
		\, d\tau \\
	&\leq C t^{-\frac{d}{2} \left( \frac{\sigma}{d} - \frac{k_0}{d} - \frac{1}{p_0} \right) - \delta} 
			\|u\|_{\mathcal{K}^p_k}^{\theta \alpha}
			\left( \sup_{t>0} 
				\tau^{\frac{d}{2} (\frac{\sigma}{d}-\frac{k_0}{d}-\frac1{p_0} )} 
				\| u(\tau) \|_{L^{p_0, \infty}_{k_0}}
			\right)^{(1-\theta)\alpha} \\
	&\hspace{-5mm} \times   \int_0^1 
			(1-\tau)^{ -\frac{d}{2} (\frac{\alpha}{p_\theta}- \frac{1}{p_0}) - \frac{\alpha k_\theta - \gamma - k_0} {2} } 
			\tau^{-\theta \alpha \frac{d}{2} (\frac{1}{q_c(\gamma)}-\frac{k}{d}-\frac1{p} ) 
			-(1-\theta)\alpha \frac{d}{2} (\frac{\sigma}{d}-\frac{k_0}{d}-\frac1{p_0} )} \, d\tau, 
\end{aligned}
\end{equation}
where 
\begin{equation}	\label{t:asym.behv.lin:pr.key.est.c0}
	\delta=\frac{d}{2} \left( \frac{\sigma}{d} - \frac{1}{q_c(\gamma)} \right) ( \alpha-1 - \theta\alpha)>0
	\quad\text{with}\quad 0\le \theta < 1-\frac1{\alpha}, 
\end{equation}
provided that 
\begin{equation}	\label{t:asym.behv.lin:pr.key.est.c1}
\begin{aligned}
	&\gamma - (\alpha-1)k_0 \le \theta \alpha ( k - k_0 ), \quad
		\theta \left( \frac{1}{p} -\frac{1}{p_0} \right) < \frac1{\alpha} -  \frac{1}{p_0}, \quad
		0 \le \frac{k_0}{d} + \frac{1}{p_0} \quad\text{and}\\
	& \frac{\gamma}{d} - (\alpha-1) \left(\frac{k_0}{d} + \frac1{p_0}\right)
		\le \theta \alpha \left(\frac{k}{d} + \frac1{p} - \frac{k_0}{d} - \frac1{p_0}\right) 
		 < 1 + \frac{\gamma}{d} - \alpha \left(\frac{k_0}{d} + \frac1{p_0}\right). 
\end{aligned}
\end{equation}
The last time-integral in \eqref{t:asym.behv.lin:pr.key.est} is convergent if and only if 
\begin{equation}	\label{t:asym.behv.lin:pr.key.est.c2}
\begin{aligned}
	&	(\alpha-1) \left( \frac1{q_c(\gamma)} - \frac{k_0}{d} - \frac1{p_0} \right)
	- \theta \alpha \left( \frac{k}{d} + \frac1{p} - \frac{k_0}{d} - \frac1{p_0}\right)>0 \\
		\text{and}\quad
	&\frac{2}{d\alpha} - \frac{\sigma}{d}+ \frac{k_0}{d} +\frac1{p_0}  
		+ \theta \left( \frac{\sigma}{d}-\frac{k_0}{d}-\frac1{p_0} - \left(\frac{1}{q_c(\gamma)}-\frac{k}{d}-\frac1{p} \right) \right)
		>0. 
\end{aligned}
\end{equation}
By Lemma \ref{l:asym.behv.lin}, we may choose $\theta$ satisfying \eqref{t:asym.behv:c1.phi} and \eqref{t:asym.behv.lin:pr.key.est.c0}. 
Thus, coming back to the estimate \eqref{t:asym.behv.lin:pr.key.est}, we have 
\begin{equation*}
	\sup_{t>0} t^{\frac{d}2(\frac{\sigma}{d}- \frac{k_0}{d} - \frac1{p_0} )+\delta}
		\|u(t) - e^{t\Delta} \Phi_\sigma \|_{L^{p_0,\infty}_{k_0}}  
	\leq C_0 + \tilde C M^{\alpha}. 
\end{equation*}
Since the constant 
\begin{align*}
	&\tilde C
	:= C   \int_0^1 
			(1-\tau)^{ -\frac{d}{2} (\frac{\alpha}{p_\theta}- \frac{1}{p_0}) - \frac{\alpha k_\theta - \gamma - k_0} {2} } 
			\tau^{-\theta \alpha \frac{d}{2} (\frac{1}{q_c(\gamma)}-\frac{k}{d}-\frac1{p} ) 
			-(1-\theta)\alpha \frac{d}{2} (\frac{\sigma}{d}-\frac{k_0}{d}-\frac1{p_0} )} \, d\tau 
\end{align*}
does not depend on $t$ and finite by the hypothesis on $\delta$, we deduce \eqref{t:asym.behv.lin:c1} for $(\tilde k, \tilde p) = (k_0,p_0)$.

We now prove \eqref{t:asym.behv.lin:c1}  for all $(\tilde k, \tilde p)$ satisfying \eqref{l:exist.crt':c.tild.kp}. 
Let $t>0$ be an arbitrary positive real number. Starting from \eqref{t:asym.behv.lin:pr.integral.eq}, we have 
\begin{equation}	\nonumber
\begin{aligned}
	u(t) - e^{t\Delta} \Phi_\sigma
		= e^{t\Delta} \varphi + a N(u)(t) - e^{t\Delta} \Phi_\sigma 
		= e^{\frac{t}2\Delta} \left(u\left(\frac{t}{2}\right) - e^{\frac{t}{2}\Delta} \Phi_\sigma \right) + a \tilde N(u)(t),
\end{aligned}
\end{equation}
where $\tilde N(u)(t)$ is defined in \eqref{l:b.strap:int.rep}. 
Note that Lemma \ref{l:linear-main} 
gives 
\begin{equation}\nonumber
	t^{-\frac{d}2({\tilde k\over d}+ {1\over \tilde p})}
	\left\|e^{\frac{t}{2}\Delta}
		\left(u \left(\frac{t}2\right)-e^{\frac{t}{2}\Delta} \Phi_\sigma \right)\right\|_{L^{\tilde p,\infty}_{\tilde k}} 
	\le t^{-\frac{d}2({k_0 \over d}+{1\over p_0})}
	\left\| u\left(\frac{t}2\right)-e^{\frac{t}{2}\Delta} \Phi_\sigma \right\|_{L^{p_0,\infty}_{k_0}}. 
\end{equation}
By using Lemma \ref{l:linear-main} 
to \eqref{l:exist.crt:strng.dcy:pr1}, we deduce
\begin{align*}
	&t^{\frac{d}2(\frac{\sigma}{d} -{\tilde k\over d}  - {1\over \tilde p})+\delta}
		\|\tilde N(u)(t) \|_{L^{\tilde p,\infty}_{\tilde k}} 
	\leq Ct^{\frac{d}2(\frac{\sigma}{d}-{\tilde k\over d}-{1\over\tilde p})+\delta} 
		\int_\frac{t}{2}^t (t-\tau)^{-\frac{d}2\frac{\alpha-1}{\tilde p}
							-\frac{(\alpha-1) \tilde k-\gamma}{2}}
		\|u(\tau)\|_{L^{\tilde p,\infty}_{\tilde k}}^{\alpha} \, d\tau,
\end{align*}
where $\delta$ is to be chosen. By Lemma \ref{l:exist.crt}, we have also 
\begin{equation}\nonumber
	\|u(t)\|_{L^{\tilde p,\infty}_{\tilde k}} 
	\leq Ct^{-\frac{d}{2} (\frac{\sigma}{d} - \frac{\tilde k}{d} - \frac{1}{\tilde p})} 
		\quad\text{and}\quad
	\|u(t)\|_{L^{\tilde p,\infty}_{\tilde k}} 
	\leq Ct^{-\frac{d}{2} (\frac{1}{q_c(\gamma)} - \frac{\tilde k}{d} - \frac{1}{\tilde p})} 
\end{equation}
for $(\tilde k, \tilde p)$ such that \eqref{l:exist.crt':c.tild.kp}, 
thanks to the additional regularity (Lemma \ref{l:exist.crt} $(\mathrm{v})$). 
Thus, by interpolation, we have 
\begin{equation}\nonumber
\begin{aligned}
	\|u(t)\|_{L^{\tilde p,\infty}_{\tilde k}} 
	&= \|u(t)\|_{L^{\tilde p,\infty}_{\tilde k}}^{\theta} \|u(t)\|_{L^{\tilde p,\infty}_{\tilde k}}^{1-\theta}
	\leq C t^{-\theta\frac{d}{2} (\frac{\sigma}{d} - \frac{\tilde k}{d} - \frac{1}{\tilde p})
			-(1-\theta)\frac{d}{2} (\frac{1}{q_c(\gamma)} - \frac{\tilde k}{d} - \frac{1}{\tilde p})} \\
	&= C t^{-\frac{d}{2} (\theta \frac{\sigma}{d} + (1-\theta)\frac{1}{q_c(\gamma)} - \frac{\tilde k}{d} - \frac{1}{\tilde p})}
	= C t^{-\frac{d}{2} (\frac{\sigma}{d}  - \frac{\tilde k}{d} - \frac{1}{\tilde p})-\delta},
\end{aligned}
\end{equation}
where 
\begin{equation}\nonumber
	\delta:=(1-\theta) \frac{d}{2}\left(\frac{\sigma}{d} - \frac{1}{q_c(\gamma)} \right). 
\end{equation}
Using the above fact, we deduce
\begin{align*}
	&t^{\frac{d}2(\frac{\sigma}{d}-{\tilde k\over d} - {1\over \tilde p})+\delta}
		\|u(t) - e^{t\Delta} \Phi_\sigma \|_{L^{\tilde p,\infty}_{\tilde k}} 
	\leq C \sup_{t>0} 
		t^{\frac{d}2(\frac{\sigma}{d} - \frac{k_0}{d} - \frac1{p_0} )+\delta} 
		\left\|u\left(\frac{t}2\right)-e^{\frac{t}{2}\Delta} \Phi_\sigma \right\|_{L^{p_0,\infty}_{k_0}} \\
	&\quad
	+2C^{\alpha} 
		\int_\frac{1}{2}^1(1-\tau)^{-\frac{d}2\frac{\alpha-1}{\tilde p}-\frac{(\alpha-1)\tilde k -\gamma}{2}} 
		\tau^{-\frac{d\alpha}2(\frac{\sigma}{d}-{\tilde k\over d}-{1\over\tilde p})-\delta} \,d\tau. 
\end{align*}
That is, there exists a finite positive constant $C$ such that 
\begin{equation}	\nonumber
	\sup_{t>0}
		t^{\frac{d}2(\frac{\sigma}{d} - \frac{\tilde k}{d} - \frac1{\tilde p} ) +\delta}
		\|u(t)-e^{t\Delta} \Phi_\sigma\|_{L^{\tilde p,\infty}_{\tilde k}}
	\leq C
\end{equation}
for $(\tilde k,\tilde p)$ satisfying \eqref{l:exist.crt':c.tild.kp} and for $\delta>0$ sufficiently small.  
Finally, \eqref{t:asym.behv.lin:c2} follows from \eqref{t:asym.behv.lin:c1} by a simple dilation argument. 
\end{proof}

\section{Asymptotic behavior : Complex-valued case}\label{sec:complex-valued}

In this section we prove Theorem \ref{t:asym.behv-complex}, 
which is more general than Theorem \ref{t:asym.behv-complex0}. 
To this end, we investigate the target systems first. 

\begin{prop}
	\label{p:HH.sys.self.sim}
Let $d\in\mathbb{N},$ $\gamma\in\R$ and $\alpha\in\R$ satisfy \eqref{t:HH.LWP.c.paramet}. 
Let $l\in\R$ and $q\in [1,\infty]$ satisfy \eqref{d:critical} and \eqref{t:HH.LWP.c.ql}. 
For $j=1, 2$, let $\sigma_j$ be such that ${2+\gamma\over \alpha-1} < \sigma_j<d.$ 
Let $\omega_1 \in L^\infty(\R^d)$ be homogeneous of degree $0$ and $\|\omega_j\|_{L^\infty}$ is sufficiently small. 
Then the following assertions hold:

\begin{itemize}

\item[(i)] There exists a self-similar solution $(z_1, z_2)$ 
	where $z_1 \in \mathcal K^{p}_{k}$ and $z_2\in \tilde{\mathcal{K}}^{p_0}_{k_0}$. 
	We have $z_1=u_{1\mathcal S}$, where $u_{1\mathcal S}$ is 
	the global self-similar solution of \eqref{HHNLH} with initial data 
	$\omega_1(x) |x|^{-\frac{2+\gamma}{{\alpha-1}}}$ 
	constructed by Theorem \ref{t:HH.self.sim}. 

\item[(ii)] There exists a self-similar solution $(w_1, w_2)$ 
	where $w_2 \in \mathcal K^{p}_{k}$ and $w_1\in \tilde{\mathcal{K}}^{p_0}_{k_0}$. 
	We have $w_2=u_{2\mathcal S}$, where $u_{2\mathcal S}$ is 
	the global self-similar solution of \eqref{HHNLH} with initial data 
	$\omega_2(x) |x|^{-\frac{2+\gamma}{{\alpha-1}}}$ 
	constructed by Theorem \ref{t:HH.self.sim}. 

\end{itemize}
\end{prop}

\begin{proof}
We only prove assertion (i), since a similar argument holds for $(w_1,w_2).$ 
The proof of the existence of $z_1$ is due to fixed-point argument, 
similarly to the real case. By uniqueness, $z_1$ coincides with $u_{1\mathcal S}$, constructed by Theorem \ref{t:HH.self.sim}. 
Once we have $z_1$, the proof of the existence of $z_2$ follows by fixed point argument as well. Indeed, it is clear by Lemma \ref{l:Kato.est2}, the map 
\begin{equation}	\nonumber
\begin{aligned}
	\mathscr{F} [z_2](t) = e^{t\Delta} (\omega_2 |\cdot|^{-\sigma_2} ) 
	+ a \int_0^t e^{(t-\tau)\Delta} 
	\left( |\cdot|^{\gamma} [u_{1\mathcal{S}}(\tau)]^{\alpha-1} z_2(\tau)\right)\, d\tau
\end{aligned}
\end{equation}
satisfies 
\begin{equation}	\nonumber
\begin{aligned}
	\|\mathscr{F} [\zeta_1] -\mathscr{F} [\zeta_2] \|_{\tilde{\mathcal{K}}^{p_0}_{k_0}}
	&\le C_3 \|u_{1\mathcal{S}}\|_{\mathcal{K}^{p}_{k}(T)}^{\alpha-1} 
		\|\zeta_1 -\zeta_2\|_{\tilde{\mathcal{K}}^{p_0}_{k_0}}, 
\end{aligned}
\end{equation}
under $\frac{d}{q_c(\gamma)} < \sigma <d$, \eqref{l:exist.crt.c.pk} and \eqref{l:exist.crt.c.p0k0}. 
By standard fixed point argument, we have a solution $z_2 \in \tilde{\mathcal{K}}^{p_0}_{k_0}$, provided that 
$u_{1\mathcal{S}}$ is small. 
\end{proof}

\begin{thm}
	\label{t:asym.behv-complex}
Let $d\in\N$, $\alpha$ and $\gamma$ satisfy \eqref{t:HH.LWP.c.paramet}. 
Let $l\in\R$, $q\in [1,\infty]$ and $r \in [1,\infty]$ satisfy \eqref{d:critical} and \eqref{t:HH.LWP.c.ql}. 
Let $\sigma_j\in\R$, $j=1,2,$ satisfy \eqref{d:sigmaj}. 
Let $\varphi = \varphi_1 + i \varphi_2 \in \mathcal{S}'(\R^d)$ be 
nontrivial and satisfy the assumptions of Theorem \ref{t:globalsol}. 
Let $u\in \mathcal{K}^{p}_{k}$ be the global solution of \eqref{HHNLH} with initial data $\varphi$ 
constructed in Theorem \ref{t:globalsol}, where $k\in\R$ and $p\in[1,\infty]$ satisfy \eqref{t:HH.LWP.c.pk}. 
Let $(u_{1\mathcal S},z_2)$ be the self-similar solution of \eqref{HHNLH-syscase3}, and 
$(w_1,u_{2\mathcal S})$ be the self-similar solution of \eqref{HHNLH-syscase4}. 
Then we have the following assertions.

\begin{itemize}

	\item[(i)] $($Nonlinear-``Modified linear" behavior$)$
	Suppose that there exist positive constants $C_1$, $C_2$, $\delta_{01}$ and $\delta_{02}$ such that 
	\begin{equation}	\label{t:asym.behv-complex:c1}
	\begin{aligned}
		& t^{\frac{d}2(\frac{\sigma_1}{d} - \frac{k}{d} - \frac1{p})+\delta_{01}} 
		\|e^{t\Delta} (\varphi_1-\Phi_{\sigma_1} )\|_{L^{p,\infty}_{k}}
		\le C_1 \\
		\text{and}\quad 
		&t^{\frac{d}2(\frac{\sigma_2}{d} - \frac{k_0}{d} - \frac1{p_0})} 
		 	\|e^{t\Delta} \varphi_2\|_{L^{p_0,\infty}_{k_0}}
		+ t^{\frac{d}2(\frac{\sigma_2}{d} - \frac{k_0}{d} - \frac1{p_0})+\delta_{02}}
			\|e^{t\Delta} (\varphi_2-\Phi_{\sigma_2})\|_{L^{p_0,\infty}_{k_0}}
		\le C_2
	\end{aligned}
	\end{equation}
	for all $t>0$, where ${2+\gamma\over \alpha-1}=\sigma_1<\sigma_2$. 
	Then there exist $\delta_j>0$ (possibly smaller than $\delta_{0j}$), $j=1, 2$, such that
	\begin{equation}
	\label{t:asym.behv.lin:c1complexi11}
		\|u_1(t)-u_{1\mathcal S}(t)\|_{L^{\tilde p,\infty}_{\tilde k}}
		\leq C t^{-\frac{d}{2} 
				({1 \over q_c(\gamma)}-{\tilde k\over d} - {1\over \tilde p} )-\delta_1},
	\end{equation}				
	and
	\begin{equation}
	\label{t:asym.behv.lin:c1complexi12}
		\left\| 
			u_2(t)-z_2(t)		\right\|_{L^{\tilde p_0,\infty}_{\tilde k_0}}
		\leq C t^{-\frac{d}{2} (\frac{\sigma_2}{d}-\frac{\tilde k_0}{d}-\frac1{\tilde p_0}) -\delta_2}
	\end{equation}	
	for sufficiently large $t>0$, where $C$ is a positive constant. 

	\item[(ii)] $($``Modified linear"-Nonlinear behavior$)$
	Suppose that there exist positive constants $C_1$, $C_2$, $\delta_{01}$ and $\delta_{02}$ such that 
	\begin{equation}	\label{t:asym.behv-complex:c2}
	\begin{aligned}
		&t^{\frac{d}2(\frac{\sigma_1}{d} - \frac{k_0}{d} - \frac1{p_0})}
			\|e^{t\Delta} \varphi_1\|_{L^{p_0,\infty}_{k_0}}
		+t^{\frac{d}2(\frac{\sigma_1}{d} - \frac{k_0}{d} - \frac1{p_0})+\delta_{01}} 
			\|e^{t\Delta} (\varphi_1-\Phi_{\sigma_1})\|_{L^{p_0,\infty}_{k_0}}
		\le C_1  \\
		\text{and}\quad 
		&t^{\frac{d}2(\frac{\sigma_2}{d} - \frac{k}{d} - \frac1{p})\delta_{02}} 
			\|e^{t\Delta} (\varphi_2-\Phi_{\sigma_2} )\|_{L^{p,\infty}_{k}}
		\le C_2 
	\end{aligned}
	\end{equation}
	for all $t>0$, where ${2+\gamma\over \alpha-1}=\sigma_2<\sigma_1$. 
	Then there exist $\delta_j>0$  (possibly smaller than $\delta_{0j}$), $j=1, 2$, such that 
	\[
		\left\| 
			u_1(t)-w_1(t)
		\right\|_{L^{\tilde p_0,\infty}_{\tilde k_0}}
		\leq C t^{-\frac{d}{2} (\frac{\sigma_1}{d}-\frac{\tilde k_0}{d}-\frac1{\tilde p_0}) -\delta_1}
	\]
	and
	\[
		\|u_2(t)-u_{2\mathcal S}(t)\|_{L^{\tilde p,\infty}_{\tilde k}}
		\leq C t^{-\frac{d}{2} 
				({1 \over q_c(\gamma)}-{\tilde k\over d} - {1\over \tilde p} )-\delta_2}
	\]
	for sufficiently large $t>0$, where $C$ is a positive constant. 

\end{itemize}
	In the above, $(\tilde k, \tilde p) \in \R \times [1,\infty]$ satisfy \eqref{t:asym.behv:c.kp},  
	$(k_0,p_0)\in \R \times[1,\infty]$ satisfy 
	\begin{equation}	\label{t:asym.behv-complex:c.k0p0}
		(\alpha-2) \left({k \over d} + {1\over p }\right) +  \left({k_0 \over d} + {1\over p_0 }\right) < \frac{2+\gamma}{d} 
		\quad\text{if}\quad
		\alpha>3 \text{ and }
		\frac{k}{d} + \frac1{p}<\frac{k_0}{d} + \frac1{p_0} 
	\end{equation}
	in addition to \eqref{l:exist.crt.c.p0k0}, 
	and $(\tilde k_0,\tilde p_0)\in \R \times [1,\infty]$ satisfy \eqref{l:exist.crt':c.tild.kp} with $\sigma$ replaced by $\sigma_j$. 
\end{thm}

\begin{proof} 
We only give the proof of (i), since the proof of (ii) is identical to it. 
Throughout the proof, we will use the following inequality: For $x\in \R,\; x'\in \R,\; b>0,$ 
\begin{equation}	\label{KeyInequality}
	\left| |x+x'|^b-|x|^b \right|
	\leq C\left\{\begin{aligned} 
		&|x'|^b &&\mbox{if}\quad 0<b\leq 1, \\
		&|x|^{b-1}|x'|+|x'|^b &&\mbox{if}\quad b\geq 1.
	\end{aligned}\right.
\end{equation}
We write the integral formulation of the system \eqref{HHNLH-sys}, that is, 
\begin{equation}	\label{intHHNLH-sys}
\left\{\begin{aligned}
	u_1(t) = e^{t\Delta} \varphi_1 + a \int_0^t e^{(t-\tau)\Delta} 
	\left( |\cdot|^{\gamma} [u_1(\tau)^2+u_2(\tau)^2]^{{\alpha-1\over 2}} u_1(\tau)\right) \, d\tau,\\
	u_2(t) = e^{t\Delta} \varphi_2+ a \int_0^t e^{(t-\tau)\Delta} 
	\left( |\cdot|^{\gamma} [u_1(\tau)^2+u_2(\tau)^2]^{{\alpha-1\over 2}} u_2(\tau)\right)\, d\tau.\\
\end{aligned}\right.
\end{equation}
By standard argument, we may see that, for $M>0$ sufficiently small, we have 
\begin{equation}	\label{t:asym.behv-complex:decay}
\begin{aligned}
	&\sup_{t>0} t^{\frac{d}{2} (\frac{1}{q_c(\gamma)} - \frac{k}{d} - \frac{1}{p})}\|u_1(t)\|_{L^{p,\infty}_{ k}} 
		\le M, \quad
	\sup_{t>0} t^{\frac{d}{2} (\frac{1}{q_c(\gamma)} - \frac{k}{d} - \frac{1}{p})} \|u_2(t)\|_{L^{ p,\infty}_{k}} \le M \\
		\quad\text{and}\quad
	&\sup_{t>0} t^{\frac{d}{2} (\frac{\sigma_2}{d} - \frac{k_0}{d} - \frac{1}{p_0})} \|u_2(t)\|_{L^{ p_0,\infty}_{k_0}} <\infty, 
\end{aligned}
\end{equation}
where $k_0,$ $p_0$ are as in Theorem \ref{t:asym.behv} but with $\sigma$ replaced by $\sigma_2.$ 
Indeed, given $\varphi = \varphi_1 + i \varphi_2$ as prepared in Theorem \ref{t:asym.behv-complex}, 
it suffices to set the contraction mapping $\mathfrak{F}$ and the metric space $(X_M,d)$ as follows: 
Let $\rho>0$ and $M>0$ satisfy 
\begin{equation}\nonumber
	\rho +2C_3  M^{\alpha} \le M
\end{equation}
and 
\begin{equation}\nonumber
	\sup_{t>0}t^{\frac{d}{2} (\frac{\sigma_2}{d} - \frac{k_0}{d} - \frac{1}{p_0})}\|e^{t\Delta} \varphi_2\|_{L^{ p_0,\infty}_{k_0}}
	\le \rho
\end{equation}
in addition to \eqref{l:exist.crt.c0} with $T=\infty$ and \eqref{l:exist.crt:c.small}. Let 
\begin{equation}\nonumber
	\mathfrak{F}[u]
	= \operatorname{Re} \mathfrak{F}[u] + i  \operatorname{Im}\mathfrak{F}[u]
	:= e^{t\Delta} \varphi 
	+ a \int_0^t e^{(t-\tau)\Delta} \left( |\cdot|^{\gamma} |u(\tau)|^{\alpha-1} u(\tau)\right) \, d\tau
\end{equation}
and
\begin{equation}\nonumber
	X_M := \{ u=u_1 + i u_2\in 
			\mathcal{K}^p_k \,;\, 
	 		\|u\|_{\mathcal{K}^p_k}\le M, \ \|u_2\|_{\tilde{\mathcal{K}}^{p_0}_{k_0}}<\infty\}, 
\end{equation}
where 
\begin{equation}\nonumber
	\|u_2\|_{\tilde{\mathcal{K}}^{p_0}_{k_0}}
	:=\sup_{t>0}t^{\frac{d}{2} (\frac{\sigma_2}{d} - \frac{k_0}{d} - \frac{1}{p_0})}\|u_2(t)\|_{L^{ p_0,\infty}_{k_0}}
\end{equation}
with the metric 
\begin{equation}\nonumber
	d(u,v) := \|u-v\|_{\mathcal{K}^{p}_k}. 
\end{equation}
Then the mapping $\mathfrak{F}$ is a contraction mapping into $X_M$. 
For the fact that $X_M$ is stable under the mapping $\mathfrak{F}$, it suffices to see   
\begin{equation}\nonumber
\begin{aligned}
	\|\operatorname{Im}\mathfrak{F}[u]\|_{\tilde{\mathcal{K}}^{p_0}_{k_0}}
	&\leq \sup_{t>0}t^{\frac{d}{2} (\frac{\sigma_2}{d} - \frac{k_0}{d} - \frac{1}{p_0})}\|e^{t\Delta} \varphi_2\|_{L^{ p_0,\infty}_{k_0}}
		+2C_3   M^{\alpha-1} \|u_2\|_{\tilde{\mathcal{K}}^{p_0}_{k_0}} < \infty
\end{aligned}
\end{equation}
by a similar calculation to the proof of Lemma \ref{l:Kato.est2}, 
where $C_3$ is the constant in Lemma \ref{l:Kato.est2} and $(k_0, p_0)$ satisfies \eqref{l:exist.crt.c.p0k0}. 
The rest of the argument is straightforward so we omit the details. By the uniqueness in $\mathcal{K}^{p}_{k}$, we 
deduce that the solution obtained by the above process is the same as the one obtained in Theorem \ref{t:globalsol}. 

We now prove \eqref{t:asym.behv.lin:c1complexi11} for $(\tilde k_0, \tilde p_0)=(k_0, p_0)$ and $(\tilde k,\tilde p) = (k,p)$. 
To simplify the presentation, we write here and in some lines below, $u_j(\tau)$ by $u_j$ for $j=1, 2$ or $1\mathcal S$. 
For $u_1$, 
we have  
\begin{equation}	\label{t:asym.behv-complex:pr.u1}
\begin{aligned}
	&t^{\frac{d}{2}(\frac1{q_c(\gamma)} - \frac{k}{d} - \frac1{ p} )
		+\delta_1}\|u_1(t)-u_{1\mathcal S}(t)\|_{L^{ p,\infty}_{k}} \\ 
	& \leq t^{\frac{d}{2}(\frac1{q_c(\gamma)} - \frac{k}{d} - \frac1{ p} ) +\delta_1} 
		\|e^{t\Delta} \tilde{\varphi}_1\|_{L^{ p,\infty}_{k}} \\ 
	&\quad + t^{\frac{d}{2}(\frac1{q_c(\gamma)} - \frac{k}{d} - \frac1{ p} ) +\delta_1}
		\int_0^t \big\|e^{(t-\tau)\Delta}
		\big[ |\cdot|^{\gamma} \big( [u_1^2+u_2^2]^{{\alpha-1\over 2}} u_1-|u_{1}|^{\alpha-1} u_{1} \big) \big]\|_{L^{ p,\infty}_{k}}d\tau\\ 
	&\quad +  t^{\frac{d}{2}(\frac1{q_c(\gamma)} - \frac{k}{d} - \frac1{ p} ) +\delta_1} 
		\int_0^t \|e^{(t-\tau)\Delta}
		\big[|\cdot|^{\gamma}\big( |u_{1}|^{\alpha-1} u_{1} -|u_{1\mathcal S}|^{\alpha-1} u_{1\mathcal S} \big)
		\big]\|_{L^{ p,\infty}_{k}}\, d\tau \\
	&=: \mathbf{L}_1(t)
		+ {\mathbf{I}}_{1}(t) +  {\mathbf{II}}_{1}(t), 
\end{aligned}
\end{equation}
where $\tilde{\varphi}_1:= (\varphi_1 - \omega_1 |\cdot|^{-\frac{2+\gamma}{{\alpha-1}}})1_{\{|\cdot|\leq A_1\}} $. 
For $u_2$, we have  
\begin{equation}	\label{t:asym.behv-complex:pr.u2}
\begin{aligned}
	&t^{\frac{d}{2}(\frac{\sigma_2}{d} - \frac{k_0}{d} - \frac1{ p_0} ) +\delta_2}
		\|u_2(t)-z_2(t)\|_{L^{ p_0,\infty}_{k_0}} \\ 
	&\leq t^{\frac{d}{2}(\frac{\sigma_2}{d} - \frac{k_0}{d} - \frac1{ p_0} ) +\delta_2} 
		\|e^{t\Delta} \tilde{\varphi}_2\|_{L^{ p_0,\infty}_{k_0}}\\ 
	&\quad + t^{\frac{d}{2}(\frac{\sigma_2}{d} - \frac{k_0}{d} - \frac1{ p_0} ) +\delta_2} 
		\int_0^t \big\|e^{(t-\tau)\Delta}\big[ 
				|\cdot|^{\gamma} \big([u_1^2+u_2^2]^{{\alpha-1\over 2}} u_2-|u_1|^{\alpha-1}u_2\big)\big]\|_{L^{ p_0,\infty}_{k_0}} \,d\tau\\ 
	&\quad + t^{\frac{d}{2}(\frac{\sigma_2}{d} - \frac{k_0}{d} - \frac1{ p_0} ) +\delta_2} 
		\int_0^t \|e^{(t-\tau)\Delta}\big[ 
		|\cdot|^{\gamma} \big(|u_1|^{\alpha-1}u_2-|z_1|^{\alpha-1}z_2\big)\big] \|_{L^{ p_0,\infty}_{k_0}} \,d\tau\\ 
	&=: \mathbf{L}_2(t)
		+ {\mathbf{I}}_{2}(t)+ {\mathbf{II}}_{2}(t), 
\end{aligned}
\end{equation}
where $\tilde{\varphi}_2:= (\varphi_2 - \omega_2 |\cdot|^{-\sigma_2})1_{\{|\cdot|\leq A_2\}}$. 
In the above, $\delta_j$ is to be determined later. 
Our aim is to show that, for a suitable choice of $\delta_j$, the terms $\mathbf{L}_j(t)$, ${\mathbf{I}}_{j}(t)$ and ${\mathbf{II}}_{j}(t)$, $j=1,2$,
all tend to zero as $t\to\infty$.

\subsection{Estimate of $\mathbf{L}_1(t)$}


Since $|\tilde{\varphi}_1|\leq \left(\|\omega\|_{L^{\infty}}+c\right)|\cdot|^{-\frac{2+\gamma}{{\alpha-1}}} 1_{\{|\cdot|\leq A_1\}},$ 
we have $\tilde\varphi_1 \in L^{q_1,\infty}_{l_1}(\R^d)$ for any $(l_1, q_1)$ such that $\frac1{q_c(\gamma)} < \frac{l_1}{d} + \frac1{q_1}.$ 
Lemma \ref{l:linear-main} now implies $e^{t\Delta}\tilde\varphi_1\in L^{p,\infty}_k(\R^d)$ and 
\[
	\|e^{t\Delta}\tilde\varphi_1\|_{L^{p,\infty}_k}
	\le C t ^{-\frac{d}{2} (\frac{1}{q_1} - \frac{1}{p}) - \frac{l_1 - k}{2}}
		\|\tilde\varphi_1\|_{L^{q_1,\infty}_{l_1}}
	=  C t ^{-\frac{d}{2} (\frac{1}{q_c(\gamma)} - \frac{k}{d} - \frac{1}{p}) - \delta_{10}}
		\|\tilde\varphi_1\|_{L^{q_1,\infty}_{l_1}}, 
\]
where $p, q_1 \in [1,\infty)$, $k \le l_1$, 
$0<\frac{k}{d} + \frac1{p} < \frac1{q_c(\gamma)} <  \frac{l_1}{d} + \frac1{q_1}\le 1$ and 
\[
	\delta_{10}
	:= \frac{d}2 \left(\frac{l_1}{d} + \frac{1}{q_1} - \frac{1}{q_c(\gamma)} \right). 
\]
Thus, 
\begin{equation}	\label{t:asym.behv-complex:pr.L1}
	\sup_{t>0}t^{\frac{d}{2} (\frac{1}{q_c(\gamma)} - \frac{k}{d} - \frac{1}{p})+\delta_1}
	\|e^{t\Delta}\tilde\varphi\|_{L^{p,\infty}_k}\le C t^{-(\delta_{10}-\delta_1)},
\end{equation}
where
\begin{equation}	\label{t:asym.behv-complex:del1_L1}
	\delta_1<\delta_{10}. 
\end{equation}

\subsection{Estimate of ${\mathbf{I}}_1(t)$}


We write, using \eqref{KeyInequality},
\begin{align*}
	\big|[u_1^2+u_2^2]^{{\alpha-1\over 2}} u_1-|u_{1}|^{\alpha-1} u_{1}\big|
	&=\big|[u_1^2+u_2^2]^{{\alpha-1\over 2}} -[u_{1}^2]^{{\alpha-1\over 2}} \big| |u_{1}|\\ 
	&\hspace{-5cm} \leq C
	\left\{\begin{aligned}
		&|u_1|^{\alpha-2}|u_2|^{2}+|u_1| |u_2|^{\alpha-1} &&\mbox{if}\quad  \alpha>3,\\ 
		&|u_1| |u_2|^{\alpha-1} &&\mbox{if}\quad  1<\alpha\leq 3.
	\end{aligned}\right.
\end{align*}
Let us put 
\[
	\mathbf{I}_{11}(t)
	:=t^{\frac{d}2(\frac1{q_c(\gamma)} - \frac{k}{d} - \frac1{ p} ) +\delta_1}
		\int_0^t \big\|e^{(t-\tau)\Delta}\big[ 
			 |\cdot|^{\gamma}|u_1(\tau)|^{\alpha-2}|u_2(\tau)|^{2}\big] \big\|_{L^{ p,\infty}_{k}} \, d\tau 
\]
and 
\[
	\mathbf{I}_{12}(t)
	:=t^{\frac{d}2(\frac1{q_c(\gamma)} - \frac{k}{d} - \frac1{ p} ) +\delta_1}
		\int_0^t \big\|e^{(t-\tau)\Delta}\big[ 
			 |\cdot|^{\gamma}|u_1(\tau)||u_2(\tau)|^{\alpha-1}\big]\big\|_{L^{ p,\infty}_{k}}d\tau. 
\]
Then we have two terms $\mathbf{I}_{11}(t)$ and $\mathbf{I}_{12}(t)$ to estimate 
if $\alpha> 3,$ and $\mathbf{I}_{12}(t)$ if $1<\alpha\le 3.$

\smallbreak
\noindent
\underline{\sf Estimate of  $\mathbf{I}_{11}(t)$ for $\alpha>3$} \ 
Setting 
\begin{equation}	\nonumber
\begin{aligned}
	&{1\over p_{{11}}}={\alpha-2\over p}+{2\over p_{{\theta_{11}}}}, \quad 
	k_{{11}}=(\alpha-2)k+2k_{{\theta_{11}}}, \\
	&{1\over p_{{\theta_{11}}}}={{\theta_{11}}\over p}+{1-{\theta_{11}}\over p_0} 
			\quad\text{and}\quad 
	k_{{\theta_{11}}}={\theta_{11}}k+(1-{\theta_{11}})k_{0}, 
\end{aligned}
\end{equation}
we have 
\begin{equation}	\label{t:asym.behv.lin:pr.key.estcomplex3}
\begin{aligned}
	\mathbf{I}_{11}(t) 
	\leq C t^{\frac{d}2(\frac1{q_c(\gamma)} - \frac{k}{d} - \frac1{ p} ) +\delta_1} \int_0^t 
			(t-\tau)^{ -\frac{d}{2} (\frac{1}{p_{{11}}}- \frac{1}{p}) - \frac{ k_{{11}} - \gamma - k} {2} } 
			\||u_1(\tau)|^{\alpha-2}|u_2(\tau)|^{2} \|_{L^{p_{{11}}, \infty}_{ k_{{11}}}}
		\, d\tau
\end{aligned}
\end{equation}
and 
\begin{align*}
	\||u_1(\tau)|^{\alpha-2}&|u_2(\tau)|^{2} \|_{L^{p_{{11}}, \infty}_{ k_{{11}}}}
	\lesssim \|u_1(\tau)\|^{\alpha-2}_{L^{p, \infty}_{ k}}
		\|u_2(\tau)\|^{2}_{L^{p_{{\theta_{11}}}, \infty}_{k_{{\theta_{11}}}}} \\
	&\lesssim \|u_1(\tau)\|^{\alpha-2}_{L^{p, \infty}_{ k}}
			\|u_2(\tau)\|_{L^{p, \infty}_{ k}}^{2\theta_{11}} \|u_2(\tau)\|_{L^{p_0, \infty}_{k_0}}^{2(1-\theta_{11})} \\
	&\lesssim M^\alpha \tau^{-(\alpha-2+2{\theta_{11}} ) \frac{d}{2} (\frac{1}{q_c(\gamma)} - \frac{k}{d} - \frac{1}{p})
			-(1-{\theta_{11}}) d(\frac{\sigma_2}{d} - \frac{k_0}{d} - \frac{1}{p_0})}. 
\end{align*}
Thus, we obtain 
\begin{equation}	\label{t:asym.behv-complex:pr.I_{11}}
\begin{aligned}
	&\mathbf{I}_{11}(t) 
	\leq CM^\alpha t^{-(\delta_{11}-\delta_1)}
		\int_0^1 (1-\tau)^{ -\frac{d}{2} ({\alpha-3\over p}+2 ({{\theta_{11}}\over p}+{1-{\theta_{11}}\over p_0})) 
						- \frac{ (\alpha-3)k+ 2({\theta_{11}}k+(1-{\theta_{11}})k_{0}) - \gamma} {2} } \\
	&\hspace{5cm} \times
		\tau^{-(\alpha-2+2{\theta_{11}} ) \frac{d}{2} (\frac{1}{q_c(\gamma)} - \frac{k}{d} - \frac{1}{p})
			-(1-{\theta_{11}}) d(\frac{\sigma_2}{d} - \frac{k_0}{d} - \frac{1}{p_0})} \, d\tau, 
\end{aligned}
\end{equation}
where 
\begin{equation}	\label{t:asym.behv-complex:del1_I11}
	\delta_1<\delta_{11}:= (1-{\theta_{11}}) d\left({\sigma_2\over d}-{1\over q_c(\gamma)}\right)
	\quad\text{with}\quad 
	0\le {\theta_{11}} < 1, 
\end{equation}
provided that the following conditions are satisfied: 
\begin{equation}	\label{t:asym.behv-complex:I_{11}c1}
\begin{aligned}
	& \gamma \leq (\alpha-3)k +2 k_{0}  +2{\theta_{11}} (k-k_0),\quad 
		{\alpha-2\over p}+{2\over p_0} + 2 \theta_{11} \left({1\over p} - {1\over p_0}\right)<1,  \\
	& 
	\frac{k}{d}+\frac{1}{p}
	\le (\alpha-2) \left({k \over d} + {1\over p }\right) + 2 \left({k_0 \over d} + {1\over p_0 }\right) - \frac{\gamma}{d} 
	+2 {\theta_{11}}\left({k \over d} + {1\over p }- \left({k_0 \over d} + {1\over p_0 }\right)\right) < 1, 
\end{aligned}
\end{equation}
thanks to Lemma \ref{l:linear-main}. The last time-integral is convergent if and only if 
\begin{equation}	\label{t:asym.behv-complex:I_{11}c2}
\begin{aligned}
	&  (\alpha-3)  \left({k \over d} + {1\over p }\right) + 2 \left({k_0 \over d} + {1\over p_0 }\right)
		+ 2 \theta_{11} \left({k \over d} + {1\over p } - \left({k_0 \over d} + {1\over p_0 }\right)\right) <\frac{2+\gamma}{d}, \\
	&(\alpha-2) \left(\frac{1}{q_c(\gamma)} - \frac{k}{d} - \frac{1}{p}\right)
			+2 \left(\frac{\sigma_2}{d} - \frac{k_0}{d} - \frac{1}{p_0} \right)\\
	&\qquad 
		+ 2{\theta_{11}} \left(\frac{1}{q_c(\gamma)} - \frac{k}{d} - \frac{1}{p} - \left(\frac{\sigma_2}{d} - \frac{k_0}{d} - \frac{1}{p_0} \right)\right) <\frac2{d}.  
\end{aligned}
\end{equation}
By Lemma \ref{l:theta11}, there exists $\theta_{11}$ satisfying \eqref{t:asym.behv-complex:del1_I11}, 
\eqref{t:asym.behv-complex:I_{11}c1} and \eqref{t:asym.behv-complex:I_{11}c2}.

\smallbreak
\noindent
\underline{\sf Estimate of  $\mathbf{I}_{12}(t)$ for $\alpha>1$} \ 
Setting 
\begin{equation}	\nonumber
\begin{aligned}
	&{1\over p_{{12}}}={1\over p}+{\alpha-1\over p_{{\theta_{12}}}}, \quad 
		k_{{12}}=k+(\alpha-1)k_{{\theta_{12}}}, \\
	&{1\over p_{{\theta_{12}}}}={{\theta_{12}}\over p}+{1-{\theta_{12}}\over p_0}
		\quad\text{and}\quad 
	k_{{\theta_{12}}}={\theta_{12}}k+(1-{\theta_{12}})k_{0}, 
\end{aligned}
\end{equation}
we have 
\begin{equation}	\label{t:asym.behv.lin:pr.key.est.complex}
\begin{aligned}
	\mathbf{I}_{12}(t)
	\leq 
		C t^{\frac{d}2(\frac1{q_c(\gamma)} - \frac{k}{d} - \frac1{ p} ) +\delta_1} \int_0^t 
			(t-\tau)^{ -\frac{d}{2} (\frac{1}{p_{12}}- \frac{1}{p}) - \frac{k_{12} - \gamma - k} {2} } 
			\||u_1(\tau)||u_2(\tau)|^{\alpha-1} \|_{L^{p_{12}, \infty}_{k_{12}} }
		\, d\tau.
\end{aligned}
\end{equation}
and
\begin{align*}
	\||u_1(\tau)||u_2(\tau)|^{\alpha-1} \|_{L^{p_{{12}}, \infty}_{k_{{12}}}}
	&\lesssim \|u_1(\tau)\|_{L^{p, \infty}_{ k}}\|u_2(\tau)\|^{\alpha-1}_{L^{p_{{\theta_{12}}}, \infty}_{k_{{\theta_{12}}}}} \\
	&\lesssim \|u_1(\tau)\|_{L^{p, \infty}_{ k}} 
		\|u_2(\tau)\|_{L^{{p}, \infty}_{k}}^{(\alpha-1)\theta_{12}}
		\|u_2(\tau)\|_{L^{{p_0}, \infty}_{k_0}}^{(\alpha-1)(1-\theta_{12})} \\
	& \hspace{-3cm} 
	\lesssim M^\alpha \tau^{- (1+{\theta_{12}} (\alpha-1)) \frac{d}{2} (\frac{1}{q_c(\gamma)} - \frac{k}{d} - \frac{1}{p}) 
					-(1-{\theta_{12}})\frac{d(\alpha-1)}{2} (\frac{\sigma_2}{d} - \frac{k_0}{d} - \frac{1}{p_0})}. 
\end{align*}
Thus, we obtain 
\begin{equation}		\label{t:asym.behv-complex:pr.I_{12}}
\begin{aligned}
	&\mathbf{I}_{12}(t)
	\leq CM^\alpha t^{-(\delta_{12}-\delta_1)}
		\int_0^1 (1-\tau)^{ -\frac{d(\alpha-1)}{2} ( {{\theta_{12}}\over p}+{1-{\theta_{12}}\over p_0} ) 
						- \frac{(\alpha-1)({\theta_{12}}k+(1-{\theta_{12}})k_{0})- \gamma} {2} }  \\
	&\hspace{4.5cm}\times
		\tau^{- (1+{\theta_{12}}(\alpha-1) )\frac{d}{2} (\frac{1}{q_c(\gamma)} - \frac{k}{d} - \frac{1}{p})
	 		-(1-{\theta_{12}}) \frac{d(\alpha-1)}{2}(\frac{\sigma_2}{d} - \frac{k_0}{d} - \frac{1}{p_0})}d\tau, 
\end{aligned}
\end{equation}
where 
\begin{equation}	\label{t:asym.behv-complex:del1_I12}
	\delta_1<\delta_{12}
	:= {d(\alpha-1)\over 2}(1-{\theta_{12}})\left({\sigma_2\over d}-{1\over q_c(\gamma)}\right)
	\quad\text{with}\quad 
	0\le {\theta_{12}} < 1, 
\end{equation}
provided that the following conditions are satisfied: 
\begin{equation}	\label{t:asym.behv-complex:I_{12}c1}
\begin{aligned}
	&{\gamma \over \alpha-1}-k_0 \leq {\theta_{12}}(k-k_{0}), \quad 
		 \theta_{12} \left({1\over p} - {1\over p_0} \right) < \frac1{\alpha-1} \left(1-{1\over p} \right)-{1\over p_0}, \\
	& 0\le \theta_{12} (\alpha-1) \left(\frac{k}{d} + \frac{1}{p} - \frac{k_0}{d} - \frac{1}{p_0} \right)
		+(\alpha-1)\left(\frac{k_0}{d} + \frac{1}{p_0} \right) - {\gamma\over d}
		<1-\frac{k}{d}-\frac{1}{p},
\end{aligned}
\end{equation}
thanks to Lemma \ref{l:linear-main}. The last time-integral is convergent if and only if 
\begin{equation}	\label{t:asym.behv-complex:I_{12}c2}
\begin{aligned}
	&\theta_{12} \left( {k \over d} + {1\over p} -{k_0\over d} - {1\over p_0} \right) 
		<\frac1{q_c(\gamma)} - \left( {k_0\over d} + {1\over p_0} \right), \\
	&\frac{1}{q_c(\gamma)} - \frac{k}{d} - \frac{1}{p} + (\alpha-1) \left(\frac{\sigma_2}{d} - \frac{k_0}{d} - \frac{1}{p_0} \right)  \\
	&\qquad
		+\theta_{12} (\alpha-1) 
			\left(\frac{1}{q_c(\gamma)} - \frac{k}{d} - \frac{1}{p} - \left(\frac{\sigma_2}{d} - \frac{k_0}{d} - \frac{1}{p_0} \right) \right) 
		<\frac2{d}. 
\end{aligned}
\end{equation}
By Lemma \ref{l:theta12}, there exists $\theta_{12}$ satisfying \eqref{t:asym.behv-complex:del1_I12}, 
\eqref{t:asym.behv-complex:I_{12}c1} and \eqref{t:asym.behv-complex:I_{12}c2}.

\subsection{Estimate of ${\mathbf{II}}_1(t)$}

For ${\mathbf{II}}_1(t)$, we recall the inequality 
\[
	\left||u_{1}|^{\alpha-1} u_{1} -|u_{1\mathcal S}|^{\alpha-1} u_{1\mathcal S}\right|
	\leq C (|u_{1}|^{\alpha-1} +|u_{1\mathcal S}|^{\alpha-1})| u_{1}- u_{1\mathcal S}|. 
\]
Then the estimate of ${\mathbf{II}}_1(t)$ can be carried out similarly to the proof of Lemma \ref{l:crt.est2}. 
Fix $t' < t$. Given sufficiently large finite $T>0$, we have 
\begin{align*}
	&{\mathbf{II}}_1(t) 
	\lesssim \int_0^1
		(1-\tau)^{-\frac{d}{2} ((\alpha-1) (\frac{k}{d} + \frac{1}{p}) -\frac{\gamma}{d} ) } 
		\tau^{-\frac{d(\alpha-1)}{2} (\frac1{q_c(\gamma)} - \frac{k}{d} - \frac1{p} ) } 
		\|u_1(t \tau)-u_{1\mathcal S}(t \tau)\|_{L^{p,\infty}_{k}} d\tau \\
	&\qquad 
	\times \left(\|u_1\|_{\mathcal{K}^{p}_{k}}^{\alpha-1}
				+ \|u_{1\mathcal S}\|_{\mathcal{K}^{p}_{k}}^{\alpha-1}\right) \\
	&\lesssim M^{\alpha-1} \int_0^1
		(1-\tau)^{-\frac{d}{2} ((\alpha-1) (\frac{k}{d} + \frac{1}{p}) -\frac{\gamma}{d} ) } 
		\tau^{-\frac{d(\alpha-1)}{2} (\frac1{q_c(\gamma)} - \frac{k}{d} - \frac1{p} ) } 
		(t\tau)^{-\frac{d}2(\frac1{q_c(\gamma)} - \frac{k}{d} - \frac1{p} )-\delta_1}  \\
	&\hspace{2cm} 
	\times \sup_{T\tau\ge t\tau\ge t'\tau} \left((t\tau)^{\frac{d}2(\frac1{q_c(\gamma)} - \frac{k}{d} - \frac1{p} )+\delta_1} 
		\|u_1(t \tau)-u_{1\mathcal S}(t \tau)\|_{L^{p,\infty}_{k}} \right) d\tau
\end{align*}
thanks to Lemma \ref{l:linear-main}. 
Since 
\begin{equation}\nonumber
\begin{aligned}
	\sup_{T\tau\ge t\tau\ge t'\tau} 
			&\left((t\tau)^{\frac{d}2(\frac1{q_c(\gamma)} - \frac{k}{d} - \frac1{p} )+\delta_1} 
				\| u_1(t\tau)- u_{1\mathcal S}(t\tau)\|_{L^{p,\infty}_k}\right) \\
	&\le T^{\delta_1} ( \|u_1\|_{\mathcal{K}^{p}_{k}} +  \|u_{1\mathcal S}\|_{\mathcal{K}^{p}_{k}} )
	\quad\text{for all}\quad \tau \in [0,1] 
\end{aligned}
\end{equation}
and 
\begin{equation}\nonumber
\begin{aligned}
	\limsup_{t\to T}&\left((t\tau)^{\frac{d}2(\frac1{q_c(\gamma)} - \frac{k}{d} - \frac1{p} )+\delta_1} 
				\| u_1(t\tau)- u_{1\mathcal S}(t\tau)\|_{L^{p,\infty}_k}\right)  \\
	&= \limsup_{t\to T} t^{\frac{d}2(\frac1{q_c(\gamma)} - \frac{k}{d} - \frac1{p} )+\delta_1} 
						\| u_1(t)- u_{1\mathcal S}(t)\|_{L^{p,\infty}_k}\quad\text{for a.e.}\quad \tau \in [0,1], 
\end{aligned}
\end{equation}
Lebesgue's dominated convergence theorem yields 
\begin{equation}	\label{t:asym.behv-complex:pr.II_{1}}
\begin{aligned}
	\limsup_{t\to T}{\mathbf{II}}_1(t) 
	&\leq 2 C M^{\alpha-1} 
		\int_0^1(1-\tau)^{-\frac{d}{2} ((\alpha-1) (\frac{k}{d} + \frac{1}{p}) -\frac{\gamma}{d} ) } 
		\tau^{-\frac{d\alpha}2(\frac1{q_c(\gamma)} - \frac{k}{d} - \frac1{p} )-\delta_1} \,d\tau \\
	&\qquad  \times\limsup_{t\to T}
			\left(t^{\frac{d}2(\frac1{q_c(\gamma)} - \frac{k}{d} - \frac1{p} )+\delta_1} 
				\| u_1(t)- u_{1\mathcal S}(t)\|_{L^{p,\infty}_k}\right), 
\end{aligned}
\end{equation}
where the last integral is convergent if  
\begin{equation}	\label{t:asym.behv-complex:del1_II1}
	\delta_1<1-{d\alpha\over 2}\left({1\over q_c(\gamma)}-{k\over d}-{1\over p}\right). 
\end{equation}
Later in the proof, we will take $M$ smaller (if necessary) in order to obtain a closed estimate, and then let $T\to\infty$.

\subsection{Estimate of $\mathbf{L}_2(t)$}

Since 
$|\tilde{\varphi}_2|\leq \left(\|\omega\|_{L^{\infty}}+c\right)|\cdot|^{-{\sigma_2}}1_{\{|\cdot|\leq A_2\}}$, 
$\varphi_2 \in L^{q_1}_{l_1}(\R^d)$ for any $(l_1, q_1)$ such that 
\begin{equation} 	
	\sigma_2 = l_0 + \frac{d}{q_0} < l_1 + \frac{d}{q_1}. 
\end{equation}
Lemma \ref{l:linear-main} now implies $e^{t\Delta}\tilde\varphi_2\in L^{p_0,\infty}_{k_0}(\R^d)$ and 
\[
	\|e^{t\Delta}\tilde\varphi_2\|_{L^{p_0,\infty}_{k_0}}
	\le C t ^{-\frac{d}{2} (\frac{1}{q_1} - \frac{1}{p_0}) - \frac{l_1 - k_0}{2}}
		\|\tilde\varphi_2\|_{L^{q_1,\infty}_{l_1}}
	=  C t ^{-\frac{d}{2} (\frac{\sigma_2}{d} - \frac{k_0}{d} - \frac{1}{p_0}) - \delta_{20}}
		\|\tilde\varphi_2\|_{L^{q_1,\infty}_{l_1}}, 
\]
where $p_0, q_1 \in [1,\infty)$, $k_0 \le l_1$, 
$0<\frac{k_0}{d} + \frac1{p_0}<  \frac{\sigma_2}{d} < \frac{l_1}{d} + \frac1{q_1}\le 1$ and 
\[
	\delta_{20} := \frac{d}{2} \left(\frac{l_1}{d}+\frac1{q_1}-\frac{\sigma_2}{d}\right). 
\]
Thus, there exists a constant $C>0$ such that 
\begin{equation}	\label{t:asym.behv-complex:pr.L2}
	\sup_{t>0} t^{\frac{d}{2} (\frac{\sigma_2}{d}-\frac{k_0}{d}-\frac1{p_0}) + \delta_2} 
	\|e^{t\Delta}\varphi_2 \|_{L^{p_0,\infty}_{k_0}} \le C t^{-(\delta_{20}-\delta_2)}
\end{equation}
where 
\begin{equation}	\label{t:asym.behv-complex:del2_L2}
	\delta_2 < \delta_{20}. 
\end{equation}

\subsection{Estimate of ${\mathbf{I}}_2(t)$}

We next show that ${\mathbf{I}}_2(t)$ is finite. Using \eqref{KeyInequality}, we have
\begin{align*}
	|[u_1^2+u_2^2]^{{\alpha-1\over 2}} u_2-|u_1|^{\alpha-1}u_2 |
	&=\big|[u_1^2+u_2^2]^{{\alpha-1\over 2}} -[u_{1}^2]^{{\alpha-1\over 2}} \big||u_2 | \\ 
	&\hspace{-5cm} \leq C
	\begin{cases}
		|u_1|^{\alpha-3}|u_2|^{3}+|u_2|^{\alpha}, \quad \mbox{if}\quad \alpha> 3,\\ 
		|u_2|^{\alpha}, \quad \mbox{if}\quad 1<\alpha\leq 3.
	\end{cases}
\end{align*}
We have two terms to estimate in $L^{ p_0,\infty}_{k_0}(\R^d)$:
\[
	\mathbf{I}_{21}(t):=t^{\frac{d}{2}(\frac{\sigma_2}{d} - \frac{k_0}{d} - \frac1{ p_0} ) +\delta_2}
	\int_0^t \big\|e^{(t-\tau)\Delta}\big[ 
		 |\cdot|^{\gamma} |u_1(\tau)|^{\alpha-3}|u_2(\tau)|^3\big]\|_{L^{ p_0,\infty}_{k_0}}d\tau 
\]
if $\alpha>3$, and 
\[
	\mathbf{I}_{22}(t):=t^{\frac{d}{2}(\frac{\sigma_2}{d} - \frac{k_0}{d} - \frac1{ p_0} ) +\delta_2}
	\int_0^t \big\|e^{(t-\tau)\Delta}\big[ 
			|\cdot|^{\gamma} |u_2(\tau)|^{\alpha} 
		\big]\|_{L^{ p_0,\infty}_{k_0}}d\tau
\]
if $1<\alpha\le 3$.

\smallbreak
\noindent
\underline{\sf Estimate of $\mathbf{I}_{21}(t)$ for $\alpha>3$}. 
Setting 
\begin{align*}
	&{1\over p_{21} }={\alpha-3\over p}+{3\over p_{{\theta_{21}}}}, \quad 
		k_{{21}}=(\alpha-3)k+3k_{{\theta_{21}}}, \\
	&{1\over p_{{\theta_{21}}}}={{\theta_{21}}\over p}+{1-{\theta_{21}}\over p_0}, \quad 
		k_{{\theta_{21}}}={\theta_{21}}k+(1-{\theta_{21}})k_{0},
\end{align*}
we have 
\begin{equation}\nonumber
\begin{aligned}
	\mathbf{I}_{21}(t) 
	&\leq C t^{\frac{d}2(\frac{\sigma_2}{d} - \frac{k_0}{d} - \frac1{ p_0} ) +\delta_2} \\
	&\quad\times		\int_0^t 
				(t-\tau)^{ -\frac{d}{2} (\frac{1}{p_{{21}} }- \frac{1}{p_0}) - \frac{ k_{{21}} - \gamma - k_0} {2} } 
				\||u_1(\tau)|^{\alpha-3}|u_2(\tau)|^{3} \|_{L^{p_{{21}}, \infty}_{ k_{{21}}}} 
			\,d\tau. 
\end{aligned}
\end{equation}
and
\begin{align*}
	\||u_1(\tau)|^{\alpha-3}|u_2(\tau)|^{3} &\|_{L^{p_{{21}}, \infty}_{ k_{21}}  }
	\lesssim \||u_1(\tau)|^{\alpha-3}\|_{L^{p/(\alpha-3), \infty}_{ (\alpha-3)k}}\||u_2(\tau)|^{3} \|_{L^{p_{{\theta_{21}}}/3, \infty}_{3k_{{\theta_{21}}}}}\\ 
	&\lesssim \|u_1(\tau)\|^{\alpha-3}_{L^{p, \infty}_{ k}}\|u_2(\tau)\|^{3}_{L^{p_{{\theta_{21}}}, \infty}_{k_{{\theta_{21}}}}} \\
	&\lesssim M^\alpha \tau^{-\frac{d}{2}(\alpha-3+3{\theta_{21}}) (\frac{1}{q_c(\gamma)} - \frac{k}{d} - \frac{1}{p})}
		\tau^{-\frac{d}{2}(1-{\theta_{21}})3 (\frac{\sigma_2}{d} - \frac{k_0}{d} - \frac{1}{p_0})}. 
\end{align*}
Thus we obtain 
\begin{equation}	\label{t:asym.behv-complex:pr.I_21}
\begin{aligned}
	&t^{\frac{d}2(\frac{\sigma_2}{d} - \frac{k_0}{d} - \frac1{ p_0} ) +\delta_2}
	\Big\| \int_0^t e^{(t-\tau)\Delta}
			\big[|\cdot|^{\gamma}|u_1(\tau)|^{\alpha-3}|u_2(\tau)|^{3}\big]\,d\tau
	\Big\|_{L^{ p_0,\infty}_{k_0}} \\
	&\leq CM^\alpha t^{-(\delta_{21}-\delta_2)} \\
	&\quad\times		\int_0^1
			(1-\tau)^{ -\frac{d}{2} (\frac{1}{p_{{21}}}- \frac{1}{p_0}) - \frac{ k_{{21}} - \gamma - k_0} {2} } 
			\tau^{-\frac{d}{2} (\alpha-3+3{\theta_{21}}) (\frac{1}{q_c(\gamma)} - \frac{k}{d} - \frac{1}{p}) 
					-\frac{d}{2}(1-{\theta_{21}})3 (\frac{\sigma_2}{d} - \frac{k_0}{d} - \frac{1}{p_0})} \,d\tau,
\end{aligned}
\end{equation}
where
\begin{equation}	\label{t:asym.behv-complex:del2_I21}
	0<\delta_2< \delta_{21}:= {d\over 2}(2-3{\theta_{21}})\left({\sigma_2\over d}-{1\over q_c(\gamma)}\right)
	\quad\text{with}\quad
	0\leq \theta_{21}<{2\over 3}, 
\end{equation}
provided that the following conditions are satisfied: 
\begin{equation}	\label{t:asym.behv-complex:I_{21}c1}
\begin{aligned}
	&\gamma\leq (\alpha-3)k+2k_0 + 3{\theta_{21}} (k-k_0), \quad 
		{\alpha-3\over p} + {3\over p_0}+3\theta_{21}\left( {1 \over p}-{1\over p_0} \right) <1, \\
	& \frac{k_0}{d}+\frac{1}{p_0}\le (\alpha-3) \left( {k\over d} + {1\over p} \right)+ 3\left( {k_0\over d} + {1\over p_0} \right) - {\gamma\over d} \\
	&\hspace{5cm}
		+3{\theta_{21}}  \left(  {k\over d} + {1\over p} - \left( {k_0\over d} + {1\over p_0} \right)\right) <1,
\end{aligned}
\end{equation}
thanks to Lemma \ref{l:linear-main}. The last time-integral is convergent if and only if 
\begin{equation}	\label{t:asym.behv-complex:I_{21}c2}
\begin{aligned}
	&(\alpha-3)\left( {k\over d} + {1\over p}\right) + 2\left( {k_0\over d} + {1\over p_0}\right)
		+3{\theta_{21}} \left( {k\over d} + {1\over p}-\left( {k_0\over d} + {1\over p_0}\right) \right) 
		- \frac{\gamma } {d} <\frac{2}{d}, \\
	&(\alpha-3) \left(\frac{1}{q_c(\gamma)} - \frac{k}{d} - \frac{1}{p}\right) 
			+ 3\left(\frac{\sigma_2}{d} - \frac{k_0}{d} - \frac{1}{p_0}\right) \\
	&\hspace{3cm}
	+3{\theta_{21}} \left(\frac{1}{q_c(\gamma)} - \frac{k}{d} - \frac{1}{p} -\left(\frac{\sigma_2}{d} - \frac{k_0}{d} - \frac{1}{p_0}\right) \right)<\frac2{d}.
\end{aligned}
\end{equation}
The existence of $\theta_{21}$ satisfying \eqref{t:asym.behv-complex:del2_I21}, 
\eqref{t:asym.behv-complex:I_{21}c1} and \eqref{t:asym.behv-complex:I_{21}c2} is ensured by Lemma \ref{l:theta21}.

\smallbreak
\noindent
\underline{\sf Estimate of $\mathbf{I}_{22}(t)$ for $\alpha>1$} \ 
We estimate $\mathbf{I}_{22}(t)$ similarly to the proof of Theorem \ref{t:asym.behv.lin}. By the same calculations leading 
to \eqref{t:asym.behv.lin:pr.key.est}, we have
\begin{equation}	\label{t:asym.behv-complex:pr.I_22}
\begin{aligned}
	\mathbf{I}_{22}(t) 
	&\leq C M^{\alpha} t^{-(\delta_{22}-\delta_2)} \int_0^1
			(1-\tau)^{ -\frac{d\alpha \theta_{22}}{2} ( \frac{1}{p} - \frac{1}{p_0}) - \frac{\alpha \theta_{22} (k -k_0)- \gamma} {2} } \\
	&\hspace{2cm}
		\times \tau^{-\theta_{22} \alpha \frac{d}{2} (\frac{1}{q_c(\gamma)}-\frac{k}{d}-\frac1{p} ) 
			-(1-\theta_{22})\alpha \frac{d}{2} (\frac{\sigma}{d}-\frac{k_0}{d}-\frac1{p_0} )} \, d\tau, 
\end{aligned}
\end{equation}
with
\begin{equation}	\label{t:asym.behv-complex:del2_I22}
	0<\delta_2<\delta_{22}:= \frac{d}{2} \left( \frac{\sigma_2}{d} - \frac{1}{q_c(\gamma)} \right) ( \alpha-1 - \theta_{22}\alpha)
	\quad\text{with}\quad 
	0\le \theta_{22} < 1-\frac1{\alpha}, 
\end{equation}
provided that 
\begin{equation}	\label{t:asym.behv-complex:I_{22}c1}
\begin{aligned}
	&\gamma - (\alpha-1)k_0 \le \theta_{22} \alpha ( k - k_0 ), \quad
		\theta_{22} \left( \frac{1}{p} -\frac{1}{p_0} \right) < \frac1{\alpha} -  \frac{1}{p_0}, \quad
		0 \le \frac{k_0}{d} + \frac{1}{p_0} \quad\text{and}\\
	& \frac{\gamma}{d} - (\alpha-1) \left(\frac{k_0}{d} + \frac1{p_0}\right)
		\le \theta_{22} \alpha \left(\frac{k}{d} + \frac1{p} - \frac{k_0}{d} - \frac1{p_0}\right) 
		 < 1 + \frac{\gamma}{d} - \alpha \left(\frac{k_0}{d} + \frac1{p_0}\right). 
\end{aligned}
\end{equation}
The time-integral is convergent if and only if 
\begin{equation}	\label{t:asym.behv-complex:I_{22}c2}
\begin{aligned}
	&	(\alpha-1) \left( \frac1{q_c(\gamma)} - \frac{k_0}{d} - \frac1{p_0} \right)
	- \theta_{22} \alpha \left( \frac{k}{d} + \frac1{p} - \frac{k_0}{d} - \frac1{p_0}\right)>0 \\
		\text{and}\quad
	&\frac{2}{d\alpha} - \frac{\sigma_2}{d}+ \frac{k_0}{d} +\frac1{p_0}  
		+ \theta_{22} \left( \frac{\sigma_2}{d}-\frac{k_0}{d}-\frac1{p_0} - \left(\frac{1}{q_c(\gamma)}-\frac{k}{d}-\frac1{p} \right) \right)
		>0. 
\end{aligned}
\end{equation}
Such a $\theta_{22}$ satisfying \eqref{t:asym.behv-complex:I_{22}c1} and \eqref{t:asym.behv-complex:I_{22}c2} exists, 
as done in the proof of Theorem \ref{t:asym.behv.lin} (Take $\theta_{22}=\theta$ in Theorem \ref{t:asym.behv.lin} 
with $\sigma$ replaced by $\sigma_2$).

\subsection{Estimate of ${\mathbf{II}}_2(t)$}

Using \eqref{KeyInequality}, we have 
\begin{align*}
	\Big||u_1|&^{\alpha-1}u_2-|u_{1\mathcal S}|^{\alpha-1}z_2\Big|
	=\Big||u_1|^{\alpha-1}u_2-|u_{1\mathcal S}|^{\alpha-1}z_2+|u_{1\mathcal S}|^{\alpha-1}u_2-|u_{1\mathcal S}|^{\alpha-1}u_2\Big| \\
	&\leq C
	\left\{\begin{aligned}
		&|u_{1\mathcal S}|^{\alpha-1}|u_2-z_2|+|u_2||u_1-u_{1\mathcal S}|(|u_1|^{\alpha-2}+|u_{1\mathcal S}|^{\alpha-2}), &&\alpha\geq 2, \\ 
		&|u_{1\mathcal S}|^{\alpha-1}|u_2-z_2|+|u_2||u_1-u_{1\mathcal S}|^{\alpha-1}, &&1<\alpha< 2.
	\end{aligned}\right.\\
	&\leq C 
	|u_{1\mathcal S}|^{\alpha-1}|u_2-z_2|+C
	\left\{\begin{aligned}
		&|u_2||u_1-u_{1\mathcal S}|(|u_1|^{\alpha-2}+|u_{1\mathcal S}|^{\alpha-2}), &&\alpha\ge2 ,\\ 
		&|u_2||u_1-u_{1\mathcal S}|^{\alpha-1}, &&1<\alpha< 2.
	\end{aligned}\right.
\end{align*}
We have three terms to estimate:
\[
	\hspace{-0.8cm}
	\mathbf{II}_{21}(t)
	:=t^{\frac{d}2(\frac{\sigma_2}{d} - \frac{k_0}{d} - \frac1{p_0} ) +\delta_2}\int_0^t \big\|e^{(t-\tau)\Delta}\big[ 
	|\cdot|^{\gamma}|u_2-z_2||u_{1\mathcal S}|^{\alpha-1}\big]\big\|_{L^{ p_0,\infty}_{k_0}}d\tau
\]
if $\alpha>1$; 
\[
	\hspace{-0.8cm}
	\mathbf{II}_{22}(t)
	:=t^{\frac{d}2(\frac{\sigma_2}{d} - \frac{k_0}{d} - \frac1{p_0} ) +\delta_2}\int_0^t \big\|e^{(t-\tau)\Delta}\big[ 
	 |\cdot|^{\gamma}|u_1-u_{1\mathcal S}||u_j|^{\alpha-2}|u_2|\big]\big\|_{L^{ p_0,\infty}_{k_0}}d\tau
\]
for $j=1, 1\mathcal{S}$ if $\alpha\geq 2$; and 
\[
	\hspace{-0.8cm}
	\mathbf{II}_{23}(t)
	:=t^{\frac{d}2(\frac{\sigma_2}{d} - \frac{k_0}{d} - \frac1{p_0} ) +\delta_2}\int_0^t \big\|e^{(t-\tau)\Delta}\big[ 
	 |\cdot|^{\gamma}|u_1-u_{1\mathcal S}|^{\alpha-1}|u_2|\big]d\tau\big\|_{L^{ p_0,\infty}_{k_0}}d\tau
\]
if $1<\alpha< 2.$

\smallbreak
\noindent
\underline{\sf Estimate of $\mathbf{II}_{21}(t)$ for $\alpha>1$}. 
For $\mathbf{II}_{21}(t)$, we carry out the estimate as in the proof of Lemma \ref{l:Kato.est2}. 
Given arbitrary $T>0$, we have 
\begin{equation}	\label{t:asym.behv-complex:pr.II_21}
\begin{aligned}
	&\mathbf{II}_{21}(t)=t^{\frac{d}2(\frac{\sigma_2}{d} - \frac{k_0}{d} - \frac1{ p0} ) +\delta_2}
		\int_0^t \big\|e^{(t-\tau)\Delta}
				\big[ |\cdot|^{\gamma}|u_2(\tau)-z_2(\tau)||u_{1\mathcal S}(\tau)|^{\alpha-1}\big]
			\big\|_{L^{ p_0,\infty}_{k_0}}d\tau \\
	&\leq C \| u_{1\mathcal S} \|^{\alpha-1} _{\mathcal{K}^{p}_{k}}
			\int_0^1 
			(1-\tau)^{ -\frac{d(\alpha-1)}{2p} - \frac{(\alpha-1) k- \gamma } {2} } 
			\tau^{-\frac{d (\alpha-1) }{2} (\frac{1}{q_c(\gamma)}-\frac{k}{d}-\frac1{p} ) 
			- \frac{d}{2} (\frac{\sigma_2}{d}-\frac{k_0}{d}-\frac1{p_0} )-\delta_2}  \\
	&\hspace{+3cm} \times 
			 \sup_{T\tau\ge t\tau\ge t'\tau} 
			 	\left( 
					(t\tau)^{\frac{d}2(\frac{\sigma_2}{d} - \frac{k_0}{d} - \frac1{ p_0} ) +\delta_2}  
					\|u_2(t\tau)-z_2(t\tau)\|_{L^{p_0, \infty}_{k_0}}
				\right) \, d\tau. 
\end{aligned}
\end{equation}
Thus, 
\begin{equation}	\nonumber
\begin{aligned}
	&\sup_{T\ge t\ge t'} \mathbf{II}_{21}(t) \\
	&\leq C \| u_{1\mathcal S} \|^{\alpha-1} _{\mathcal{K}^{p}_{k}}
			\int_0^1 
			(1-\tau)^{ -\frac{d(\alpha-1)}{2p} - \frac{(\alpha-1) k- \gamma } {2} } 
			\tau^{-\frac{d (\alpha-1) }{2} (\frac{1}{q_c(\gamma)}-\frac{k}{d}-\frac1{p} ) 
			- \frac{d}{2} (\frac{\sigma_2}{d}-\frac{k_0}{d}-\frac1{p_0} )-\delta_2}  \\
	&\hspace{+3cm} \times 
			 \sup_{T\tau\ge t\tau\ge t'\tau} 
			 	\left( 
					(t\tau)^{\frac{d}2(\frac{\sigma_2}{d} - \frac{k_0}{d} - \frac1{ p_0} ) +\delta_2}  
					\|u_2(t\tau)-z_2(t\tau)\|_{L^{p_0, \infty}_{k_0}}
				\right) \, d\tau. 
\end{aligned}
\end{equation}
Since 
\begin{equation}\nonumber
\begin{aligned}
	\sup_{T\tau\ge t\tau\ge t'\tau} &
			 	\left( 
					(t\tau)^{\frac{d}2(\frac{\sigma_2}{d} - \frac{k_0}{d} - \frac1{ p_0} ) +\delta_2}  
					\|u_2(t\tau)-z_2(t\tau)\|_{L^{p_0, \infty}_{k_0}}
				\right) \\
	&\le T^{\delta_2} ( \|u_2\|_{\tilde{\mathcal{K}}^{p_0}_{k_0}} + \|z_2\|_{\tilde{\mathcal{K}}^{p_0}_{k_0}} )
	\quad\text{for all}\quad \tau \in [0,1] 
\end{aligned}
\end{equation}
and 
\begin{equation}\nonumber
\begin{aligned}
	\limsup_{t\to T}& (t\tau)^{\frac{d}2(\frac{\sigma_2}{d} - \frac{k_0}{d} - \frac1{ p_0} ) +\delta_2}  
					\|u_2(t\tau)-z_2(t\tau)\|_{L^{p_0, \infty}_{k_0}} \\
	&= \limsup_{t\to T} t^{\frac{d}2(\frac{\sigma_2}{d} - \frac{k_0}{d} - \frac1{ p_0} ) +\delta_2}  
					\|u_2(t)-z_2(t)\|_{L^{p_0, \infty}_{k_0}} \quad\text{for a.e.}\quad \tau \in [0,1], 
\end{aligned}
\end{equation}
Lebesgue's dominated convergence theorem yields 
\begin{equation}	\nonumber
\begin{aligned}
	& \limsup_{t\to T} \mathbf{II}_{21}(t) \\
	&\leq C \| u_{1\mathcal S} \|^{\alpha-1} _{\mathcal{K}^{p}_{k}}
			\int_0^1 
			(1-\tau)^{ -\frac{d(\alpha-1)}{2p} - \frac{(\alpha-1) k- \gamma } {2} } 
			\tau^{-\frac{d (\alpha-1) }{2} (\frac{1}{q_c(\gamma)}-\frac{k}{d}-\frac1{p} ) 
			- \frac{d}{2} (\frac{\sigma_2}{d}-\frac{k_0}{d}-\frac1{p_0} )-\delta_2}  \,d\tau \\
	&\hspace{+3cm} \times 
			 \limsup_{t\to T}
			 	\left( 
					t^{\frac{d}2(\frac{\sigma_2}{d} - \frac{k_0}{d} - \frac1{ p_0} ) +\delta_2}  
					\|u_2(t)-z_2(t)\|_{L^{p_0, \infty}_{k_0}}
				\right)\\
	&\leq C M^{\alpha-1} 
			\limsup_{t\to T}
			 	\left( 
					t^{\frac{d}2(\frac{\sigma_2}{d} - \frac{k_0}{d} - \frac1{ p_0} ) +\delta_2}  
					\|u_2(t)-z_2(t)\|_{L^{p_0, \infty}_{k_0}}
				\right), 
\end{aligned}
\end{equation}
where
\begin{equation}	\label{t:asym.behv-complex:del2_II21}
	0<\delta_2<\delta_{21}' :=1-{d\over 2}\left({\sigma_2\over d}-{k_0\over d}-{1\over p_0}\right)
		-{d(\alpha-1)\over 2}\left({1\over q_c(\gamma)}-{k\over d}-{1\over p}\right),
\end{equation}
provided that the following conditions are satisfied:
\begin{equation}	\label{t:asym.behv-complex:II_{21}c1}
\begin{aligned}
	{\alpha-1\over p}+{1\over p_0} <1, \quad 
	\frac{\gamma}{\alpha-1} \leq k, \quad
	(\alpha-1) \left( {k \over d}+{1\over p}\right) + {k_0 \over d}+{1\over p_0} < 1. 
\end{aligned}
\end{equation}
The last time-integral is convergent if and only if 
\begin{equation}	\label{t:asym.behv-complex:II_{21}c2}
\begin{aligned}
	&(\alpha-1) \left( {k \over d}+{1\over p}\right)  <\frac{2+\gamma}{d}, \\
	&\frac{d}{2} \left( \frac{\sigma_2}{d}+\frac{2+\gamma}{d} -(\alpha-1) \left(\frac{k}{d}+\frac1{p} \right) 
			- \left(\frac{k_0}{d}+\frac1{p_0} \right) \right)+\delta_2<1. 
\end{aligned}
\end{equation}
The above conditions along with the existence of $\delta_{12}'$ are ensured by \eqref{l:exist.crt.c.pk} and \eqref{l:exist.crt.c.p0k0}.

\smallbreak
\noindent
\underline{\sf Estimate of $\mathbf{II}_{22}(t)$ for $\alpha>2.$}
By a similar calculation to that of $\mathbf{II}_{21}(t)$, we have for arbitrary $T>0$
\begin{equation}	\label{t:asym.behv-complex:pr.II_22}
\begin{aligned}
	& \limsup_{t\to T} \mathbf{II}_{22}(t)  \\ 
	&\leq C M^{\alpha-1} t^{-(\delta_1 -\delta_2)} 
			\limsup_{t\to T} 
			\left( 
				 t^{\frac{d}2(\frac{1}{q_c(\gamma)} - \frac{k}{d} - \frac1{ p} ) +\delta_1}  
				\|u_1 (t)-u_{1\mathcal S} (t)\|_{L^{p, \infty}_{k}}
			\right), 
\end{aligned}
\end{equation}
under \eqref{t:asym.behv-complex:II_{21}c1} and \eqref{t:asym.behv-complex:II_{21}c2}, where 
\begin{equation}	\label{t:asym.behv-complex:del1_II22}
	0<\delta_2<\delta_1<\delta_{21}'. 
\end{equation}
The estimate of $\mathbf{II}_{22}(t)$ for $j=1\mathcal{S}$ is similar to the above, and gives the same conditions.

\smallbreak
\noindent
\underline{\sf Estimate of ${\mathbf{II}_{23}}(t)$ for $1<\alpha\leq 2.$}
By a similar calculation to that of $\mathbf{II}_{21}(t)$, we have 
\begin{equation}	\label{t:asym.behv-complex:pr.II_23}
\begin{aligned}
	& \limsup_{t\to T} \mathbf{II}_{23}(t) \\ 
	&\leq C M^{\alpha-1} t^{-(\delta_1-\delta_2) }
		 \limsup_{t\to T} 
		\left( 
			t^{\frac{d}2(\frac{1}{q_c(\gamma)} - \frac{k}{d} - \frac1{p} ) +\delta_1} 
			\|u_1(t)-u_{1\mathcal S}(t)\|_{L^{ p,\infty}_{k}} 
		\right) 
\end{aligned}
\end{equation}
under \eqref{t:asym.behv-complex:II_{21}c1} and \eqref{t:asym.behv-complex:II_{21}c2}, where
$\delta_1$ and $\delta_2$ again satisfy \eqref{t:asym.behv-complex:del1_II22}. 

\subsection{Completion of the proof of Theorem \ref{t:asym.behv-complex}}

Having obtained $\theta_{11}$ and $\theta_{12}$, we may choose $\delta_1$ sufficiently small 
so that \eqref{t:asym.behv-complex:del1_L1}, \eqref{t:asym.behv-complex:del1_I11}, \eqref{t:asym.behv-complex:del1_I12}, 
\eqref{t:asym.behv-complex:del1_II1} and \eqref{t:asym.behv-complex:del1_II22} are satisfied. 
Similarly, having obtained $\theta_{21}$ and $\theta_{22}$, we may choose $\delta_2$ sufficiently small 
so that \eqref{t:asym.behv-complex:del2_L2}, \eqref{t:asym.behv-complex:del2_I21}, 
\eqref{t:asym.behv-complex:del2_I22} and \eqref{t:asym.behv-complex:del2_II21} are satisfied. 

For \underline{$1<\alpha\le2$}, \eqref{t:asym.behv-complex:pr.u1}, \eqref{t:asym.behv-complex:pr.L1}, 
\eqref{t:asym.behv-complex:pr.I_{12}}, \eqref{t:asym.behv-complex:pr.II_{1}}, 
\eqref{t:asym.behv-complex:pr.u2}, \eqref{t:asym.behv-complex:pr.L2}, 
\eqref{t:asym.behv-complex:pr.I_22}, \eqref{t:asym.behv-complex:pr.II_21} and \eqref{t:asym.behv-complex:pr.II_23} imply that 
given a finite $T>0$, we have the estimate  
\begin{equation}		\nonumber
\begin{aligned}
	&t^{\frac{d}{2}(\frac1{q_c(\gamma)} - \frac{k}{d} - \frac1{ p} )+\delta_1}
		\|u_1(t)-u_{1\mathcal S}(t)\|_{L^{ p,\infty}_{k}}
	+t^{\frac{d}{2}(\frac{\sigma_2}{d} - \frac{k_0}{d} - \frac1{ p_0} ) +\delta_2}
		\|u_2(t)-z_2(t)\|_{L^{ p_0,\infty}_{k_0}} \\
	&\leq \mathbf{L}_1(t)
		+ \mathbf{I}_{12}(t)
		+ {\mathbf{II}}_{1}(t)
		+ \mathbf{L}_2(t) 
		+ {\mathbf{I}}_{22}(t)
		+ {\mathbf{II}}_{21}(t)
		+ {\mathbf{II}}_{23}(t)\\
	&\leq C t^{-(\delta_{10}-\delta_1)} 
		+ C t^{-(\delta_{20}-\delta_2)} 
		+ CM^\alpha t^{-(\delta_{12}-\delta_1)} 
		+ C M^{\alpha} t^{-(\delta_{22}-\delta_2)} \\
	&\quad
		+ C M^{\alpha-1} 
			\limsup_{t\to T}\left(
				t^{\frac{d}2(\frac1{q_c(\gamma)} - \frac{k}{d} - \frac1{p} )+\delta_1} 
				\| u(t)- u_{1\mathcal S}(t)\|_{L^{p,\infty}_k}
			\right) \\
	&\quad 
		+ C M^{\alpha-1} t^{-(\delta_{12}' -\delta_2)}
			\limsup_{t\to T} \left( 
				 t^{\frac{d}2(\frac{\sigma_2}{d} - \frac{k_0}{d} - \frac1{ p_0} ) +\delta_2}  
				\|u_2 (t)-z_2 (t)\|_{L^{p_0, \infty}_{k_0}}
			\right) \\
	&\quad 
		+ C M^{\alpha-1}  t^{ - (\delta_1-\delta_2) } 
		 	\limsup_{t\to T}\left(  t^{\frac{d}2(\frac{1}{q_c(\gamma)} - \frac{k}{d} - \frac1{p} ) +\delta_1} 
		 				\|u_1 (t)-u_{1\mathcal S} (t)\|_{L^{ p,\infty}_{k}} \right). 
\end{aligned}
\end{equation}
Taking supremum over $T \ge t\ge t'$ on both sides of the above and letting $t'\to T$, we deduce 
\begin{equation}		\nonumber
\begin{aligned}
	&\limsup_{t\to T} \left(t^{\frac{d}{2}(\frac1{q_c(\gamma)} - \frac{k}{d} - \frac1{ p} )+\delta_1}
		\|u_1(t)-u_{1\mathcal S}(t)\|_{L^{ p,\infty}_{k}} \right) \\
	&\quad
	+\limsup_{t\to T} \left(t^{\frac{d}{2}(\frac{\sigma_2}{d} - \frac{k_0}{d} - \frac1{ p_0} ) +\delta_2}
		\|u_2(t)-z_2(t)\|_{L^{ p_0,\infty}_{k_0}} \right) \\
	&\leq C T^{-(\delta_{10}-\delta_1)} 
		+ C T^{-(\delta_{20}-\delta_2)} 
		+ C M^\alpha T^{-(\delta_{12}-\delta_1)} 
		+ C M^{\alpha} T^{-(\delta_{22}-\delta_2)} \\
	&\quad
		+ C M^{\alpha-1} \limsup_{t \to T} \left(t^{\frac{d}2(\frac1{q_c(\gamma)} - \frac{k}{d} - \frac1{p} )+\delta_1} 
				\| u(t)- u_{1\mathcal S}(t)\|_{L^{p,\infty}_k}\right) \\
	&\quad 
		+ C M^{\alpha-1} T^{-(\delta_{12}' -\delta_2)}
			\limsup_{t\to T} \left( 
				 t^{\frac{d}2(\frac{\sigma_2}{d} - \frac{k_0}{d} - \frac1{ p_0} ) +\delta_2}  
				\|u_2(t)-z_2(t)\|_{L^{p_0, \infty}_{k_0}}
			\right) \\
	&\quad 
		+ C M^{\alpha-1}  T^{ - (\frac{d(\alpha-1)}2 \delta_1-\delta_2) } 
		 	\limsup_{t\to T}\left( t^{\frac{d}2(\frac{1}{q_c(\gamma)} - \frac{k}{d} - \frac1{p} ) +\delta_1} 
		 				\|u_1(t)-u_{1\mathcal S}(t)\|_{L^{ p,\infty}_{k}} \right). 
\end{aligned}
\end{equation}
Note that we may assume that $T$ is sufficiently large, say $T>1$. 
Now by taking $M$ smaller if necessary, we have the closed estimate 
\begin{equation}		\nonumber
\begin{aligned}
	&\limsup_{t\to T} \left(t^{\frac{d}{2}(\frac1{q_c(\gamma)} - \frac{k}{d} - \frac1{ p} )+\delta_1}
		\|u_1(t)-u_{1\mathcal S}(t)\|_{L^{ p,\infty}_{k}} \right) \\
	&\quad
	+\limsup_{t\to T} \left(t^{\frac{d}{2}(\frac{\sigma_2}{d} - \frac{k_0}{d} - \frac1{ p_0} ) +\delta_2}
		\|u_2(t)-z_2(t)\|_{L^{ p_0,\infty}_{k_0}} \right) \\
	&\leq C T^{-(\delta_{10}-\delta_1)} 
		+ C T^{-(\delta_{20}-\delta_2)} 
		+ C M^\alpha T^{-(\delta_{12}-\delta_1)} 
		+ C M^{\alpha} T^{-(\delta_{22}-\delta_2)}. 
\end{aligned}
\end{equation}
As the constant $C$ does not depend on $T$, finally letting $T\to\infty,$ we conclude 
\begin{equation}	\label{t:asym.behv-complex:pr.conv}
\begin{aligned}
	&\limsup_{t\to \infty} \left(t^{\frac{d}{2}(\frac1{q_c(\gamma)} - \frac{k}{d} - \frac1{ p} )+\delta_1}
		\|u_1(t)-u_{1\mathcal S}(t)\|_{L^{ p,\infty}_{k}} \right) \\
	&\quad
	+\limsup_{t\to \infty} \left(t^{\frac{d}{2}(\frac{\sigma_2}{d} - \frac{k_0}{d} - \frac1{ p_0} ) +\delta_2}
		\|u_2(t)-z_2(t)\|_{L^{ p_0,\infty}_{k_0}} \right) =0 
\end{aligned}
\end{equation}
for $1<\alpha\le 2$. 

For \underline{$2<\alpha\le 3$}, \eqref{t:asym.behv-complex:pr.u1}, \eqref{t:asym.behv-complex:pr.L1}, 
\eqref{t:asym.behv-complex:pr.I_{11}}, \eqref{t:asym.behv-complex:pr.I_{12}},  
\eqref{t:asym.behv-complex:pr.II_{1}}, 
imply that given a finite $T>0$, we have the estimate 
\begin{equation}		\nonumber
\begin{aligned}
	&t^{\frac{d}{2}(\frac1{q_c(\gamma)} - \frac{k}{d} - \frac1{ p} )
		+\delta_1}\|u_1(t)-u_{1\mathcal S}(t)\|_{L^{ p,\infty}_{k}} \\
	&\leq \mathbf{L}_1(t)
		+ \mathbf{I}_{12}(t)
		+ {\mathbf{II}}_{1}(t) 
		+ \mathbf{L}_2(t) 
		+ {\mathbf{I}}_{22}(t)
		+ {\mathbf{II}}_{21}(t)
		+ {\mathbf{II}}_{22}(t)\\
	&\leq C t^{-(\delta_{10}-\delta_1)} 
		+ C t^{-(\delta_{20}-\delta_2)} 
		+ CM^\alpha t^{-(\delta_{12}-\delta_1)} 
		+ C M^{\alpha} t^{-(\delta_{22}-\delta_2)} \\
	&\quad
		+ C M^{\alpha-1} \limsup_{t\to T} \left(t^{\frac{d}2(\frac1{q_c(\gamma)} - \frac{k}{d} - \frac1{p} )+\delta_1} 
				\| u(t)- u_{1\mathcal S}(t)\|_{L^{p,\infty}_k}\right) \\
	&\quad 
		+ C M^{\alpha-1} t^{-(\delta_{12}' -\delta_2)}
			\limsup_{t\to T}\left( 
				 t^{\frac{d}2(\frac{\sigma_2}{d} - \frac{k_0}{d} - \frac1{ p_0} ) +\delta_2}  
				\|u_2(t)-z_2(t)\|_{L^{p_0, \infty}_{k_0}}
			\right) \\
	&\quad 
		+C M^{\alpha-1} t^{-(\delta_{21}' -\delta_2)}
			\limsup_{t\to T}\left(
				 t^{\frac{d}2(\frac{1}{q_c(\gamma)} - \frac{k}{d} - \frac1{ p} ) +\delta_1}  
				\|u_1(t)-u_{1\mathcal S}(t)\|_{L^{p, \infty}_{k}}
			\right) 
\end{aligned}
\end{equation}
for $t\in [0,T]$. By the same argument as above, we deduce \eqref{t:asym.behv-complex:pr.conv} for $2<\alpha\le 3$. 

For \underline{$\alpha> 3$}, \eqref{t:asym.behv-complex:pr.u1}, \eqref{t:asym.behv-complex:pr.L1}, 
\eqref{t:asym.behv-complex:pr.I_{11}}, \eqref{t:asym.behv-complex:pr.I_{12}},  
\eqref{t:asym.behv-complex:pr.II_{1}}, 
imply that given a finite $T>0$, we have the estimate 
\begin{equation}		\nonumber
\begin{aligned}
	&t^{\frac{d}{2}(\frac1{q_c(\gamma)} - \frac{k}{d} - \frac1{ p} )
		+\delta_1}\|u_1(t)-u_{1\mathcal S}(t)\|_{L^{ p,\infty}_{k}} \\
	&\leq \mathbf{L}_1(t)
		+\mathbf{I}_{11}(t) 
		+ \mathbf{I}_{12}(t)
		+ {\mathbf{II}}_{1}(t) 
		+ \mathbf{L}_2(t) 
		+ {\mathbf{I}}_{22}(t)
		+ {\mathbf{II}}_{21}(t)
		+ {\mathbf{II}}_{22}(t)\\
	&\leq C t^{-(\delta_{10}-\delta_1)} 
		+ C t^{-(\delta_{20}-\delta_2)} 
		+ CM^\alpha t^{-(\delta_{11}-\delta_1)}
		+ CM^\alpha t^{-(\delta_{12}-\delta_1)} 
		+ C M^{\alpha} t^{-(\delta_{22}-\delta_2)} \\
	&\quad
		+ C M^{\alpha-1} \limsup_{t\to T}  \left(t^{\frac{d}2(\frac1{q_c(\gamma)} - \frac{k}{d} - \frac1{p} )+\delta_1} 
				\| u(t)- u_{1\mathcal S}(t)\|_{L^{p,\infty}_k}\right) \\
	&\quad 
		+ C M^{\alpha-1} t^{-(\delta_{12}' -\delta_2)}
			\limsup_{t\to T} \left( 
				 t^{\frac{d}2(\frac{\sigma_2}{d} - \frac{k_0}{d} - \frac1{ p_0} ) +\delta_2}  
				\|u_2(t)-z_2(t)\|_{L^{p_0, \infty}_{k_0}}
			\right) \\
	&\quad 
		+C M^{\alpha-1} t^{-(\delta_{21}' -\delta_2)}
			\limsup_{t\to T}\left(
				 t^{\frac{d}2(\frac{1}{q_c(\gamma)} - \frac{k}{d} - \frac1{ p} ) +\delta_1}  
				\|u_1(t)-u_{1\mathcal S}(t)\|_{L^{p, \infty}_{k}}
			\right) 
\end{aligned}
\end{equation}
for $t\in [0,T]$. By the same argument as the case $1<\alpha\le 2$, we deduce \eqref{t:asym.behv-complex:pr.conv} for $\alpha> 3$. 

Gathering all the above estimates, we conclude
\begin{align*}
	\lim_{t\to\infty} t^{\frac{d}{2} (\frac{1}{q_c(\gamma)} 
				-\frac{ k}{d} - \frac{1}{ p})+\delta_1} 
		\|u_1(t)-u_{1\mathcal S}(t)\|_{L^{ p,\infty}_{ k}} =0 
\end{align*}
and
\begin{equation}
	\lim_{t\to\infty}t^{\frac{d}{2}(\frac{\sigma_2}{d} - \frac{k_0}{d} - \frac1{ p_0} ) +\delta_2}
		\|u_2(t)-z_2(t)\|_{L^{ p_0,\infty}_{k_0}} =0, 
\end{equation}
as desired. 

Lastly, we may prove \eqref{t:asym.behv.lin:c1complexi11} for all $(\tilde k_0, \tilde p_0)$ 
satisfying \eqref{l:exist.crt':c.tild.kp} and for all $(\tilde k,\tilde p)$ satisfying \eqref{t:asym.behv:c.kp} 
by the argument similar to the proof of \eqref{l:exist.crt:strng.dcy.est}. We omit the details. 
This completes the proof of Theorem \ref{t:asym.behv-complex}. 

\end{proof}

\section{Appendix}
\par

\subsection{Auxiliary lemmas}

\begin{lem}
	\label{l:asym.behv.lin}
There exists $\theta$ satisfying \eqref{t:asym.behv.lin:pr.key.est.c0}, \eqref{t:asym.behv.lin:pr.key.est.c1} and \eqref{t:asym.behv.lin:pr.key.est.c2}
under conditions \eqref{l:exist.crt.c.pk} and \eqref{l:exist.crt.c.p0k0}. 
\end{lem}

\begin{proof}
First note that we have $\gamma - (\alpha-1)k_0 \le \theta \alpha ( k - k_0 )$ if and only if 
\begin{equation}	\nonumber
\left\{\begin{aligned}
	&\frac{\gamma - (\alpha-1)k_0}{\alpha(k-k_0)} \le \theta && \text{if}\quad k_0 < k, \\
	&\theta \le \frac{(\alpha-1)k_0-\gamma}{\alpha(k_0-k)}  && \text{if}\quad k_0 > k. 
\end{aligned}\right.
\end{equation}
As 
\[
	1 - \frac1{\alpha} \le \frac{(\alpha-1)k_0-\gamma}{\alpha(k_0-k)}  \quad\text{if}\quad  k_0 > k,
\]
we only impose 
\begin{equation}\nonumber
	\frac{\gamma - (\alpha-1)k_0}{\alpha(k-k_0)} \le \theta  \quad\text{if}\quad  k_0 < k
\end{equation}
in the following. Similarly, we have $\theta ( \frac{1}{p} -\frac{1}{p_0} ) < \frac1{\alpha} -  \frac{1}{p_0}$ if and only if 
\begin{equation}	\nonumber
\left\{\begin{aligned}
	& \frac{ \frac{1}{p_0} - \frac1{\alpha}}{\frac{1}{p_0} -\frac{1}{p}}  < \theta &&\text{if}\quad \frac{1}{p} <\frac{1}{p_0}, \\
	&\theta < \frac{\frac1{\alpha} -  \frac{1}{p_0}}{\frac{1}{p} -\frac{1}{p_0}}  &&\text{if}\quad \frac{1}{p} >\frac{1}{p_0}. 
\end{aligned}\right.
\end{equation}
As 
\[
	1 - \frac1{\alpha} < \frac{\frac1{\alpha} -  \frac{1}{p_0}}{\frac{1}{p} -\frac{1}{p_0}} \quad\text{if}\quad \frac{1}{p} >\frac{1}{p_0}, 
\]
we only impose 
\begin{equation}\nonumber
	\frac{ \frac{1}{p_0} - \frac1{\alpha}}{\frac{1}{p_0} -\frac{1}{p}} < \theta \quad\text{if}\quad  \frac{1}{p} <\frac{1}{p_0}
\end{equation}
in the following. 

The case $\frac{k_0}{d} + \frac1{p_0} = \frac{k}{d} + \frac1{p}$ is easy, 
as most of \eqref{t:asym.behv.lin:pr.key.est.c1} and \eqref{t:asym.behv.lin:pr.key.est.c2} become trivial thanks to \eqref{l:exist.crt.c.pk} 
and \eqref{l:exist.crt.c.p0k0}. It suffices to take $\theta$ such that
\begin{equation}\nonumber
	 \frac{\frac{\sigma}{d}-\frac{2}{d\alpha} - \frac{k_0}{d} -\frac1{p_0} }
		{\frac{\sigma}{d}-\frac{1}{q_c(\gamma)} } < \theta <1-\frac1{\alpha}. 
\end{equation}

Therefore, we are left with the cases $\frac{k_0}{d} + \frac1{p_0} < \frac{k}{d} + \frac1{p}$ and 
$\frac{k_0}{d} + \frac1{p_0} > \frac{k}{d} + \frac1{p}$. Now, if 
\begin{equation}	\nonumber
	\frac{k_0}{d} + \frac1{p_0} < \frac{k}{d} + \frac1{p}, 
\end{equation}
then we may choose $\theta$ as follows:
\begin{equation}	\nonumber
\begin{aligned}
	& \max\left\{0, \frac{\frac{\gamma}{d} - (\alpha-1) \left(\frac{k_0}{d} + \frac1{p_0}\right) }
			{ \alpha \left(\frac{k}{d} + \frac1{p} - \frac{k_0}{d} - \frac1{p_0}\right) } \right\} \le \theta \\
	\text{and}\quad
	& \frac{\frac{\sigma}{d}-\frac{2}{d\alpha} - \frac{k_0}{d} -\frac1{p_0} }
		{\frac{\sigma}{d}-\frac{k_0}{d}-\frac1{p_0} - \left(\frac{1}{q_c(\gamma)}-\frac{k}{d}-\frac1{p} \right)} < \theta <1-\frac1{\alpha}. 
\end{aligned}
\end{equation}
	\smallbreak
If 
\begin{equation}	\label{t:asym.behv.lin:assmpt2}
	\frac{k_0}{d} + \frac1{p_0} > \frac{k}{d} + \frac1{p} 
\end{equation}
and $\frac{\sigma}{d}-\frac{k_0}{d}-\frac1{p_0} - \left(\frac{1}{q_c(\gamma)}-\frac{k}{d}-\frac1{p} \right) >0$, 
then \eqref{t:asym.behv.lin:pr.key.est.c0}, \eqref{t:asym.behv.lin:pr.key.est.c1} and \eqref{t:asym.behv.lin:pr.key.est.c2} amount to 
\begin{equation}	\nonumber
\begin{aligned}
		&0\le \theta 
		\quad\text{and}\quad
		\max\Bigg\{ 
			\frac{\alpha \left(\frac{k_0}{d} + \frac1{p_0}\right) - 1 - \frac{\gamma}{d} }
				{\alpha \left(\frac{k_0}{d} + \frac1{p_0} - \frac{k}{d} - \frac1{p} \right) }, 
			\frac{(\alpha-1) \left( \frac{k_0}{d} + \frac1{p_0} - \frac1{q_c(\gamma)} \right)}
				{ \alpha \left( \frac{k_0}{d} + \frac1{p_0} - \frac{k}{d} - \frac1{p} \right)}, \\
		&\hspace{+4cm}
			\frac{\frac{\sigma}{d} - \frac{2}{d\alpha}  -  \frac{k_0}{d} - \frac1{p_0} }
				{\frac{\sigma}{d}-\frac{k_0}{d}-\frac1{p_0} - \left(\frac{1}{q_c(\gamma)}-\frac{k}{d}-\frac1{p} \right)}
		\Bigg\}
		< \theta < 1 - \frac1{\alpha}. 
\end{aligned}
\end{equation}
On the other hand, if \eqref{t:asym.behv.lin:assmpt2} and 
$\frac{\sigma}{d}-\frac{k_0}{d}-\frac1{p_0} - \left(\frac{1}{q_c(\gamma)}-\frac{k}{d}-\frac1{p} \right) < 0$ hold, 
then \eqref{t:asym.behv.lin:pr.key.est.c0}, \eqref{t:asym.behv.lin:pr.key.est.c1} and \eqref{t:asym.behv.lin:pr.key.est.c2} amount to 
\begin{equation}	\nonumber
\begin{aligned}
		&0\le \theta
			\ \text{and}\ 
		\max\Bigg\{ 
			\frac{\alpha \left(\frac{k_0}{d} + \frac1{p_0}\right) - 1 - \frac{\gamma}{d} }
				{\alpha \left(\frac{k_0}{d} + \frac1{p_0} - \frac{k}{d} - \frac1{p} \right) }, 
			\frac{(\alpha-1) \left( \frac{k_0}{d} + \frac1{p_0} - \frac1{q_c(\gamma)} \right)}
				{ \alpha \left( \frac{k_0}{d} + \frac1{p_0} - \frac{k}{d} - \frac1{p} \right)} \Bigg\} 
			< \theta 
			< 1 - \frac1{\alpha}. 
\end{aligned}
\end{equation}
In any case, the existence of $\theta$ satisfying 
\eqref{t:asym.behv.lin:pr.key.est.c0}, \eqref{t:asym.behv.lin:pr.key.est.c1} and \eqref{t:asym.behv.lin:pr.key.est.c2} 
is ensured by \eqref{l:exist.crt.c.pk} and \eqref{l:exist.crt.c.p0k0}. 
\end{proof}

\begin{lem}
	\label{l:theta11}
There exists $\theta_{11}$ satisfying \eqref{t:asym.behv-complex:del1_I11}, 
\eqref{t:asym.behv-complex:I_{11}c1} and \eqref{t:asym.behv-complex:I_{11}c2} under conditions 
\eqref{l:exist.crt.c.pk}, \eqref{l:exist.crt.c.p0k0} and \eqref{t:asym.behv-complex:c.k0p0}.
\end{lem}

\begin{proof}
We first note that 
$\gamma \leq (\alpha-3)k +2 k_{0}  +2{\theta_{11}} (k-k_0)$ is equivalent to 
\begin{equation}	\nonumber
\left\{\begin{aligned}
	&\frac{\gamma -(\alpha-3)k -2 k_{0}  }{2(k-k_0)} \le \theta_{11} && \text{if}\quad k_0 < k, \\
	&\theta_{11} \le \frac{ (\alpha-3)k +2 k_{0}-\gamma}{2(k_0-k)}  && \text{if}\quad k_0 > k. 
\end{aligned}\right.
\end{equation}
As 
\[
	1 \le \frac{ (\alpha-3)k +2 k_{0}-\gamma}{2(k_0-k)} \quad\text{if}\quad k_0 > k,
\]
we only impose 
\begin{equation}\nonumber
	\frac{\gamma -(\alpha-3)k -2 k_{0}  }{2(k-k_0)} \le \theta_{11} \quad\text{if}\quad k_0 < k
\end{equation}
in the following. Moreover, ${\alpha-2\over p}+{2\over p_0} + 2 \theta_{11} \left({1\over p} - {1\over p_0}\right)<1$ is equivalent to 
\begin{equation}	\nonumber
\left\{\begin{aligned}
	& \frac{ {\alpha-2\over p}+{2\over p_0} - 1}{2(\frac{1}{p_0} -\frac{1}{p})}  < \theta_{11} &&\text{if}\quad \frac{1}{p} <\frac{1}{p_0}, \\
	&\theta_{11} < \frac{1- {\alpha-2\over p}-{2\over p_0}}{2(\frac{1}{p} -\frac{1}{p_0})}  &&\text{if}\quad \frac{1}{p} >\frac{1}{p_0}. 
\end{aligned}\right.
\end{equation}
As 
\[
	1<  \frac{1- {\alpha-2\over p}-{2\over p_0}}{2(\frac{1}{p} -\frac{1}{p_0})} \quad\text{if}\quad\frac{1}{p} >\frac{1}{p_0}, 
\]
we only impose 
\begin{equation}\nonumber
	 \frac{ {\alpha-2\over p}+{2\over p_0} - 1}{2(\frac{1}{p_0} -\frac{1}{p})}  < \theta_{11}
	\quad\text{if}\quad \frac{1}{p} <\frac{1}{p_0}
\end{equation}
in the following. 

When $\frac{k_0}{d} + \frac1{p_0} = \frac{k}{d} + \frac1{p}$, 
the last condition of \eqref{t:asym.behv-complex:I_{11}c2} amounts to 
\begin{equation}\nonumber
	\frac1{q_c(\gamma)} - \frac2{d} + (\alpha-1) \frac{\sigma_2}{d} - \alpha \left(\frac{k}{d} + \frac{1}{p} \right)
	<\theta_{12} (\alpha-1) \left( \frac{\sigma_2}{d} - \frac{1}{q_c(\gamma)}\right),
\end{equation}
while all other conditions in \eqref{t:asym.behv-complex:I_{12}c1} and \eqref{t:asym.behv-complex:I_{12}c2} 
are trivial by \eqref{l:exist.crt.c.pk} and \eqref{l:exist.crt.c.p0k0}. Thus, it suffices to take 
\begin{equation}\nonumber
	0\le  \theta_{11} 
	\quad\text{and}\quad
	\frac{\frac1{q_c(\gamma)} - \frac2{d} + (\alpha-1) \frac{\sigma_2}{d} - \alpha \left(\frac{k}{d} + \frac{1}{p} \right)}
		{(\alpha-1) \left( \frac{\sigma_2}{d} - \frac{1}{q_c(\gamma)}\right)}
	< \theta_{11} < 1. 
\end{equation}

When $\frac{k}{d} + \frac1{p}>\frac{k_0}{d} + \frac1{p_0}$, it suffices to take $\theta_{11}$ such that 
\begin{equation}	\nonumber
\begin{aligned}
	&\max\left\{0,  \frac{ {\gamma\over d} -  (\alpha-3) \left({k \over d} + {1\over p }\right) - 2 \left({k_0 \over d} + {1\over p_0 }\right)  } 
		{2 \left(\frac{k}{d} + \frac{1}{p} - \frac{k_0}{d} - \frac{1}{p_0} \right)} \right\}
	\le \theta_{11} \\
	\text{and}\quad&
	\frac{(\alpha-2) \left(\frac{1}{q_c(\gamma)} - \frac{k}{d} - \frac{1}{p}\right)
			+2 \left(\frac{\sigma_2}{d} - \frac{k_0}{d} - \frac{1}{p_0} \right) - \frac2{d}}
		{2 \left(\frac{\sigma_2}{d} - \frac{k_0}{d} - \frac{1}{p_0}  - \left(\frac{1}{q_c(\gamma)} - \frac{k}{d} - \frac{1}{p}\right) \right)}  
	< \theta_{11} 
	<1. 
\end{aligned}
\end{equation}
When 
\begin{equation}\nonumber
	\frac{k}{d} + \frac1{p}<\frac{k_0}{d} + \frac1{p_0}
		\quad\text{and}\quad
	\frac{1}{q_c(\gamma)} + \frac{k}{d} + \frac{1}{p} > \frac{\sigma_2}{d} - \frac{k_0}{d} - \frac{1}{p_0}, 
\end{equation}
it suffices to take $\theta_{11}$ such that 
\begin{equation}	\nonumber
\begin{aligned}
	&0\le  \theta_{11} \quad\text{and} \quad 
	\max\left\{
	\frac{ (\alpha-2)  \left({k \over d} + {1\over p }\right) + 2 \left({k_0 \over d} + {1\over p_0 }\right)- {\gamma\over d} - 1}
		{ 2 \left( \frac{k_0}{d} + \frac{1}{p_0} - \frac{k}{d} - \frac{1}{p}\right)},  \right. \\
	&\hspace{5cm}
	\left. \frac{ (\alpha-3)  \left({k \over d} + {1\over p }\right) + 2 \left({k_0 \over d} + {1\over p_0 }\right) - \frac{2+\gamma}{d} }
		{ {k_0\over d} + {1\over p_0} -  {k \over d} - {1\over p} } 
	\right\}
		<\theta_{11}<1. 
\end{aligned}
\end{equation}
under the extra assumption \eqref{t:asym.behv-complex:c.k0p0}: 
\begin{equation}\nonumber
	(\alpha-2) \left({k \over d} + {1\over p }\right) +  \left({k_0 \over d} + {1\over p_0 }\right) < \frac{2+\gamma}{d} 
	\quad\text{if}\quad
	\frac{k}{d} + \frac1{p}<\frac{k_0}{d} + \frac1{p_0}. 
\end{equation}
Lastly, when 
\begin{equation}\nonumber
	\frac{k}{d} + \frac1{p}<\frac{k_0}{d} + \frac1{p_0}
		\quad\text{and}\quad
	\frac{1}{q_c(\gamma)} + \frac{k}{d} + \frac{1}{p} < \frac{\sigma_2}{d} - \frac{k_0}{d} - \frac{1}{p_0}, 
\end{equation}
it suffices to take $\theta_{11}$ such that 
\begin{equation}	\nonumber
\begin{aligned}
	&0\le  \theta_{11} 
	\quad\text{and} \quad \\
	&\max\left\{
	\frac{ (\alpha-2)  ({k \over d} + {1\over p }) + 2 ({k_0 \over d} + {1\over p_0 })- {\gamma\over d} - 1}
		{ 2 ( \frac{k_0}{d} + \frac{1}{p_0} - \frac{k}{d} - \frac{1}{p})}, \, 
	\frac{ (\alpha-3)  \left({k \over d} + {1\over p }\right) + 2 ({k_0 \over d} + {1\over p_0 }) - \frac{2+\gamma}{d} }
		{ {k_0\over d} + {1\over p_0} -  {k \over d} - {1\over p} } 	\right.  \\
	&\left. 
	\hspace{4cm}
	\frac{(\alpha-2) (\frac{1}{q_c(\gamma)} - \frac{k}{d} - \frac{1}{p})
			+2 (\frac{\sigma_2}{d} - \frac{k_0}{d} - \frac{1}{p_0} ) - \frac2{d}}
		{\frac{\sigma_2}{d} - \frac{k_0}{d} - \frac{1}{p_0}  - ( \frac{1}{q_c(\gamma)} - \frac{k}{d} - \frac{1}{p}) } 
	\right\} 
	<\theta_{11}<1. 
\end{aligned}
\end{equation}
with \eqref{t:asym.behv-complex:c.k0p0}. 
In all cases listed above, the existence of $\theta_{11}$ is ensured by \eqref{l:exist.crt.c.pk}, \eqref{l:exist.crt.c.p0k0} 
and \eqref{t:asym.behv-complex:c.k0p0}.
\end{proof}

\begin{lem}
	\label{l:theta12}
There exists $\theta_{12}$ satisfying \eqref{t:asym.behv-complex:del1_I12}, 
\eqref {t:asym.behv-complex:I_{12}c1} and \eqref {t:asym.behv-complex:I_{12}c2} under conditions 
\eqref{l:exist.crt.c.pk} and \eqref{l:exist.crt.c.p0k0}.
\end{lem}

\begin{proof}
We first note that 
${\gamma \over \alpha-1}-k_0 \leq {\theta_{12}}(k-k_{0})$ is equivalent to 
\begin{equation}	\nonumber
\left\{\begin{aligned}
	&\frac{\gamma - (\alpha-1)k_0}{(\alpha-1)(k-k_0)} \le \theta_{12} && \text{if}\quad k_0 < k, \\
	&\theta_{12} \le \frac{(\alpha-1)k_0-\gamma}{(\alpha-1)(k_0-k)}  && \text{if}\quad k_0 > k. 
\end{aligned}\right.
\end{equation}
As 
\[
	1 \le  \frac{(\alpha-1)k_0-\gamma}{(\alpha-1)(k_0-k)} \quad\text{if}\quad k_0 > k,
\]
we only impose 
\begin{equation}\nonumber
	\frac{\gamma - (\alpha-1)k_0}{(\alpha-1)(k-k_0)}  \le \theta_{12} \quad\text{if}\quad k_0 < k
\end{equation}
in the following. Moreover,  $\theta_{12} ({1\over p} - {1\over p_0} ) < \frac1{\alpha-1} (1-{1\over p} )-{1\over p_0}$ is equivalent to 
\begin{equation}	\nonumber
\left\{\begin{aligned}
	& \frac{ \frac{1}{p_0} - \frac1{\alpha-1} (1-{1\over p} )}{\frac{1}{p_0} -\frac{1}{p}}  < \theta_{12} &&\text{if}\quad \frac{1}{p} <\frac{1}{p_0}, \\
	&\theta_{12} < \frac{\frac1{\alpha-1} (1-{1\over p} ) -  \frac{1}{p_0}}{\frac{1}{p} -\frac{1}{p_0}}  &&\text{if}\quad \frac{1}{p} >\frac{1}{p_0}. 
\end{aligned}\right.
\end{equation}
As 
\[
	1< \frac{\frac1{\alpha-1} (1-{1\over p} ) -  \frac{1}{p_0}}{\frac{1}{p} -\frac{1}{p_0}} \quad\text{if}\quad\frac{1}{p} >\frac{1}{p_0}, 
\]
we only impose 
\begin{equation}\nonumber
	\frac{ \frac{1}{p_0} - \frac1{\alpha-1} (1-{1\over p} )}{\frac{1}{p_0} -\frac{1}{p}} <  \theta_{12} 
	\quad\text{if}\quad \frac{1}{p} <\frac{1}{p_0}
\end{equation}
in the following. 

When $\frac{k_0}{d} + \frac1{p_0} = \frac{k}{d} + \frac1{p}$, 
the last condition of \eqref{t:asym.behv-complex:I_{12}c2} amounts to 
\begin{equation}\nonumber
	\frac1{q_c(\gamma)} - \frac2{d} + (\alpha-1) \frac{\sigma_2}{d} - \alpha \left(\frac{k}{d} + \frac{1}{p} \right)
	<\theta_{12} (\alpha-1) \left( \frac{\sigma_2}{d} - \frac{1}{q_c(\gamma)}\right),
\end{equation}
while all other conditions in \eqref{t:asym.behv-complex:I_{12}c1} and \eqref{t:asym.behv-complex:I_{12}c2} 
are trivial by \eqref{l:exist.crt.c.pk} and \eqref{l:exist.crt.c.p0k0}. Thus, it suffices to take 
\begin{equation}\nonumber
	0\le  \theta_{12} 
	\quad\text{and}\quad
	\frac{\frac1{q_c(\gamma)} - \frac2{d} + (\alpha-1) \frac{\sigma_2}{d} - \alpha \left(\frac{k}{d} + \frac{1}{p} \right)}
		{(\alpha-1) \left( \frac{\sigma_2}{d} - \frac{1}{q_c(\gamma)}\right)}
	< \theta_{12} < 1. 
\end{equation}
When $\frac{k}{d} + \frac1{p}>\frac{k_0}{d} + \frac1{p_0}$, it suffices to take $\theta_{12}$ such that 
\begin{equation}	\nonumber
\begin{aligned}
	&\max\left\{0,  \frac{ {\gamma\over d} - (\alpha-1)\left(\frac{k_0}{d} + \frac{1}{p_0} \right) } 
		{(\alpha-1) \left(\frac{k}{d} + \frac{1}{p} - \frac{k_0}{d} - \frac{1}{p_0} \right)} \right\}
	\le \theta_{12} \\
	\text{and}\quad&
	\frac{\frac{1}{q_c(\gamma)}- \frac2{d}- \frac{k}{d} - \frac{1}{p} + (\alpha-1) \left(\frac{\sigma_2}{d} - \frac{k_0}{d} - \frac{1}{p_0} \right)}
		{(\alpha-1) \left(\frac{\sigma_2}{d} - \frac{k_0}{d} - \frac{1}{p_0}  - \left(\frac{1}{q_c(\gamma)} - \frac{k}{d} - \frac{1}{p}\right) \right)}  
	< \theta_{12} 
	<1. 
\end{aligned}
\end{equation}
When 
\begin{equation}\nonumber
	\frac{k}{d} + \frac1{p}<\frac{k_0}{d} + \frac1{p_0}
		\quad\text{and}\quad
	\frac{1}{q_c(\gamma)} + \frac{k}{d} + \frac{1}{p} > \frac{\sigma_2}{d} - \frac{k_0}{d} - \frac{1}{p_0}, 
\end{equation}
it suffices to take $\theta_{12}$ such that 
\begin{equation}	\nonumber
\begin{aligned}
	&0\le  \theta_{12} \\
	\text{and} \quad
	&\max\left\{
	\frac{\frac{k}{d}+\frac{1}{p} +(\alpha-1)\left(\frac{k_0}{d} + \frac{1}{p_0} \right) - {\gamma\over d} - 1}
		{ (\alpha-1) \left( \frac{k_0}{d} + \frac{1}{p_0} - \frac{k}{d} - \frac{1}{p}\right)}, 
	 \frac{{k_0\over d} + {1\over p_0} - \frac1{q_c(\gamma)} }
		{ {k_0\over d} + {1\over p_0} -  {k \over d} - {1\over p} } 
	\right\}
		<\theta_{12}<1. 
\end{aligned}
\end{equation}
Lastly, when 
\begin{equation}\nonumber
	\frac{k}{d} + \frac1{p}<\frac{k_0}{d} + \frac1{p_0}
		\quad\text{and}\quad
	\frac{1}{q_c(\gamma)} + \frac{k}{d} + \frac{1}{p} < \frac{\sigma_2}{d} - \frac{k_0}{d} - \frac{1}{p_0}, 
\end{equation}
it suffices to take $\theta_{12}$ such that 
\begin{equation}	\nonumber
\begin{aligned}
	&0\le  \theta_{12} 
	\quad\text{and} \quad
	\max\left\{
	\frac{\frac{k}{d}+\frac{1}{p} +(\alpha-1)\left(\frac{k_0}{d} + \frac{1}{p_0} \right) - {\gamma\over d} - 1}
		{ (\alpha-1) \left( \frac{k_0}{d} + \frac{1}{p_0} - \frac{k}{d} - \frac{1}{p}\right)}, \,
	\frac{{k_0\over d} + {1\over p_0} - \frac1{q_c(\gamma)} }
		{ {k_0\over d} + {1\over p_0} -  {k \over d} - {1\over p} }, \, \right. \\
	&\hspace{4cm}
	\left. \frac{ \frac{1}{q_c(\gamma)} - \frac{k}{d} - \frac{1}{p} + (\alpha-1) \left(\frac{\sigma_2}{d} - \frac{k_0}{d} - \frac{1}{p_0} \right) -\frac2{d} }
		{(\alpha-1) \left(\frac{\sigma_2}{d} - \frac{k_0}{d} - \frac{1}{p_0}- \left(\frac{1}{q_c(\gamma)} - \frac{k}{d} - \frac{1}{p} \right) \right) } 
	\right\}
		<\theta_{12}<1. 
\end{aligned}
\end{equation}
In all cases listed above, the existence of $\theta_{12}$ is ensured by \eqref{l:exist.crt.c.pk} and \eqref{l:exist.crt.c.p0k0}.
\end{proof}

\begin{lem}
	\label{l:theta21}
There exists $\theta_{21}$ satisfying \eqref{t:asym.behv-complex:del2_I21}, 
\eqref{t:asym.behv-complex:I_{21}c1} and \eqref{t:asym.behv-complex:I_{21}c2} under conditions
\eqref{l:exist.crt.c.pk} and \eqref{l:exist.crt.c.p0k0}.
\end{lem}

\begin{proof}
We first note that 
$\gamma\leq (\alpha-3)k+2k_0 + 3{\theta_{21}} (k-k_0)$ is equivalent to 
\begin{equation}	\nonumber
\left\{\begin{aligned}
	&\frac{\gamma -(\alpha-3)k -2 k_{0}  }{3(k-k_0)} \le \theta_{21} && \text{if}\quad k_0 < k, \\
	&\theta_{21} \le \frac{ (\alpha-3)k +2 k_{0}-\gamma}{3(k_0-k)}  && \text{if}\quad k_0 > k. 
\end{aligned}\right.
\end{equation}
As 
\[
	{2\over 3} \le \frac{ (\alpha-3)k +2 k_{0}-\gamma}{3(k_0-k)} \quad\text{if}\quad k_0 > k,
\]
we only impose 
\begin{equation}\nonumber
	\frac{\gamma -(\alpha-3)k -2 k_{0}  }{3(k-k_0)} \le \theta_{21} \quad\text{if}\quad k_0 < k
\end{equation}
in the following. Moreover, ${\alpha-3\over p} + {3\over p_0}+3\theta_{21}\left( {1 \over p}-{1\over p_0} \right) <1$ is equivalent to 
\begin{equation}	\nonumber
\left\{\begin{aligned}
	& \frac{ {\alpha-3\over p}+{3\over p_0} - 1}{3(\frac{1}{p_0} -\frac{1}{p})}  < \theta_{21} &&\text{if}\quad \frac{1}{p} <\frac{1}{p_0}, \\
	&\theta_{21} < \frac{1- {\alpha-3\over p}-{3\over p_0}}{3(\frac{1}{p} -\frac{1}{p_0})}  &&\text{if}\quad \frac{1}{p} >\frac{1}{p_0}. 
\end{aligned}\right.
\end{equation}
As 
\[
	{2\over 3}< \frac{1- {\alpha-3\over p}-{3\over p_0}}{3(\frac{1}{p} -\frac{1}{p_0})}  \quad\text{if}\quad\frac{1}{p} >\frac{1}{p_0}, 
\]
we only impose 
\begin{equation}\nonumber
	\frac{ {\alpha-3\over p}+{3\over p_0} - 1}{3(\frac{1}{p_0} -\frac{1}{p})}   < \theta_{21}
	\quad\text{if}\quad \frac{1}{p} <\frac{1}{p_0}
\end{equation}
in the following. 

When $\frac{k_0}{d} + \frac1{p_0} = \frac{k}{d} + \frac1{p}$, 
the last condition of \eqref{t:asym.behv-complex:I_{21}c2} amounts to 
\begin{equation}\nonumber
	\frac{\alpha-3}{q_c(\gamma)} -\frac2{d}
			+ \frac{3\sigma_2}{d} - \alpha \left(\frac{k}{d} + \frac{1}{p} \right)
	<3{\theta_{21}} \left(\frac{\sigma_2}{d} -\frac{1}{q_c(\gamma)} \right).
\end{equation}
while all other conditions in \eqref{t:asym.behv-complex:I_{21}c1} and \eqref{t:asym.behv-complex:I_{21}c2} 
are trivial by \eqref{l:exist.crt.c.pk} and \eqref{l:exist.crt.c.p0k0}. Thus, it suffices to take 
\begin{equation}\nonumber
	0\le  \theta_{21} 
	\quad\text{and}\quad
	\frac{\frac{\alpha-3}{q_c(\gamma)} -\frac2{d}+ \frac{3\sigma_2}{d} - \alpha \left(\frac{k}{d} + \frac{1}{p} \right)}
		{3 \left( \frac{\sigma_2}{d} - \frac{1}{q_c(\gamma)}\right)}
	< \theta_{21} < {2\over 3}. 
\end{equation}

When $\frac{k}{d} + \frac1{p}>\frac{k_0}{d} + \frac1{p_0}$, it suffices to take $\theta_{21}$ such that 
\begin{equation}	\nonumber
\begin{aligned}
	&\max\left\{0,  \frac{ {\gamma\over d} -(\alpha-3) \left( {k\over d} + {1\over p} \right)- 2\left( {k_0\over d} + {1\over p_0} \right) } 
		{3 \left(\frac{k}{d} + \frac{1}{p} - \frac{k_0}{d} - \frac{1}{p_0} \right)} \right\}
	\le \theta_{12} \\
	\text{and}\quad&
	\frac{(\alpha-3) \left(\frac{1}{q_c(\gamma)} - \frac{k}{d} - \frac{1}{p}\right) 
			+ 3\left(\frac{\sigma_2}{d} - \frac{k_0}{d} - \frac{1}{p_0}\right)- \frac2{d}}
		{3\left(\frac{\sigma_2}{d} - \frac{k_0}{d} - \frac{1}{p_0}  - \left(\frac{1}{q_c(\gamma)} - \frac{k}{d} - \frac{1}{p}\right) \right)}  
	< \theta_{12} 
	< {2\over 3}. 
\end{aligned}
\end{equation}

When 
\begin{equation}\nonumber
	\frac{k}{d} + \frac1{p}<\frac{k_0}{d} + \frac1{p_0}
		\quad\text{and}\quad
	\frac{1}{q_c(\gamma)} + \frac{k}{d} + \frac{1}{p} > \frac{\sigma_2}{d} - \frac{k_0}{d} - \frac{1}{p_0}, 
\end{equation}
it suffices to take $\theta_{21}$ such that 
\begin{equation}	\nonumber
\begin{aligned}
	&0\le  \theta_{21} \quad\text{and} \quad\\
	&\max\left\{
	\frac{(\alpha-3) ( {k\over d} + {1\over p} )+ 3( {k_0\over d} + {1\over p_0} ) - {d+\gamma\over d} }
		{ 3 ( \frac{k_0}{d} + \frac{1}{p_0} - \frac{k}{d} - \frac{1}{p})}, 
	 \frac{(\alpha-3)( {k\over d} + {1\over p}) + 2( {k_0\over d} + {1\over p_0})-\frac{2+\gamma}{d}}
		{ 3 ( \frac{k_0}{d} + \frac{1}{p_0} - \frac{k}{d} - \frac{1}{p}) } 
	\right\} \\
	&\hspace{3cm}	<\theta_{21}<{2\over 3}. 
\end{aligned}
\end{equation}

Lastly, when 
\begin{equation}\nonumber
	\frac{k}{d} + \frac1{p}<\frac{k_0}{d} + \frac1{p_0}
		\quad\text{and}\quad
	\frac{1}{q_c(\gamma)} + \frac{k}{d} + \frac{1}{p} < \frac{\sigma_2}{d} - \frac{k_0}{d} - \frac{1}{p_0}, 
\end{equation}
it suffices to take $\theta_{21}$ such that 
\begin{equation}	\nonumber
\begin{aligned}
	&0\le  \theta_{21} 
	\quad\text{and} \quad\\
	&\max\left\{
\frac{(\alpha-3) ( {k\over d} + {1\over p} )+ 3( {k_0\over d} + {1\over p_0} ) - {d+\gamma\over d} }
		{ 3 ( \frac{k_0}{d} + \frac{1}{p_0} - \frac{k}{d} - \frac{1}{p})}, \,
	 \frac{(\alpha-3)( {k\over d} + {1\over p}) + 2( {k_0\over d} + {1\over p_0})-\frac{2+\gamma}{d}}
		{ 3 ( \frac{k_0}{d} + \frac{1}{p_0} - \frac{k}{d} - \frac{1}{p}) }  \, \right. \\
	&\hspace{4cm}
	\left. \frac{ (\alpha-3) \left(\frac{1}{q_c(\gamma)} - \frac{k}{d} - \frac{1}{p}\right) 
			+ 3\left(\frac{\sigma_2}{d} - \frac{k_0}{d} - \frac{1}{p_0}\right)-\frac2{d} }
		{3 \left(\frac{\sigma_2}{d} - \frac{k_0}{d} - \frac{1}{p_0}- \left(\frac{1}{q_c(\gamma)} - \frac{k}{d} - \frac{1}{p} \right) \right) } 
	\right\}
		<\theta_{21}<{2\over 3}.  
\end{aligned}
\end{equation}
In all cases listed above, the existence of $\theta_{21}$ is ensured by \eqref{l:exist.crt.c.pk} and \eqref{l:exist.crt.c.p0k0}.
\end{proof}

\subsection{Subcritical local theory}

The following is dedicated to the local well-posedness of \eqref{HHNLH} 
in the subcritical space $L^{q,r}_{l}(\R^d)$ for $(l,q)$ satisfying \eqref{d:subcritical}. 

\begin{thm}[Local well-posedness in the subcritical space]
	\label{t:HH.LWP.sub}
Let $d\in\mathbb{N},$ $\gamma\in\R$ and $\alpha\in\R$ satisfy 
\begin{equation}	\label{t:HH.LWP.sub.c.paramet}
	\gamma> -\min(2,d)
		\quad\text{and}\quad
	\alpha> \max\left\{1, \, 1+\frac{\gamma}{d} \right\}. 
\end{equation}
Let $r \in [1,\infty]$ and let $l\in\R$ and $q\in [1,\infty]$ satisfy 
\begin{equation}	\label{t:HH.LWP.sub.c.ql}
	\frac{\gamma}{\alpha-1} \le l
	\quad\text{and}\quad
	0\le \frac{l}{d} + \frac{1}{q} 
		< \min\left\{1, \frac1{q_c(\gamma)} \right\}. 
\end{equation}
Then the Cauchy problem \eqref{HHNLH} is locally well-posed in $L^{q,r}_{l}(\R^d)$ for arbitrary data. 
More precisely, the following assertions hold. 
\begin{itemize}
\item[$(\rm{i})$] {\rm (}Existence{\rm )} 
	For any $\varphi \in L^{q,r}_{l}(\R^d),$ there exist a positive number 
	$T$ depending only on $\|\varphi\|_{L^{q,r}_{l}}$ and an 
	$L^{q,r}_{l}(\R^d)$-mild solution $u$ to \eqref{intHHNLH} satisfying 
	\begin{equation}\nonumber	
		\|u\|_{\mathcal{K}^{p}_{k}(T)} 
			\le 2 \|e^{t\Delta} \varphi \|_{\mathcal{K}^{p}_{k}(T)},
	\end{equation}
	where $k\in\R$ and $p\in[1,\infty]$ satisfy 
	\begin{equation}	\label{t:HH.LWP.sub.c.kp}
	\begin{aligned}
		&\alpha < p \le \infty, \quad
			\frac{l+\gamma}{\alpha} \le k \le l, \\
		&0 \le \frac{k}{d} + \frac{1}{p} 
			\quad\text{and}\quad
		\frac1{\alpha}\left( \frac{l}{d} + \frac1{q} + \frac{\gamma}{d} \right)
			<  \frac{k}{d} +\frac1{p}
			< \min\left\{ \frac{l}{d} + \frac{1}{q}, \, \frac{d+\gamma}{d\alpha} \right\}. 
	\end{aligned}
	\end{equation}
	Moreover, the solution can be extended to the maximal interval 
	$[0,T_m),$ where $T_m$ is defined by \eqref{d:Tm}. 
\item[$(\rm{ii})$] {\rm (}Uniqueness in ${\mathcal{K}}^{p}_{k}(T)${\rm )} 
	Let $T>0.$ If $u, v \in {\mathcal{K}}^{p}_{k}(T)$ satisfy 
	\eqref{HHNLH} with $u(0) = v(0)=\varphi,$ then $u=v$ on $[0,T].$ 
\item[$(\rm{iii})$] {\rm (}Continuous dependence on initial data{\rm )}  
	For any initial data $\varphi$ and $\psi$ in $L^{q,r}_{l}(\R^d),$ 
	let $T(\varphi)$ and $T(\psi)$ be the corresponding existence time given by $(\rm{i})$. 
	Then there exists a constant $C$ depending on $\varphi$ and $\psi$ such that 
	the corresponding solutions $u$ and $v$ satisfy 
	\begin{equation}\nonumber	
		\|u-v\|_{L^\infty(0,T;L^{q,r}_{l}) \cap {\mathcal{K}}^{p}_{k}(T)} 
		\le C T^{\frac{d(\alpha-1)}{2}(\frac1{q_c(\gamma)} 
						- \frac{l}{d}-\frac{1}{q})} 
						\|\varphi-\psi\|_{L^{q,r}_{l}}
	\end{equation}
	for $T< \min\{T(\varphi), T(\psi)\}.$ 
\item[$(\rm{iv})$] {\rm (}Blow-up criterion{\rm )} If $T_m<\infty,$ 
	then $\displaystyle \lim_{t\rightarrow T_m-0}\|u(t)\|_{L^{q,r}_{l}}=\infty.$ 
	Moreover, the following lower bound of blow-up rate holds: 
	there exists a positive constant $C$ independent of $t$ such that 
	\begin{equation}\label{t:HH.LWP:Tm}
		\|u(t)\|_{L^{q,r}_{l}} 
		\ge \frac{C}{(T_m - t)^{\frac{d(\alpha-1)}{2}(\frac1{q_c(\gamma)} 
						- \frac{l}{d}-\frac{1}{q})} }
	\end{equation}
	for $t\in (0,T_m)$.
\end{itemize}
\end{thm}

The above theorem complements the corresponding result in \cite{CIT2022}. 
The proof is carried out in a similar way as that of Theorem \ref{t:HH.LWP}. 
For the details of the proof, we refer to \cites{BenTayWei2017, CIT2022} and the references therein.

\subsection{Nonexistence of weak positive solutions}
For $q\in [1,\infty]$ and $l\in \R$, we introduce the weighted local Lebesgue spaces $L^q_{l,loc}(\R^d)$ given by
\[
	L^q_{l,loc}(\R^d)
	:=\left\{ f \in L^0(\R^d) \,;\, f|_K\in L^{q}_l(\R^d) \text{ for all compact set } K\subset \R^d \right\}, 
\] 
where $L^0(\mathbb R^d)$ is the set of all Lebesgue measurable functions on $\mathbb R^d$. 
We give a precise definition of weak solutions to the Cauchy problem (\ref{HHNLH}).

\begin{defi}[Weak solution]
	\label{d:w.sol}
Let $T\in (0,\infty]$, $d\in \N$, $\gamma\in \R$, $a=1$ and $\alpha\ge 1$. 
We call a function $u:[0,T)\times \R^d\rightarrow \C$ a weak solution to \eqref{HHNLH} 
if $u$ belongs to $L^{\alpha}_t(0,T;L^{\alpha}_{\frac{\gamma}{\alpha},loc}(\R^d))$ 
and if it satisfies the equation in the distributional sense, i.e., 
\begin{align}\label{weak}
\notag&\int_{\R^d} u(T',x) \eta (T',x) \, dx-\int_{\R^d} u_0(x) \eta (0,x) \, dx\\
	&= \int_{[0,T']\times\R^d} u(t ,x)(\partial_t\eta+\Delta \eta) (t ,x) 
	+ |x|^{\gamma} |u(t, x)|^{\alpha-1} u(t,x) \,\eta(t,x)  \, dx\,dt
\end{align}
for all $T'\in [0,T)$ and for all $\eta \in C^{1,2}([0,T)\times \R^d)$ such that 
$\operatorname{supp} \eta(t, \cdot)$ with $t\in [0,T)$ is compact in $\R^d$. 
\end{defi}

The following theorem asserts the nonexistence of local positive weak solutions to \eqref{HHNLH} 
in the supercritical case \eqref{d:supercritical} for some data in $L^{q,1}_{l}(\R^d)$.
This theorem is an extension of \cite[Theorem 1.16]{CITT2022}, 
since it treats initial data belonging to $L^{q,1}_l(\R^d)$, 
while \cite[Theorem 1.16]{CITT2022} treats initial data belonging to a larger space $L^{q}_l(\R^d)$.

\begin{thm}[Nonexistence of local positive weak solutions]
\label{t:nonex}
Let $d\in \mathbb N$, $\gamma \in \mathbb R$, $a=1$ and $\alpha>1$ satisfy $\alpha>\alpha_F(\gamma)$. 
Let $q\in [1,\infty]$ and $l\in\R$ satisfy \eqref{d:supercritical}. Then there exists an initial data $u_0 \in {L^{q,1}_l} (\R^d)$ such that the 
problem \eqref{HHNLH} with $u(0)=u_0$ has no local positive weak solution. 
\end{thm}

\begin{proof}
The proof is based on method of \cite[Proposition 2.4, Theorem 2.5]{II-15}. 
Let $T\in (0,1)$. On the contrary we assume that the conclusion of Theorem \ref{t:nonex} does not hold. 
Then there exists a positive weak solution $u$ to \eqref{HHNLH} (See Definition \ref{d:w.sol}). 
We introduce $\eta \in C^\infty_0([0,\infty))$ and $\phi\in C^\infty_0(\mathbb R^d),$ satisfying 
$0\leq \eta, \, \phi,$ 
\[
	\eta (t)
	:= 
	\begin{cases}
		1,\quad 0\le t \le \frac12,\\
		0,\quad t\ge1,
	\end{cases}
		\quad\text{and}\quad 
	\phi(x)
	:= 
	\begin{cases}
		1,\quad |x| \le \frac12,\\
		0,\quad |x|\ge1.
	\end{cases}
\]
Let $b\in\mathbb N$ satisfy $b>\frac{2\alpha}{\alpha-1}$. We also introduce $\psi_T\in C_0^{\infty}([0,T)\times\R^d)$ given by
\[
	\psi_T(t,x) 
	:= \eta^b\left(\frac{t}{T}\right) \phi^b\left(\frac{x}{\sqrt{T}}\right)
	=\left\{\eta\left(\frac{t}{T}\right)\right\}^b\left\{\phi\left(\frac{x}{\sqrt{T}}\right)\right\}^b.
\]
By a direct computation, there exists $C>0$ independent of $T$ such that for any $(t,x)\in [0,T)\times\R^d$ the estimate
\[
	|\partial_t \psi_T (t,x)|+|\Delta\psi_T(t,x)|
		\le CT^{-1} \psi_T(t,x)^\frac{1}{\alpha}
\] holds. We define a non-negative function $I:[0,T)\rightarrow \R_{\ge 0}$ given by
\[
	I(\tau):=\int_{[0,\tau)\times \{|x|<\sqrt{\tau}\}}|x|^{\gamma} u(t,x)^{\alpha} \, \psi_{\tau}(t,x) \, dtdx,
	\ \ \ \text{for}\ \tau\in (0,T).
\]
We note that $I(T)<\infty$, since $u\in L_t^{\alpha}(0,T;L^{\alpha}_{\frac{\gamma}{\alpha},loc}(\R^d))$.
By using the weak form \eqref{weak}, the above estimate, H\"older's inequality and Young's inequality, the estimates hold:
\[
\begin{aligned}
	I(\tau) + \int_{|x|<\sqrt{\tau}} u_0(x)  \phi^b\left(\frac{x}{\sqrt{\tau}}\right)\, dx
	&= \left|\int_{[0,\tau)\times \{|x|<\sqrt{\tau}\}}u(\partial_t \psi_{\tau} + \Delta \psi_{\tau})\,dt\,dx \right|\\
	& \le C\tau^{-1}\int_{[0,\tau)\times \{|x|<\sqrt{\tau}\}}|u| \psi_{\tau}^{\frac{1}{\alpha}} \,dtdx\\
	&\le CI(\tau)^{\frac{1}{\alpha}}K(\tau)^{\frac{1}{\alpha'}}
	\le \frac{1}{2}I(\tau)+CK(\tau),
\end{aligned}
\]
where $\alpha'$ is the H\"older conjugate of $\alpha$, i.e., $1= \frac{1}{\alpha} + \frac{1}{\alpha'}$ and 
$K:[0,T)\rightarrow \R_{>0}$ is defined by
\begin{align*}
	K(\tau):&=\tau^{-\alpha'}\int_{[0,\tau)\times\{|x|<\sqrt{\tau}\}}|x|^{-\frac{\gamma\alpha'}{\alpha}}dxdt\\
	&=|S_{d-1}|\tau^{1-\alpha'} \int_0^{\sqrt{\tau}}r^{-\frac{\gamma\alpha'}{\alpha}+d-1} dr
	=C \tau^{1-\alpha'-\frac{\gamma}{2(\alpha-1)}+\frac{d}{2}},
\end{align*}
where to obtain the last identity we have used the assumption $\alpha>1+\frac{\gamma}{d}$ and $C$ is a positive constant.
Summarizing the estimates obtained now, we have
\begin{equation}\label{ineq1}
\begin{aligned}
	\int_{|x|<\sqrt{\tau}} u_0(x) \phi^b\left(\frac{x}{\sqrt{\tau}}\right)\, dx
	\le I(\tau) + 2\int_{|x|<\sqrt{\tau}} u_0(x) \phi^b \left(\frac{x}{\sqrt{\tau}}\right)\, dx
	\le C \tau^{- \frac{2+\gamma}{2(\alpha-1)} + \frac{d}{2}}.
\end{aligned}
\end{equation}
Here we choose a initial data $u_0$ as
\[
u_0(x) := 
\begin{cases}
	|x|^{-l-\frac{d}{q}}(\log(1+|x|))^{\kappa} \quad & |x|\le 1,\\
	0 & \text{otherwise},
\end{cases}
\]
where $\kappa$ is taken as
\[
  \max\left\{0,l+\frac{d}{q}-d\right\}<\kappa<d\left(\frac{l}{d}+\frac{1}{q}-\frac{1}{q_c}\right).
\]
The existence of such $\kappa$ follows from the assumptions $\frac{l}{d}+\frac{1}{q}>\frac{1}{q_c}$ and $\alpha>\alpha_{F}(d,\gamma)$. Since $\kappa>0$, we can see that $u_0 \in L^{q,1}_l (\R^d)$. Since $0<\tau<T<1$, $u_0\ge 0$, $\log(1+r)\ge (\log2)r$ on $(0,1)$ and $\kappa>l+\frac{d}{q}-d$, we have
\begin{equation}\label{ineq2}
\begin{aligned}
	\int_{|x|<\sqrt{\tau}} u_0(x) \phi^b \left(\frac{x}{\sqrt{\tau}}\right)\, dx
	&\ge \int_{|x|<\frac{1}{2}\sqrt{\tau}}u_0(x)dx\\
	&=|S_{d-1}|\int_0^{\frac{1}{2}\sqrt{\tau}}r^{-l-\frac{d}{q}+d-1}(\log(1+r))^{\kappa} dr\\
	 &\ge C \tau^{\frac{d}{2}-\frac{1}{2}\left(l+\frac{d}{q}\right)+\frac{\kappa}{2}}.
\end{aligned}
\end{equation}
By combining \eqref{ineq1} and \eqref{ineq2} and $\kappa<d(\frac{l}{d}+\frac{1}{q}-\frac{1}{q_c})$, we obtain
\begin{equation}\label{contradiction}
	0< C \le \tau^{\frac{d}{2}\left(\frac{l}{d}+\frac{1}{q}-\frac{1}{q_c}\right)-\frac{\kappa}{2}} \to 0 
	\quad \text{as }\tau\to 0,
\end{equation}
which leads to a contradiction. The proof is complete.
\end{proof}

\section*{Acknowledgement}
\par
The first author is supported by Grant-in-Aid for Young Scientists (No. 21K13821), Japan Society for the Promotion of Science. 
The second author is supported by JST CREST (No. JPMJCR1913), Japan 
and the Grant-in-Aid for Scientific Research (B) (No.18H01132) and Young Scientists Research (No.19K14581), JSPS. 
The third author is supported by Grant for Basic Science Research Projects from The Sumitomo Foundation (No.210788). 
The fourth author is supported by Tamkeen under the NYU Abu Dhabi Research Institute grant of the center SITE CG002.


\end{document}